\documentclass[11pt,reqno]{amsart}

\usepackage{amsmath,amssymb,mathrsfs}
\usepackage{graphicx,cite, times}
\usepackage{cases}
 \usepackage{dsfont}
\newtheorem{theo}{Theorem}[section]
\newtheorem{coro}[theo]{Corollary}
\newtheorem{lemm}[theo]{Lemma}
\newtheorem{prop}[theo]{Proposition}
\newtheorem{rema}[theo]{Remark}
\newtheorem{defi}[theo]{Definition}

\numberwithin{equation}{section}

\begin{document}

\title[Quasi-periodic solutions to nonlinear beam equation]{
Quasi-periodic solutions to nonlinear beam equation on Compact Lie Groups with a multiplicative potential}

\author{Bochao Chen}
\address{School of
Mathematics and Statistics, Center for Mathematics and
Interdisciplinary Sciences, Northeast Normal University, Changchun, Jilin 130024, P.R.China}
\email{chenbc758@nenu.edu.cn}

\author{Yixian Gao}
\address{School of
Mathematics and Statistics, Center for Mathematics and
Interdisciplinary Sciences, Northeast Normal University, Changchun, Jilin 130024, P.R.China}
\email{gaoyx643@nenu.edu.cn}

\author{Shan Jiang}
\address{School of
Mathematics and Statistics, Center for Mathematics and
Interdisciplinary Sciences, Northeast Normal University, Changchun, Jilin 130024, P.R.China}
\email{jiangs973@nenu.edu.cn}

\author{Yong Li}
\address{School of
Mathematics and Statistics, Center for Mathematics and
Interdisciplinary Sciences, Northeast Normal University, Changchun, Jilin 130024, P.R.China}
\email{yongli@nenu.edu.cn}

\thanks{The research  of YG was  supported in part by
 NSFC grant 11671071,  NSFJL grant: 20160520094JH, 20170101044JC and EDJL grant JJKH20170904KJ. The research of YL was supported in part by NSFC grant: 11571065, 11171132 and National Research Program of China Grant
2013CB834100}


\begin{abstract}
The goal of this work is to study the  existence  of quasi-periodic solutions in time to  nonlinear
beam equations with a multiplicative potential. The  nonlinearities are required  to only  finitely differentiable and  the frequency is along a pre-assigned direction.
The result holds on any compact Lie group or  homogenous manifold with respect to a compact Lie group, which includes the standard torus $\mathbf{T}^{d}$, the special orthogonal group $SO(d)$, the special unitary group $SU(d)$, the spheres $S^d$ and  the real and complex Grassmannians.
  The proof is based on
 a  differentiable  Nash-Moser iteration scheme.
\end{abstract}

\maketitle

\section{Introduction}
This paper concerns    the existence of quasi-periodic solutions of the forced nonlinear beam equation
\begin{align}\label{E1.1}
u_{tt}+\Delta^2u+V( \boldsymbol x)u=\epsilon f(\omega t, \boldsymbol x,u),\quad  \boldsymbol x\in \boldsymbol{M},
\end{align}
where $\bf{\boldsymbol M}$ is any simply connected compact Lie group with dimension $d$ and rank $r$, $\epsilon>0$,
the frequency $\omega\in\mathbf{R}^{\nu}$, $V\in C^q(\boldsymbol M;\mathbf{R})$ and $f\in C^q(\mathbf{T}^\nu\times \boldsymbol M\times \mathbf{R};\mathbf{R})$, where $q$ is large enough. Assume that the frequency vector $\omega$ satisfies
\begin{align}\label{E1.8}
\omega=\lambda{\omega}_0,\quad\lambda\in\Lambda:=[1/2,3/2], \quad|{\omega}_{0}|\leq1,
\end{align}
where $|\cdot|$ will be  defined later  in \eqref{E1.6}. For some $\gamma_0>0$, the following Diophantian condition holds:
\begin{align}\label{E1.96}
|{\omega}_0\cdot l|\geq 2\gamma_0|l|^{-\nu},\quad\forall l\in\mathbf{Z}^{\nu}\backslash\{0\}.
\end{align}
Moreover we suppose
\begin{align}\label{E1.84}
\Delta^2+V(\boldsymbol x)\geq \kappa_0\mathrm{I}\quad\text{with}\quad\kappa_0>0.
\end{align}

Equation \eqref{E1.1} is interesting by itself. It is  derived from the following Euler-Bernoulli beam equation
\begin{equation*}
\frac{\mathrm{d}^{2}}{\mathrm{d}x^{2}}\left(EI\frac{\mathrm{d}^{2}u}{\mathrm{d}x^{2}}\right)=\mathfrak{g},
\end{equation*}
which describes the relationship between  the applied load and the beam's deflection, where the curve $u(x)$ describes the deflection of the beam at some position $x$ in the $z$ direction, $\mathfrak{g}$ is distributed load which may be a function of $x$, $u$ or other variables, $I$ is the second moment of area of the beam's cross-section, $E$ is the elastic modulus, the product $EI$  is the flexural rigidity. Derivatives of the deflection $u$ have significant physical significance: $u_x$ is the slope of the beam, $-EI u_{xx}$ is the bending moment of the beam and $-(EIu_{xx})_{x}$ is the shear force of the beam. The dynamic beam equation is the Euler-Lagrange equation
\begin{equation}\label{E4.2}
\mathfrak{m}\frac{\partial^{2}u}{\partial t^{2}}+\frac{\partial^{2}}{\partial x^{2}}\left(EI\frac{\partial^{2}u}{\partial x^{2}}\right)=\mathfrak{g},
\end{equation}
where $\mathfrak{m}$ is the mass per unit length. If $E$ and $I$ are independent of $x$, then equation \eqref{E4.2} can be reduced to
\begin{align*}
\mathfrak{m}u_{tt}+EIu_{xxxx}=\mathfrak{g}.
\end{align*}
After a time rescaling $t\rightarrow ct$ with $c=\sqrt{\frac{EI}{\mathfrak{m}}}$, we obtain
\begin{align*}
u_{tt}+u_{xxxx}=\tilde{\mathfrak{g}}\quad \text{with}\quad\tilde{\mathfrak{g}} = \frac{\mathfrak{m}}{EI}\mathfrak{g}.
\end{align*}

The search for periodic or quasi-periodic solutions  to nonlinear PDEs has a long standing tradition.
 There  are  two main approaches: one is the infinite-dimensional KAM (Kolmogorov-Arnold-Moser)
theory to  Hamiltonian PDEs, refer to
 Kuksin \cite{Kuksin1987}, Wayne \cite{Wayne1990},
P\"{o}schel \cite{poschel1996quasi}. The main difficulty, namely the presence of
arbitrarily “small divisors” in the expansion series of the solutions, is handled via KAM
theory. Later, another more direct bifurcation approach was
 established by  Craig
 and Wayne \cite{craig1993newton} and improved by Bourgain \cite{bourgain1998quasi,bourgain1995construction,bourgain1994construction}  based on Lyapunov-Schmidt
procedure,
 to solved
the small divisors problem, for periodic solutions, with an analytic Newton iterative  scheme.
This approach  is often called as the Craig-Wayne-Bourgain method, which is different from KAM. The main
advantage for this approach is to require only the so called “first order Melnikov” non-resonance conditions for solving the linearized equations  at each step of the iteration.

Up to now,  the existence of quasi-periodic solutions of the
 nonlinear beam equations using KAM theory and Nash-Moser iteration  have received much attention by the mathematical communities. For the $1$ D beam equation,  when $f(\omega t, x, u) = O(u^3)$ is an analytic, odd function, in \cite{geng2003kam} they  proved the existence of   linearly stable
small-amplitude quasi-periodic solution  by an infinite KAM theorem \cite{Poschel1996kam}.  In \cite{Gao2015quasi}, Chang, Gao, and
Li got quasi-periodic solutions for the 1-dimensional nonlinear beam equation with prescribed
frequencies under Dirichlet boundary condition. Later, Wang and
Si in \cite{Wang2012result}
 considered 1-dimensional beam equation
with quasi-periodically forced perturbations $f(\omega t,  x, u)=\epsilon \phi(t) h(u)$.  Also,  in \cite{Liang2006Quasi} they studied  the  existence of
 quasi-periodic solutions of the 1-dimensional completely resonant nonlinear beam equation.
  Some KAM-theorems for small-amplitude solutions of
equations \eqref{E1.1} on $\mathbf{T}^{d}$  with typical $V(\boldsymbol x) =m$   were obtained
in \cite{gengyou2006kam,geng2006kam}.   Both works treat
equations with a constant-coefficient analytical nonlinearity $f(\omega t, \boldsymbol x, u)=f (u).$
 Subsequently, by a KAM type theorem,
 Mi and Cong in \cite{Lufang2015Quasi} proved the  existence of  quasi-periodic solutions
of the nonlinear beam equation
\begin{equation*}
u_{tt}+\Delta^{2}u+V(\boldsymbol x)\ast
u+\epsilon\Big(-\sum_{i,j=1}^{d}b_{ij}(u,\nabla u)\partial_{i}\partial_{j}u+g(u,\nabla u)\Big)=0, \quad \boldsymbol  x\in \mathbf{T}^{d}
\end{equation*}
and got the linear stability for corresponding solutions.
In perturbations term, $\nabla u\equiv(\partial_{x_1}u,\partial_{x_2}u,\cdots,\partial_{x_d}u)$ the derivatives of $u$ with respect to the space variables, and $b_{ij},g$ are real analytic function.
 In \cite{Eliasson2016beam}, Eliasson, Gr\'{e}bert, and Kuksin
proved that for  the nonlinear beam equation
\begin{equation*}
u_{tt}+\Delta^{2}u+mu+\partial_{u}f(\boldsymbol x,u)=0,\quad \boldsymbol  x\in\mathbf{T}^{d},
\end{equation*}
where $f(\boldsymbol x,u)=u^4+O(u^5)$, has many linearly stable or unstable small-amplitude quasi-periodic solutions.
 All above proofs depend on
some type of KAM techniques and are carried out in analytic nonlinearities cases. However, there is
no existence result for nonlinear beam equation with perturbations having only finitely differentiable
regularities. Recently, in \cite{Shi2016On}, based on a Nash-Moser type implicit function theorem, Shi prove
the existence of quasi-periodic solution of the following beam equation
\begin{align}
u_{tt}+\Delta^{2}u+mu=\epsilon f(\omega t, \boldsymbol x, u),\quad \boldsymbol  x\in\mathbf{T}^{d},
\end{align}
where $f$ is finitely differentiable, $\omega $ is the frequency, $m>0$. Note that these previous results are confined to tori on existence of quasi-periodic solutions of the nonlinear beam equation.
The reason why these results are confined to tori is that their proofs require specific
properties of the eigenvalues, while the eigenfunctions must be the exponentials or, at least,
strongly ``localised close to exponentials''.
 According to the harmonic analysis on compact Lie groups and the theory of
the highest weight which provides an accurate description of the eigenvalues of the Lapalce-Beltrami
operator as well as the multiplication rules of its eigenfunctions.
In \cite{Berti2011Duke} Berti and Procesi proved the  existence
of time-periodic solutions for NLW and NLS on $\boldsymbol M$, where $\boldsymbol M$ is any compact Lie group or  homogenous manifold
with respect to a compact Lie group. Later, in \cite{berti2015abstract}, Berti, Corsi and Procesi extended this result
to the case of quasi-periodic solutions in time of the following NLW and NLS:
\begin{align*}
u_{tt}-\Delta u+mu=\epsilon f(\omega t, \boldsymbol x,u),\quad\mathrm{i}u_{t}-\Delta u+mu=\epsilon f(\omega t, \boldsymbol x,u),\quad \boldsymbol x\in \boldsymbol M.
\end{align*}
Moreover, in \cite{Berti2012nonlinearity,Berti2013Quasi}, Berti and Bolle considered  the NLW and NLS respectively, with a multiplicative potential:
\begin{align*}
u_{tt}-\Delta u+V(\boldsymbol x)u=\epsilon f(\omega t, \boldsymbol x,u),\quad\mathrm{i}u_{t}-\Delta u+V(\boldsymbol x)u=\epsilon f(\omega t,\boldsymbol x,u),\quad \boldsymbol x\in \mathbf{T}^d.
\end{align*}
In the present  paper,
our goal is to  prove the existence of quasi-periodic solutions in time of the nonlinear beam equation \eqref{E1.1} with a multiplicative potential $V( \boldsymbol x)$  and finite regularity nonlinearities
 on any compact Lie group or  homogenous manifold with respect to a compact Lie group.
 There are three main difficulties in this work:  (i)  a multiplicative potential in  higher dimensions. The eigenvalues of the operator  $\Delta^2 + V (\boldsymbol x)$ appear in clusters of unbounded sizes and the the
 eigenfunctions  are (in general) not localized with respect to the exponentials. We will use the similar  properties of the eigenvalues and the eigenfunctions of the opearator $-\Delta + V (\boldsymbol x)$  in \cite{berti2015abstract}.
 (ii) the finite differentiable regularities of the nonlinearity. Clearly,  a difficulty when working with functions having only
Sobolev regularity is that the Green functions will exhibit only a polynomial decay off
the diagonal, and not exponential (or subexponential). A key concept one must exploit is
the interpolation/tame estimates. (iii) the nonlinear beam equation  are defined not only
on tori, but on any compact Lie group or  homogenous manifold with respect to a compact Lie group, which includes the standard torus $\mathbf{T}^{d}$, the special orthogonal group $SO(d)$, the special unitary group $SU(d)$, the spheres $S^d$, the real and complex Grassmannians, and so on, recall \cite{Brocker1995representation}.


The rest of the paper is organized  as follows: we state the main result (see Theorem \ref{theo2}) and introduce several notations in subsection \ref{sec:2.1}. In subsection \ref{sec:2.2}, we define the strong $s$-norm of a matrix $M$ and introduce its properties. Section \ref{sec:3} is devoted to give the iterative theorem, see Theorem \ref{theo1}. In subsection \ref{sec:3.1}, we give a multiscale analysis of the linearized operators $\mathfrak{L}_{N}(\epsilon,\lambda,u)$ (recall \eqref{E3.25}) as \cite{Berti2012nonlinearity}, see Proposition \ref{pro1}.  Our aim is to check that the assumption $(\mathrm{A3})$ in Proposition \ref{pro1} holds in subsection \ref{sec:3.2}. Under that Proposition \ref{pro1} and \ref{pro3} hold, we have to remove some $\lambda$ in $\Lambda$, recall \eqref{E1.8}. In subsection \ref{sec:3.3}, the measure of the excluded $\lambda$ satisfies \eqref{E3.10} and \eqref{E3.77} respectively. In subsection \ref{sec:3.4}, we establish Theorem \ref{theo1} and give the proof. At the end of the construction, we prove that the measure of the parameter $\lambda$ satisfying Theorem \ref{theo2} is a large measure Cantor-like set in subsection \ref{sec:3.5}.  Finally, in section \ref{sec:4}, we list the the proof of some related results for the sake of completeness.

\section{Main results}
\subsection{Notations}\label{sec:2.1}
After a time rescaling $\varphi=\omega\cdot t$,  we consider the existence of  solutions $u(\varphi,\boldsymbol x)$ of
\begin{align}\label{E1.2}
(\lambda\omega_0\cdot \partial_{\varphi})^2u+\Delta^2 u+V(\boldsymbol x)u=\epsilon f(\varphi, \boldsymbol x,u),\quad \boldsymbol x\in \boldsymbol M.
\end{align}
Define an index set $\mathfrak{N}$ as
\begin{align*}
\mathfrak{N}:=\mathbf{Z}^{\nu}\times \Gamma_{+}(\boldsymbol M)\quad\text{with}\quad\Gamma_{+}(\boldsymbol M):=\left\{{j}\in\mathbf{R}^r:~{j}=\sum\limits_{k=1}^{r}j_k\mathrm{\textbf{w}}_k,~ j_k\in\mathbf{N}\right\},
\end{align*}
where $\Gamma_{+}(\boldsymbol M)$ is contained in an $r$-dimensional lattice (in general not orthogonal)
\begin{align*}
\Gamma:=\left\{{j}\in\mathbf{R}^r:\quad j=\sum\limits_{k=1}^{r}j_k\mathrm{\textbf{w}}_k,\quad j_k\in\mathbf{Z}\right\}
\end{align*}
generated by independent vectors $\mathrm{\textbf{w}}_1,\cdots,\mathrm{\textbf{w}}_r\in\mathbf{R}^r$. There exists an integer $\mathfrak{z}\in\mathbf{N}$ such that the fundamental weights satisfy
\begin{align}\label{E1.56}
\textbf{w}_k\cdot\textbf{w}_{k'}\in \mathfrak{z}^{-1}\mathbf{Z}, \quad\forall k,k'=1,\cdots,r.
\end{align}
Moreover $\Gamma_{+}(\boldsymbol M)$ is required to satisfy a product structure, namely
\begin{align}\label{E1.95}
j&=\sum\limits_{k=1}^{r}{j_{k}\mathrm{\textbf{w}}_{k}},\quad j'=\sum\limits_{k=1}^{r}{j'_{k}\mathrm{\textbf{w}}_{k}}\nonumber\\
&\Rightarrow j''=\sum\limits_{k=1}^{r}{j''_{k}\mathrm{\textbf{w}}_{k}}\in\Gamma_{+}(\boldsymbol M) \quad\text{if}\quad\min\left\{j_k,j'_k\right\}\leq j''_k\leq\max{\{j_k, j'_k\}},\quad\forall~k=1,\cdots,r.
\end{align}
Remark that \eqref{E1.95} is used only in the proof of Lemma \ref{lemma5}.

We briefly recall the relevant properties of harmonic analysis on compact Lie group, see \cite{Berti2011Duke}. The eigenvalues of the Laplace-Beltrami operator $\Delta$ on $\boldsymbol M$ are
\begin{align*}
\lambda_j:=-\|j+\rho\|^2+\|\rho\|^2
\end{align*}
with respect to the the eigenfunctions
\[\textbf{e}_{j,p}(\boldsymbol x),\quad \boldsymbol x\in\boldsymbol M,\quad j\in\Gamma_{+}(\boldsymbol M),\quad p=1,\cdots,\mathfrak{d}_j,\] where
 $\|\cdot\|$ stands for the Euclidean norm on $\mathbf{R}^r$, $\rho:=\sum_{k=1}^{r}\mathrm{\textbf{w}}_k$,
$\textbf{e}_{j}(\boldsymbol x)$ is
 the (unitary) matrix associated to an irreducible unitary representations $(\mathrm{R}_{\mathrm{V}_{j}},\mathrm{V}_{j})$ of
$\boldsymbol M$, namely
\begin{align*}
(\textbf{e}_{j}(\boldsymbol x))_{p,p'}=\langle\mathrm{R}_{\mathrm{V}_{j}}(\boldsymbol x)\mathrm{v}_{p},\mathrm{v}_{p'}\rangle,~~~\mathrm{v}_{p},\mathrm{v}_{p'}\in
\mathrm{V}_{j},
\end{align*}
where $(\mathrm{v}_{p})_{p=1,\ldots,\mathrm{dim}\mathrm{V}_{j}}$ is an orthonormal basis of the finite dimensional euclidean space $\mathrm{V}_{j}$ with scalar product $\langle\cdot,\cdot\rangle$. Denote by $\mathcal{N}_{j}$ the eigenspace of $\Delta$ with respect to $\lambda_j$. The degeneracy of the eigenvalue $\lambda_j$ satisfies
\begin{align*}
\mathfrak{d}_j\leq \|j+\rho\|^{d-r}.
\end{align*}
Furthermore, by the Peter-Weyl theorem, we have the following
 orthogonal decomposition
\begin{align*}
    L^{2}(\boldsymbol M)=\bigoplus_{j\in\Gamma_{+}(\boldsymbol M)}\mathcal{N}_{j}.
\end{align*}
Given $\mathfrak{n}={(l,j)}\in\mathfrak{N}$ and $\mathfrak{U},\mathfrak{U}_{1},\mathfrak{U}_{2}\subset \mathfrak{E}\subset\mathfrak{N}$, define
\begin{align}\label{E1.6}
&|\mathfrak{n}|:=\max{\{|l|,|j|\}},\quad|l|:=\max_{1\leq k\leq \nu}|l_k|,\quad|j|:=\max_{1\leq k\leq r}|j_k|;\\
\mathrm{diam}(\mathfrak{E})&:=\sup_{\mathfrak{n},\mathfrak{n}'\in \mathfrak{E}}|\mathfrak{n}-\mathfrak{n}'|,
\quad \mathrm{d}(\mathfrak{U}_{1},\mathfrak{U}_{2}):=\inf_{\mathfrak{n}\in\mathfrak{U}_{1},\mathfrak{n}'\in\mathfrak{U}_{2}}|\mathfrak{n}-\mathfrak{n}'|,
\quad \mathrm{d}(\mathfrak{n},\mathfrak{U}):=\inf_{\mathfrak{n}'\in\mathfrak{U}}|\mathfrak{n}-\mathfrak{n}'|.\nonumber
\end{align}
\begin{rema}
We set $\mathfrak{n}-\mathfrak{n}'=0$ if $\mathfrak{n}-\mathfrak{n}'\in\mathbf{Z}^{\nu}\times(\Gamma\setminus \Gamma_{+}(\boldsymbol M))$.
\end{rema}
For some constants $c_2>c_1>0$, the following  holds:
\begin{align}\label{E1.4}
c_1|\mathfrak{n}|\leq \sqrt{\|l\|^2+\|j+\rho\|^2}\leq c_2|\mathfrak{n}|, \quad\forall \mathfrak{n}=(l,j)\in\mathfrak{N}.
\end{align}
Decomposing
\begin{align*}
u(\varphi,\boldsymbol x)=\sum\limits_{\mathfrak{n}\in\mathfrak{N}}\textbf{u}_\mathfrak{n}e^{\mathrm{i}l\cdot\varphi}
\textbf{e}_{j}(\boldsymbol x)=\sum\limits_{(l,j)\in\mathfrak{N}}e^{\mathrm{i}l\cdot\varphi}\sum\limits_{p=1}^{\mathfrak{d}_j}u_{l,j,p}\textbf{e}_{j,p}(\boldsymbol x),
\end{align*}
the Sobolev space $H^s$ is defined by
\begin{align}\label{E1.3}
H^s:=H^s(\mathfrak{N};\mathbf{R}):=\left\{u=\sum\limits_{\mathfrak{n}\in\mathfrak{N}}\textbf{u}_\mathfrak{n}e^{\mathrm{i}l\cdot\varphi}\textbf{e}_{j}(\boldsymbol x):\textbf{u}_\mathfrak{n}\in\mathbf{C}^{\mathfrak{d}_j},~\|u\|^2_s=\sum\limits_{\mathfrak{n}\in\mathfrak{N}}\langle w_{\mathfrak{n}}\rangle^{2s}\|\textbf{u}_{\mathfrak{n}}\|^2_0<+\infty\right\}
\end{align}
with $\langle w_\mathfrak{n}\rangle:=\max\{c_1,1,~(\|l\|^2+\|j+\rho\|^2)^{1/2}\}$, where $c_1$ is seen in \eqref{E1.4}, $\|\textbf{u}_\mathfrak{n}\|^2_0:=2\pi\sum_{p=1}^{\mathfrak{d}_j}|u_{l,j,p}|^2$.
There also exist ${b}_2>{b}_1>0$ such that
\begin{equation}\label{E1.50}
{b}_1|j|\leq\|j\|\leq {b}_2|j|,\quad\forall j\in\Gamma_{+}(\boldsymbol M).
\end{equation}

For $s\geq s_0>(\nu+d)/2$, the Sobolev space ${H}^s$ has the following properties:
\begin{align*}
(\mathrm{1})\quad&\|uv\|_{{s}}\leq C(s)\|u\|_{s}\|v\|_{s}, \quad\forall u,v\in{H}^s;\\
(\mathrm{2})\quad&\|u\|_{L^{\infty}}\leq C(s)\|u\|_{s},\quad\forall u\in H^s;\\
(\mathrm{3})\quad&\|uv\|_s\leq C(s)(\|u\|_s\|v\|_{s_0}+\|u\|_{s_0}\|v\|_s),\quad\forall u,v\in{H}^s.
\end{align*}
The above properties $(\mathrm{1}),(\mathrm{2})\text{ and }(\mathrm{3})$ are also seen in {\cite[Lemma 2.13]{Berti2011Duke}}.

Let $V(\boldsymbol x)=m+\bar{V}(\boldsymbol x)$, where $m$ is the average of $V(\boldsymbol x)$ and $\bar{V}$ has zero average. Define the composition operator on Sobolev spaces
\begin{align}\label{E1.102}
F:\quad H^s&\rightarrow H^s,\quad u\mapsto f(\varphi,\boldsymbol x,u),
\end{align}
where  $f\in C^q(\mathbf{T}^{\nu}\times\boldsymbol M\times\mathbf{R};\mathbf{R})$. The core of a Nash-Moser iteration is the invertibility of the following linearized operator
\begin{align}\label{E1.73}
{\mathfrak{L}}(\epsilon,\lambda,u):=L_{\lambda}-\epsilon(\mathrm{D}F)(u)=&L_\lambda-\epsilon (\partial_{u}f)(\varphi, \boldsymbol x,u)
=D_\lambda+\bar{V}(\boldsymbol x)-\epsilon (\partial_{u}f)(\varphi, \boldsymbol x,u),
\end{align}
where
\begin{align}\label{E1.75}
&L_\lambda:=(\lambda{\omega}_0\cdot\partial_{\varphi})^2+\Delta^2+V(\boldsymbol x),\quad D_\lambda:=(\lambda{\omega}_0\cdot\partial_{\varphi})^2+\Delta^2+m.
\end{align}
In the Fourier basis $e^{\mathrm{i}l\cdot\varphi}\textbf{e}_{j}(\boldsymbol x)$, the operator ${\mathfrak{L}}(\epsilon,\lambda,u)$ (see \eqref{E1.73}) is represented by the infinite-dimensional self-adjoint matrix
\begin{align}\label{E1.72}
\mathcal{A}(\epsilon,\lambda,u):=\mathcal{D}(\lambda)+\mathcal{T}(\epsilon,u)=\mathcal{D}(\lambda)+\mathcal{T}'-\epsilon\mathcal{T}''(u),
\end{align}
where $\mathcal{D}(\lambda):=\mathrm{diag}_{\mathfrak{n}\in\mathfrak{N}}(\mu_{\mathfrak{n}}(\lambda)\mathrm{I}_{\mathfrak{d}_j})$, with $\mu_{\mathfrak{n}}(\lambda)=-{(\lambda{\omega}_0\cdot l)}^2+\lambda^2_j+m$, and
\begin{align}
{(\mathcal{T}')}^{\mathfrak{n}'}_{\mathfrak{n}}:=(\bar{V})_{j-j'},\quad {(\mathcal{T}'')}^{\mathfrak{n}'}_{\mathfrak{n}}:=(a)_{\mathfrak{n}-\mathfrak{n}'}=a_{l-l'}(j-j')\label{E1.100}
\end{align}
with $a(\varphi, \boldsymbol x):=(\partial_{u}f)(\varphi,\boldsymbol x,u(\varphi, \boldsymbol x))$. Similarly, we also define
\begin{align}\label{E1.101}
{\mathfrak{L}}(\epsilon,\lambda,u,\theta):=&L_{\lambda}(\theta)-\epsilon(\mathrm{D}F)(u)=L_\lambda(\theta)-\epsilon (\partial_{u}f)(\varphi, \boldsymbol x,u)\nonumber\\
=&D_\lambda(\theta)+\bar{V}(\boldsymbol x)-\epsilon (\partial_{u}f)(\varphi, \boldsymbol x,u),
\end{align}
where $F$ is seen in \eqref{E1.102}, and
\begin{align}\label{E1.85}
&L_\lambda(\theta)=(\lambda{\omega}_0\cdot\partial_{\varphi}+\mathrm{i}\theta)^2+\Delta^2+V(\boldsymbol x),\quad D_\lambda(\theta)=(\lambda{\omega}_0\cdot\partial_{\varphi}+\mathrm{i}\theta)^2+\Delta^2+m.
\end{align}
For all $\theta\in\mathbf{R}$, the operator ${\mathfrak{L}}(\epsilon,\lambda,u,\theta)$ (see \eqref{E1.101}) is represented by the infinite-dimensional self-adjoint matrix depending on $\theta$
\begin{align}\label{E1.97}
\mathcal{A}(\epsilon,\lambda,u,\theta):=\mathcal{D}(\lambda,\theta)+\mathcal{T}(\epsilon,u)=\mathcal{D}(\lambda,\theta)+\mathcal{T}'
-\epsilon\mathcal{T}''(u),
\end{align}
where $\mathcal{D}(\lambda,\theta):=\mathrm{diag}_{\mathfrak{n}\in\mathfrak{N}}(\mu_{\mathfrak{n}}(\lambda,\theta)\mathrm{I}_{\mathfrak{d}_j})$, with $\mu_{\mathfrak{n}}(\lambda,\theta)=-{(\lambda{\omega}_0\cdot l+\theta)}^2+\lambda^2_j+m$, and $\mathcal{T}',\mathcal{T}''$ are given in \eqref{E1.100}. In addition denote by $\mathcal{A}_{N,l_0,j_0}(\epsilon,\lambda,u,\theta)$ the submatrices of $\mathcal{A}(\epsilon,\lambda,u,\theta)$ centered at ($l_0,j_0$), where
\begin{align}\label{E1.94}
\mathcal{A}_{N,l_0,j_0}(\epsilon,\lambda,u,\theta):=\mathcal{A}_{|l-l_0|\leq N,|j-j_0|\leq N}(\epsilon,\lambda,u,\theta).
\end{align}
We use the simpler notations
\begin{align}
&\mathcal{A}_{N,j_0}(\epsilon,\lambda,u,\theta):=\mathcal{A}_{N,0,j_0}(\epsilon,\lambda,u,\theta)\quad\text{if}\quad l_0=0;\label{E2.48}\\
&\mathcal{A}_{N}(\epsilon,\lambda,u,\theta):=\mathcal{A}_{N,0,0}(\epsilon,\lambda,u,\theta)\quad\text{if}\quad(l_0,j_0)=(0,0);\label{E2.49}\\
&\mathcal{A}_{N,j_0}(\epsilon,\lambda,u):=\mathcal{A}_{N,0,j_0}(\epsilon,\lambda,u,0)\quad\text{if}\quad l_0=0,\theta=0;\label{E2.50}\\
&\mathcal{A}_{N}(\epsilon,\lambda,u):=\mathcal{A}_{N,0,0}(\epsilon,\lambda,u,0)\quad\text{if}\quad l_0=0,j_0=0,\theta=0.\label{E2.51}
\end{align}
Clearly, the following crucial covariance property holds:
\begin{align}\label{E1.93}
\mathcal{A}_{N,l_0,j_0}(\epsilon,\lambda,u,\theta)=\mathcal{A}_{N,j_0}(\epsilon,\lambda,u,\theta+\lambda{\omega}_0\cdot l_0).
\end{align}

The main result of this paper is
\begin{theo}\label{theo2}
Let $\bf{\boldsymbol M}$ be any simply connected compact Lie group with dimension $d$ and rank $r$. Assume  \eqref{E1.96} holds,  then there exist $s:=s(\nu,d,r)$, $q:=q(\nu,d,r)\in\mathbf{N}$,$\epsilon_0>0$, a map
\begin{align*}
u(\epsilon,\cdot)\in C^1(\Lambda;H^s) \quad\text{with}\quad u(0,\lambda)=0,
\end{align*}
and a Cantor-like set $\mathscr{D}_{\epsilon}\subset\Lambda$ of asymptotically full Lebesgue measure, namely
\begin{align*}
meas(\mathscr{D}_{\epsilon})\rightarrow1 \quad\text{as}\quad \epsilon\rightarrow 0,
\end{align*}
such that, for all $V\in C^ q$ satisfies \eqref{E1.84}, $f\in C^q$ and $\lambda\in\mathscr{D}_{\epsilon}$, $u(\epsilon,\lambda)$ is a solution of \eqref{E1.2} with $\omega=\lambda\omega_0$.
\end{theo}
\begin{rema}
 If $V,f\in C^{\infty}$, then $u(\epsilon,\lambda)\in C^{\infty}(\mathbf{T}^{\nu}\times \boldsymbol M;\mathbf{R})$, which can be completed  just by  making small modifications in the proofs of lemmas \ref{lemma22} and \ref{lemma23}, and formulae \eqref{E3.83}-\eqref{E3.84}.
 \end{rema}
\begin{rema} In fact,  here
$\boldsymbol M$ may be a  homogeneous manifold   with respect to a compact Lie group, namely
\begin{align*}
\boldsymbol M:=G/G_0,\quad G:=\mathrm{G}\times\mathbf{T}^{r'},
\end{align*}
where $G_0$ is a closed subgroup of $G$, $\mathrm{G}$ is a simple connected compact Lie group, $\mathbf{T}^{r'}$ is a tori. The eigenvalues of the Laplace-Beltrami operator $\Delta$ on $\boldsymbol M$ are
\begin{align*}
\lambda'_{\vec{j}}:=-\|\vec{j}+\vec{\rho}\|^2+\|\vec{\rho}\|^2=-\|j^{(1)}+\rho\|^2+\|\rho\|^2-\|j^{(2)}\|^2
\end{align*}
with respect to the the eigenfunctions
\[\textbf{e}_{j^{(1)},p}(\boldsymbol x^{(1)})e^{\mathrm{i}j^{(2)}\cdot \boldsymbol x^{(2)}},\quad \vec{\boldsymbol x}=(\boldsymbol x^{(1)},\boldsymbol x^{(2)})\in\mathrm{G}\times\mathbf{T}^{r'},\quad p=1,\cdots,\mathfrak{d}'_j,\]
where $\vec{\rho}=(\rho,0)$, $\vec{j}=(j^{(1)},j^{(2)})$ belongs to a subset of $\Gamma_{+}(\mathrm{G})\times\mathbf{Z}^{r'}$, and $\mathfrak{d}'_{\vec{j}}\leq \mathfrak{d}_{j^{(1)}}$ (recall {\cite[Theorem 2.9]{Berti2011Duke}}).
\end{rema}
\subsection{Matrices with off-diagonal decay}\label{sec:2.2}
For $\mathfrak{B},\mathfrak{C}\subset \mathfrak{N}$, a bounded linear operator
 $\mathfrak{L}:H^s_{\mathfrak{B}}\rightarrow H^s_{\mathfrak{C}}$ is represented by a matrix in
\begin{equation*}
\mathcal {M}_{\mathfrak{C}}^{\mathfrak{B}}=\left\{(M_{\mathfrak{n}'}^{\mathfrak{n}''})_{\mathfrak{n}'\in \mathfrak{C},\mathfrak{n}''\in \mathfrak{B}}  \quad\text{with}\quad M_{\mathfrak{n}'}^{\mathfrak{n}''}\in \mathrm{Mat}(\mathfrak{d}_{j'}\times \mathfrak{d}_{j''},\mathbf{C})\right\},
\end{equation*}
where $H^{s}_{\mathfrak{B}}:=\left\{u=\sum\limits_{\mathfrak{n}\in\mathfrak{N}}\textbf{u}_\mathfrak{n}e^{\mathrm{i}l\cdot\varphi}\textbf{e}_{j}(\boldsymbol x)\in H^s: \textbf{u}_\mathfrak{n}=0~\text{if}~\mathfrak{n}\notin \mathfrak{B} \right\}$. Define the $L^{2}$-operator norm
\begin{equation*}
\|M_{\mathfrak{B}}^{\mathfrak{C}}\|_{0}=\sup_{h\in H_\mathfrak{B}}\frac{\|M_\mathfrak{C}^{\mathfrak{B}}h\|_{0}}{\|h\|_0}.
 \end{equation*}
Moreover we introduce the strong $s$-norm of a matrix $M\in\mathcal {M}_{\mathfrak{C}}^{\mathfrak{B}}$ as follows:
\begin{defi}
The $s$-norm of any matrix $M\in\mathcal {M}_{\mathfrak{C}}^{\mathfrak{B}}$ is defined by
 \begin{align}\label{E1.82}
|M|^2_{s}:=K_0\sum\limits_{\mathfrak{n}\in\mathfrak{N}}[M(\mathfrak{n})]^2\langle \mathfrak{n}\rangle^{2s}
\end{align}
where $K_{0}>4\sum_{\mathfrak{n}\in \mathrm{Z}^{\nu}\times\Gamma}\langle
\mathfrak{n}\rangle^{-2s_{0}}$, $\langle \mathfrak{n}\rangle=\max(1,|\mathfrak{n}|)$, and
\begin{equation*}
 [M(\mathfrak{n})]=
\begin{cases}
\sup_{\mathfrak{n}'-\mathfrak{n}''=\mathfrak{n},~\mathfrak{n}\in \mathfrak{C},\mathfrak{n}'\in \mathfrak{B}}\| M_{\mathfrak{n}'}^{\mathfrak{n}''}\|_{0}\quad &\text{if }\quad \mathfrak{n}\in \mathfrak{C}-\mathfrak{B},  \\
0 &\text{if }\quad \mathfrak{n}\notin \mathfrak{C}-\mathfrak{B}.
 \end{cases}
 \end{equation*}
 \end{defi}
It is obvious that the $s$-norm in \eqref{E1.82} satisfies that $|\cdot|_s\leq|\cdot|_{s'}$ for all $0<s\leq s'$. The following properties (see lemmas \ref{lemma11}-\ref{lemma2}) on the strong $s$-norm are given in \cite{berti2015abstract}.

\begin{lemm}\label{lemma11}(\cite[Corollary3.3]{berti2015abstract})
Let $\mathfrak{L}:u(\varphi,\boldsymbol x)\mapsto g(\varphi,\boldsymbol x)u(\varphi,\boldsymbol x)$ be a linear operator in $L^2$, self-adjoint. Then, $\forall s>(\nu+d)/2$, the following holds:
\begin{align}\label{E1.22}
|\mathfrak{L}|_s\leq C(s)\|g\|_{s+\varrho}\quad\text{with}\quad\varrho=(2\nu+d+r+1)/2.
\end{align}
\end{lemm}
\begin{lemm}\label{lemma7} ({\rm Interpolation} \cite[Lemma2.6]{berti2015abstract})
There exists $C(s)\geq1$, with $C(s_{0})=1$, such that for all subset $\mathfrak{B},\mathfrak{C},\mathfrak{D}\subset\mathfrak{N}$, one has:
\begin{align}\label{E1.23}
|M_1M_2|_s\leq \frac{1}{2}|M_1|_{s_0}|M_2|_s+\frac{C(s)}{2}|M_1|_s|M_2|_{s_0},\quad\forall s\geq s_0,\forall M_{1}\in\mathcal {M}_{\mathfrak{C}}^{\mathfrak{B}},M_{2}\in
 \mathcal {M}_{\mathfrak{B}}^{\mathfrak{D}},
\end{align}
in particular,
\begin{align}\label{E1.24}
|M_1M_2|_s\leq C(s)|M_1|_{s}|M_2|_s,\quad\forall s\geq s_0,\forall M_{1}\in\mathcal {M}_{\mathfrak{C}}^{\mathfrak{B}},M_{2}\in
 \mathcal {M}_{\mathfrak{B}}^{\mathfrak{D}}.
\end{align}
\end{lemm}
\begin{lemm}\label{lemma8}  (\cite[Lemma2.7]{berti2015abstract})
For all subset $\mathfrak{B},\mathfrak{C}\subset\mathfrak{N}$, we have
\begin{align}\label{E1.25}
\|Mh\|_s\leq C(s)(|M|_{s_0}\|h\|_s+|M|_s\|h\|_{s_0})\quad \forall M\in\mathcal {M}_{\mathfrak{C}}^{\mathfrak{B}},\forall h\in H^s_\mathfrak{B}.
\end{align}
\end{lemm}
\begin{lemm}\label{lemma9} ({\rm Smoothing}  \cite[Lemma2.8]{berti2015abstract})
For all subset $\mathfrak{B},\mathfrak{C}\subset\mathfrak{N}$ and all $s'\geq s\geq0$, one has that, for $N\geq 2$,
\\
$\mathrm{(1)}$ If $M^{\mathfrak{n}'}_{\mathfrak{n}}=0$ for all $|\mathfrak{n}'-\mathfrak{n}|\leq N$, then
\begin{align}\label{E1.61}
|M|_{s}\leq N^{-(s'-s)}|M|_{s'},\quad \forall M\in\mathcal{M}^{\mathfrak{B}}_{\mathfrak{C}}.
\end{align}
\\
$\mathrm{(2)}$ If $M^{\mathfrak{n}'}_{\mathfrak{n}}=0$ for all $|\mathfrak{n}'-\mathfrak{n}|>N$, then
\begin{align}\label{E1.62}
|M|_{s'}\leq N^{s'-s}|M|_{s},\quad|M|_{s}\leq N^{s+\nu+r}\|M\|_{0},\quad \forall M\in\mathcal{M}^{\mathfrak{B}}_{\mathfrak{C}}.
\end{align}
\end{lemm}
\begin{lemm}\label{lemma10} ({\rm Decay along lines} \cite[Lemma2.9-2.10]{berti2015abstract})
For all subset $\mathfrak{B},\mathfrak{C}\subseteq\mathfrak{N}$ and all $s\geq0$, one has that, for some constant $K_1>0$,
\begin{align}\label{E1.21}
|M|_s\leq K_1|M_\mathfrak{n}|_{s+\nu+r},\quad \forall M\in\mathcal{M}^{\mathfrak{B}}_{\mathfrak{C}},
\end{align}
where $M_\mathfrak{n},\mathfrak{n}\in \mathfrak{C}$ denote its $\mathfrak{n}$-th line. Moreover
\begin{align}\label{E1.66}
\|M\|_{0}\leq|M|_{s_0},\quad \forall M\in\mathcal{M}^{\mathfrak{B}}_{\mathfrak{C}}.
\end{align}
\end{lemm}
Denote by $^{[-1]}M$ any left inverse of $M$.
\begin{lemm}\label{lemma2} ({\rm Perturbation of left-invertible matrices} \cite[Lemma2.12]{berti2015abstract})
If $M\in\mathcal{M}^{\mathfrak{B}}_{\mathfrak{C}}$ has a left invertible matrix $^{[-1]}M$, then, $\forall P\in\mathcal{M}^{\mathfrak{B}}_{\mathfrak{C}}$, with $|^{[-1]}M|_{s_0}|P|_{s_0}\leq1/2$,  the matrix $M+P$ has a left inverse with
\begin{align}\label{E1.33}
|^{[-1]}(M+P)|_{s_0}\leq 2|^{[-1]}M|_{s_0},
\end{align}
and, for all $s\geq s_0$,
\begin{align}\label{E1.92}
\quad |^{[-1]}(M+P)|_{s}\leq C(s)(|^{[-1]}M|_{s}+|^{[-1]}M|^{2}_{s_0}|P|_s).
\end{align}
Moreover, if $\|^{[-1]}M\|_{0}\|P\|_{0}\leq1/2$, then there exists a left inverse $^{[-1]}(M+P)$ of $M+P$ with
\begin{align}\label{E1.34}
\|^{[-1]}(M+P)\|_{0}\leq 2\|^{[-1]}M\|_{0}.
\end{align}
\end{lemm}

\section{Nash-Moser iterative scheme}\label{sec:3}
Consider the orthogonal splitting $H^s=H_{N_n}\oplus H_{N_n}^{\bot}$, where $H^s$ is defined in \eqref{E1.3} and
\begin{align}\label{E1.5}
H_{N_n}~&:=~\Big\{u\in H^s~:~u=\sum\limits_{|\mathfrak{n}|\leq N_n}\textbf{u}_{\mathfrak{n}}e^{\mathrm{i}l\cdot\varphi}\textbf{e}_{j}(\boldsymbol x)\Big\},\\
H^{\bot}_{N_n}~&:=~\Big\{u\in H^s~:~u=\sum\limits_{|\mathfrak{n}|> N_n}\textbf{u}_{\mathfrak{n}}e^{\mathrm{i}l\cdot \varphi}\textbf{e}_{j}(\boldsymbol x)\nonumber\Big\},
\end{align}
with $\textbf{u}_{\mathfrak{n}}\in\mathbf{C}^{\mathfrak{d}_j}$ and
\begin{align}\label{E3.41}
N_{n+1}:=N^2_n,\quad \text{namely}\quad N_{n+1}=N^{2^{n+1}}_0.
\end{align}
Furthermore $P_{N_n},P^{\bot}_{N_n}$ denote the orthogonal projectors onto $H_{N_n}$ and $H^{\bot}_{N_n}$ respectively, namely
\begin{align}\label{E1.7}
P_{N_n}:H^s\rightarrow H_{N_n},\quad P^{\bot}_{N_n}:H^s\rightarrow H^{\bot}_{N_n}.
\end{align}
Then, by means of \eqref{E1.4}, $\forall n\in\mathbf{N},\forall s\geq0,\forall \kappa\geq0$, the following hold:
\begin{align}
&\|P_{N_n}u\|_{s+\kappa}\leq c^\kappa_2N_n^\kappa\|u\|_{s},\quad\forall u\in H^{s},\label{E1.10}\\
&\|P_{N_n}^{\bot}u\|_{s}\leq c^{-\kappa}_1N_n^{-\kappa}\|u\|_{s+\kappa},\quad \forall u\in H^{s+\kappa}.\label{E1.11}
\end{align}
In addition, for all $j_0\in\Gamma^{+}(\boldsymbol M)$, denote by $P_{N,j_0}$ the orthogonal projector from $H^s$ onto the subspace
\begin{align*}
H_{N,j_0}:=\Big\{u\in H^s~:~u=\sum_{|(l,j-j_0)|\leq N}\textbf{u}_{l,j}e^{\mathrm{i}l\cdot\varphi}\textbf{e}_{j}(\boldsymbol x),~\textbf{u}_{l,j}\in\mathbf{C}^{\mathfrak{d}_j}\Big\}.
\end{align*}
This shows that $H_{{N_n},0}=H_{N_n}$ (see \eqref{E1.5}), $P_{N_n,0}=P_{N_n}$ (see \eqref{E1.7}).

 Moreover let $\check{P}_{N,j_0}$ denote the orthogonal projector from $H^{s_0}(\boldsymbol M)$ onto the space
\begin{align*}
\check{H}_{N,j_0}:=\Big\{u\in H^{s_0}~:~u=\sum_{|j-j_0|\leq N}\textbf{u}_{j}\textbf{e}_{j}(\boldsymbol x),~\textbf{u}_{j}\in\mathbf{C}^{\mathfrak{d}_j}\Big\}.
\end{align*}
Remark that the functions on $H^{s_0}(\boldsymbol M)$ depend only on $\boldsymbol x$.

For all $s\geq s_1$ with $s_1\geq s_0>\frac{\nu+d}{2}$, if $f\in C^q(\mathbf{T}^{\nu}\times\boldsymbol M\times\mathbf{R};\mathbf{R})$ with
\begin{align}\label{E1.74}
q\geq s_2+\varrho+2,
\end{align}
then the composition operator $F$ (recall \eqref{E1.102})
has the following standard properties $(P1)$-$(P3)$ (see \cite{berti2015abstract}):\\
$(P1)$(\text{Regularity})~$F\in C^2(H^s,H^s)$.\\
$(P2)$(\text{Tame~estimates})~$\forall u,h\in H^s$ with $\|u\|_{s_1}\leq1$,
\begin{align}
\|F(u)\|_s&\leq C(s)(1+\|u\|_s),\label{E1.14}\\
\|(\mathrm{D}F)(u)h\|_s&\leq C(s)(\|h\|_s+\|u\|_s\|h\|_{s_1}),\label{E1.86}\\
\|\mathrm{D}^2F(u)[h,v]\|_s&\leq C(s)(\|u\|_s\|h\|_{s_1}\|v\|_{s_1}+\|v\|_{s}\|h\|_{s_1}+\|v\|_{s_1}\|h\|_{s}).\label{E1.15}
\end{align}
$(P3)$(\text{Taylor~tame~estimate})~$\forall u,h\in H^s$ with $\|u\|_{s_1}\leq1$, $\|h\|_{s_1}\leq1$,
\begin{align}
&\quad\|F(u+h)-F(u)-(\mathrm{D}F)(u)h\|_{s_1}\leq C(s_1)\|h\|^2_{s_1},\label{E103}\\
&\|F(u+h)-F(u)-(\mathrm{D}F)(u)h\|_{s}\leq C(s)(\|u\|_s\|h\|^2_{s_1}+\|h\|_{s_1}\|h\|_s).\label{E1.16}
\end{align}
In addition the potential $V$ satisfies that, for some fixed constant $C$,
\begin{align}\label{E1.99}
\|V\|_{C^q}\leq C,
\end{align}
where $q$ is defined in \eqref{E1.74}.

\subsection{The multiscale analysis}\label{sec:3.1}
Let
\begin{align}\label{E3.25}
\mathfrak{L}_{N}(\epsilon,\lambda,u):=P_{N}\mathfrak{L}(\epsilon,\lambda,u)_{|H_{N}},
\end{align}
where $\mathfrak{L}(\epsilon,\lambda,u)$ given by \eqref{E1.73}. To guarantee the convergence of the iteration, we need sharper estimates on inversion of the linearized operators $\mathfrak{L}_{N}(\epsilon,\lambda,u)$:
 \begin{align*}
|\mathfrak{L}^{-1}_{N}(\epsilon,\lambda,u)|_{s_1}=O(N^{\tau_2+\delta s}),\quad\delta\in(0,1),\tau_2>0,\forall s>0,
\end{align*}
Hence, for fixed $l\in\mathbf{Z}^{\nu}, \theta\in\mathbf{R}$, some extra properties on the following linear operator
\begin{equation*}
(-{(\lambda{\omega}_0\cdot l+\theta)}^2)\mathrm{I}_{\mathfrak{d}_j}+\check{P}_{N,j_0}(\Delta^2+V(\boldsymbol x))_{|\check{H}_{N,j_0}}
\end{equation*}
are required.
\begin{lemm}\label{lemma3}
For fixed $l\in\mathbf{Z}^{\nu}$, $\theta\in\mathbf{R}$, provided
\begin{align*}
\|((-{(\lambda{\omega}_0\cdot l+\theta)}^2)\mathrm{I}_{\mathfrak{d}_j}+\check{P}_{N,j_0}(\Delta^2+V(\boldsymbol x))_{|\check{H}_{N,j_0}})^{-1}\|_{L^2_{x}}\leq N^{\tau},
\end{align*}
for $N\geq\tilde{N}(s_2,V)$ large enough, one has:
\begin{align*}
|((-{(\lambda{\omega}_0\cdot l+\theta)}^2)\mathrm{I}_{\mathfrak{d}_j}+\check{P}_{N,j_0}(\Delta^2+V(\boldsymbol x))_{|\check{H}_{N,j_0}})^{-1}|_{s}\leq({1}/{2})N^{\tau_2+\delta s},\quad \forall s\in[s_0,s_2].
\end{align*}
\end{lemm}
\begin{proof}
The proof is given in the Appendix.
\end{proof}
Lemma \ref{lemma3} shows that there exists $N_0:=N_0(s_2,\kappa_0,V)\in\mathbf{N}$ (see the first step in the proof of Theorem \ref{theo1}) such that for fixed $l\in[-{N},{N}]^{\nu}\cap\mathbf{Z}^{\nu},\theta\in\mathbf{R}$ with $ \tilde{N}(s_2,V)\leq {N}_0^{{1}/{\chi_0}}\leq N\leq N_0$, if
\begin{align}\label{E1.83}
\|((-{(\lambda{\omega}_0\cdot l+\theta)}^2)\mathrm{I}_{\mathfrak{d}_j}+\check{P}_{{N},j_0}(\Delta^2+V(\boldsymbol x))_{|\check{H}_{{N},j_0}})^{-1}\|_{L^2_{x}}\leq N^{\tau}
\end{align}
holds, then we have
\begin{align}\label{E1.76}
|((-{(\lambda{\omega}_0\cdot l+\theta)}^2)\mathrm{I}_{\mathfrak{d}_j}+\check{P}_{N,j_0}(\Delta^2+V(\boldsymbol x))_{|\check{H}_{N,j_0}})^{-1}|_{s}\leq({1}/{2})N^{(\tau_2+\delta s)},\quad\forall s\in[s_0,s_2].
\end{align}
In addition, in the Fourier basis $e^{\mathrm{i}l\cdot \varphi}$, definition \eqref{E1.85} implies
\[P_{{N},j_0}({L}_\lambda(\theta))_{|H_{{N},j_0}}=\sum\limits_{|l|\leq {N}}((-{(\lambda{\omega}_0\cdot l+\theta)}^2)\mathrm{I}_{\mathfrak{d}_j}+\check{P}_{{N},j_0}(\Delta^2+V(\boldsymbol x))_{|\check{H}_{{N},j_0}})e^{\mathrm{i}l\cdot \varphi},\]
which gives rise to
\begin{align}\label{E1.77}
|(P_{{N},j_0}({L}_\lambda(\theta))_{|H_{{N},j_0}})^{-1}|_s\stackrel{\eqref{E1.82},~\eqref{E1.76}}{\leq}\frac{1}{2}{N}^{(\tau_2+\delta s)},\quad\forall s\in[s_0,s_2].
\end{align}
Then,  from \eqref{E1.22}, \eqref{E1.14}, \eqref{E1.77} and $\|u\|_{s_1}\leq1$, we deduce
\begin{align*}
|(P_{{N},j_0}({L}_\lambda(\theta))_{|H_{{N},j_0}})^{-1}|_{s_1-\varrho}|\epsilon P_{{N},j_0}((\mathrm{D}F)(u))|_{s_1-\varrho}&\leq
\frac{\epsilon }{2}{N}^{\tau_2+\delta( s_1-\varrho)}\| P_{{N},j_0}((\mathrm{D}F)(u))\|_{s_1}\leq1/2
\end{align*}
for $\epsilon {N}^{\tau_2+\delta( s_1-\varrho)}\leq \tilde{\mathfrak{c}}(s_1)$ small enough. Hence it follows from \eqref{E1.77} and Lemma \ref{lemma2} that
\begin{align}\label{E1.78}
|\mathcal{A}^{-1}_{N,j_0}(\epsilon,\lambda,u,\theta)|_s\stackrel{\eqref{E1.33}}{\leq}N^{\tau_2+\delta s}, \quad\forall s\in[s_0,s_1-\varrho].
\end{align}

Based on the fact \eqref{E1.78}, we have the following definitions.
\begin{defi}[$N$-good/$N$-bad matrix]\label{def1}
The matrix $\mathcal{A}\in\mathcal{M}^{\mathfrak{F}}_{\mathfrak{F}}$ with $\mathfrak{F}\subset\mathfrak{N}$ and $\mathrm{diam}(\mathfrak{F})\leq 4N$ is $N$-good if $\mathcal{A}$ is invertible with
\begin{align}\label{E2.2}
|\mathcal{A}^{-1}|_s\leq N^{\tau_2+\delta s}, \quad\forall s\in[s_0,s_1-\varrho].
\end{align}
Otherwise $\mathcal{A}$ is $N$-bad.
\end{defi}
\begin{defi}[Regular/Singular sites]\label{def2}
The index $\mathfrak{n}=(l,j)\in\mathfrak{N}$ is regular for $\mathcal{A}$ if $|\tilde{\mu}_\mathfrak{n}|\geq\tilde{\Theta}$, where $\mathcal{A}^{\mathfrak{n}}_{\mathfrak{n}}:=\tilde{\mu}_\mathfrak{n}\mathrm{I}_{\mathfrak{d}_j}$ with $\tilde{\mu}_\mathfrak{n}:=\tilde{\mu}_\mathfrak{n}(\epsilon,\lambda,\theta):=-(\lambda{\omega}_0\cdot l+\theta)^2+\lambda^2_j+m-\epsilon \bar{m}$.  Otherwise $\mathfrak{n}$ is singular.
\end{defi}
Note that $\bar{m}$ denotes the average of $a(\varphi,\boldsymbol x)$ on $\mathbf{T}^{\nu}\times \boldsymbol M$, where $a(\varphi,\boldsymbol x):=(\partial_{u}f)(\varphi,\boldsymbol x,u(\varphi,\boldsymbol x))$, and that $\tilde{\Theta}$ is given in Proposition \ref{pro1}.
\begin{defi}[($\mathcal{A},N$)-regular/($\mathcal{A},N$)-singular site]\label{def3}
For $\mathcal{A}\in\mathcal{M}^{\mathfrak{A}}_{\mathfrak{A}}$, we say that $\mathfrak{n}\in \mathfrak{A}\subset\mathfrak{N}$ is
$(\mathcal{A},N)$-regular if there exists $\mathfrak{F}\subset \mathfrak{A} $ with $\mathrm{diam}(\mathfrak{F})\leq 4N$, $\mathrm{d}(\mathfrak{n},\mathfrak{A}\backslash \mathfrak{F})\geq N$ such that $\mathcal{A}^{\mathfrak{F}}_{\mathfrak{F}} $ is $N$-good.
\end{defi}
\begin{defi}[($\mathcal{A},N$)-good/($\mathcal{A},N$)-bad site]\label{def4}
The index $\mathfrak{n}=(l,j)\in\mathfrak{N}$ is $(\mathcal{A},N)$-good if it is regular for $\mathcal{A}$ or $(\mathcal{A},N)$-regular. Otherwise we say that $\mathfrak{n}$ is $(\mathcal{A},N)$-bad.
\end{defi}

Define
\begin{align}\label{E1.80}
\mathscr{B}_{N}(j_0):=\mathscr{B}_{N}(j_0;\epsilon,\lambda,u):=\left\{\theta\in\mathbf{R}:\mathcal{A}_{N,j_0}(\epsilon,\lambda,u,\theta)~
\text{is}~N\text{-bad}\right\}.
\end{align}
\begin{defi}[$N$-good/$N$-bad parameters]\label{def7}
A parameter $\lambda\in \Lambda$ is $N$-good for $\mathcal{A}$ if one has that, for all $j_0\in \Gamma_{+}(\boldsymbol M)$,
\begin{align}\label{E2.6}
\mathscr{B}_{N}(j_0)\subset\bigcup_{q=1}^{N^{\nu+d+r+5}}I_q,~\text{where}~ I_q=I_q(j_0)~\text{intervals~with}~\mathrm{meas}(I_q)\leq N^{-\tau}.
\end{align}
Otherwise we say that $\lambda$ is $N$-bad.
\end{defi}
In addition denote
\begin{align}\label{E1.81}
\mathscr{G}_{N}(u):=\{\lambda\in \Lambda:\lambda ~\text{is}~N\text{-good~for}~\mathcal{A}\}.
\end{align}
As a result we have the following lemma:
\begin{lemm}\label{lemma19}
There exist $N_0:=N_0(s_2,\kappa_0,V)\in\mathbf{N}$ and $\tilde{\mathfrak{c}}(s_1)>0$ such that if $\epsilon N_0^{\tau_2+\delta (s_1-\varrho)}\leq \tilde{\mathfrak{c}}(s_1)$,
then  $\forall \kappa_0^{-\frac{1}{\tau}}<N_0^{1/\chi_0}\leq N\leq N_0$, $\forall\|u\|_{s_1}\leq1$, $\forall\epsilon\in[0,\epsilon_0]$, we have $\mathscr{G}_{N}(u)=\Lambda$, 
where $\mathscr{G}_{N}(u)$ is defined in \eqref{E1.81}.
\end{lemm}
\begin{proof}
Denote by $\hat{\lambda}_{j,p},p=1,\cdots,\mathfrak{d}_j$ the eigenvalues of $\check{P}_{N_0,j_0}(\Delta^2+V(\boldsymbol x))_{|\check{H}_{N_0,j_0}}$. It follows from  the fact of  \eqref{E1.78} deduced by \eqref{E1.83} and Definition \ref{def1} that, $\forall|(l,j-j_0)|\leq N,\forall 1\leq p\leq \mathfrak{d}_j$,
\begin{align*}
|-(\lambda{\omega}_0\cdot l+\theta)^2+\hat{\lambda}_{j,p}|> N^{-\tau}\Rightarrow\mathcal{A}_{N,j_0}(\lambda,\epsilon,\theta)~\text{is}~N\text{-good},
\end{align*}
which carries out
\begin{align*}
\mathscr{B}_{N}(j_0)\subset\bigcup_{{|(l,j-j_0)|\leq N,1\leq p\leq \mathfrak{d}_j}}\{\theta\in\mathbf{R}:|-(\lambda{\omega}_0\cdot l+\theta)^2+\hat{\lambda}_{j,p}|\leq N^{-\tau}\}.
\end{align*}
Assumption \eqref{E1.84} implies $\hat{\lambda}_{j,p}\geq\kappa_0>0$. Hence, for all $N>\kappa_0^{-\frac{1}{\tau}}$, we get
\begin{align*}
\{\theta\in\mathbf{R}:|-(\lambda{\omega}_0\cdot l+\theta)^2+\hat{\lambda}_{j,p}|\leq N^{-\tau}\}\subset \mathscr{I}_1\cup \mathscr{I}_2,
\end{align*}
where
\begin{align*}
\mathscr{I}_1&=\left\{\theta\in\mathbf{R}:-\sqrt{\hat{\lambda}_{j,p}+N^{-\tau}}\leq \theta+\lambda{\omega}_0\cdot l\leq-\sqrt{\hat{\lambda}_{j,p}-N^{-\tau}}\right\},\\
\mathscr{I}_2&=\left\{\theta\in\mathbf{R}:\sqrt{\hat{\lambda}_{j,p}-N^{-\tau}}\leq \theta+\lambda\omega_0\cdot l\leq\sqrt{\hat{\lambda}_{j,p}+N^{-\tau}}\right\}.
\end{align*}

It is easy that, for $q=1,2$,
\begin{align*}
\mathrm{meas}(\mathscr{I}_q)=&\sqrt{\hat{\lambda}_{j,p}+N^{-\tau}}-\sqrt{\hat{\lambda}_{j,p}-N^{-\tau}}=\frac{2N^{-\tau}}
{\sqrt{\hat{\lambda}_{j,p}+N^{-\tau}}+\sqrt{\hat{\lambda}_{j,p}-N^{-\tau}}}\\
\leq &\frac{2N^{-\tau}}{\sqrt{\hat{\lambda}_{j,p}}}\leq\frac{2N^{-\tau}}{\sqrt{\kappa_0}}.
\end{align*}
Since $|(l,j-j_0)|\leq N,1\leq p\leq \mathfrak{d}_j$ with $\mathfrak{d}_j\leq\|j+\rho\|^{d-r}$, we obtain
\begin{align*}
\mathscr{B}_{N}(j_0)\subset\bigcup_{q=1}^{K_1\mathcal{C}N^{\nu+d}}I_q,\quad I_q~ \text{intervals~with}~\mathrm{meas}(I_q)\leq N^{-\tau},
\end{align*}
where $K_1=[2/\sqrt{\kappa_0}]+1$. The symbol $[~\cdot~]$ denotes the integer part.
\end{proof}

Our goal is to show that a matrix $\mathcal{A}$ at the larger scale $N'$, where
\begin{align}\label{E2.1}
N'=N^{\chi} \quad\text{with}\quad \chi>1
\end{align}
is $N'$-good under some conditions, see \eqref{E2.44}-\eqref{E2.46} and $(\mathrm{A1})$-$(\mathrm{A3})$ in proposition \ref{pro1}.
\begin{prop}\label{pro1}
Assume
\begin{align}\label{E2.44}
&\delta\in(0,1/2),\quad\tau_2>2\tau+\nu+r+1,\quad C_1:=C_1(\nu,d,r)\geq2,\quad \chi\in[\chi_0,2\chi_0],\\
&\chi_0(\tau_2-2\tau-\nu-r)>3(\mathfrak{e}+C_1(s_0+\nu+r)),\quad \chi_0\delta> C_1,\label{E2.45}\\
&3\mathfrak{e}+2\chi_0(\tau+\nu+r)+C_1s_0<s_1-\varrho\leq s_2,\label{E2.46}
\end{align}
where $\mathfrak{e}:=\tau_2+\nu+r+s_0$. For all given $\Upsilon>0$, there exist $\tilde{\Theta}:=\tilde{\Theta}(\Upsilon,s_1)>0$ large enough and $\bar{N}(\Upsilon,\tilde{\Theta},s_2)\in\mathbf{N}$ such that: for all $N\geq\bar{N}(\Upsilon,\tilde{\Theta},s_2)$ and $\mathfrak{A}\in\mathfrak{N}$ with $\mathrm{diam}(\mathfrak{A})\leq4N'$, if $\mathcal{A}\in\mathcal{M}^{\mathfrak{A}}_{\mathfrak{A}}$ satisfies

$(\mathrm{A1})$ $|\mathcal{Q|}_{s_1-\varrho}\leq\Upsilon$ with $\mathcal{Q}=\mathcal{A}-\mathrm{Diag}(\mathcal{A})$,

$(\mathrm{A2})$ $\|\mathcal{A}^{-1}\|_{0}\leq(N')^{\tau}$,

$(\mathrm{A3})$ The set of ($\mathcal{A},N$)-bad sites $\mathfrak{B}$ admits a partition $\bigcup_{\alpha}\mathfrak{O}_{\alpha}$ into disjoint clusters with
\begin{align}\label{E2.30}
\mathrm{diam}(\mathfrak{O}_{\alpha})\leq N^{C_1},\quad \mathrm{d}(\mathfrak{O}_\alpha,\mathfrak{O}_\beta)\geq N^2, \quad
\forall\alpha\neq\beta,
\end{align}
then $\mathcal{A}$ is $N'$-good with
\begin{align}\label{E2.43}
|\mathcal{A}^{-1}|_{s}\leq\frac{1}{4}(N')^{\tau_2}((N')^{\delta s}+|\mathcal{Q}|_{s}),\quad \forall s\in[s_0,s_2].
\end{align}
\end{prop}
\begin{proof}
The proof of the proposition is shown in the Appendix.
\end{proof}

\subsection{Separation properties of bad sites}\label{sec:3.2}
Let us check that the assumption $(\mathrm{A3})$ in Proposition \ref{pro1} holds.
\begin{defi}[($\mathcal{A}(u,\theta),N$)-strongly-regular/($\mathcal{A}(u,\theta),N$)-weakly-singular site]\label{def5}
The index $\mathfrak{n}=(l,j)\in\mathfrak{N}$ is ($\mathcal{A}(u,\theta),N$)-strongly-regular if $\mathcal{A}_{N,l,j}(\epsilon,\lambda,u,\theta)$ (see \eqref{E1.94}) is $N$-good, where $\mathcal{A}(\epsilon,\lambda,u,\theta)$ is defined in \eqref{E1.97}. Otherwise $\mathfrak{n}$ is $(\mathcal{A}(u,\theta),N)$-weakly-singular.
\end{defi}
\begin{defi}[($\mathcal{A}(u,\theta),N$)-strongly-good/($\mathcal{A}(u,\theta),N$)-weakly-bad site]\label{def6}
The index $\mathfrak{n}=(l,j)\in\mathfrak{N}$ is ($\mathcal{A}(u,\theta),N$)-strongly-good if it is regular for $\mathcal{A}(\epsilon,\lambda,u,\theta)$ or all the sites $\mathfrak{n}'=(l',j')$ with $\mathrm{d}(\mathfrak{n},\mathfrak{n}')\leq N$ are ($\mathcal{A}(u,\theta),N$)-strongly-regular. Otherwise $\mathfrak{n}$ is $(\mathcal{A}(u,\theta),N)$-weakly-bad.
\end{defi}
\begin{lemm}\label{lemma5}
For $j_0\in \Gamma_{+}(\boldsymbol M), \chi\in[\chi_0,2\chi_0]$, if the site $\mathfrak{n}=(l,j)\in\mathfrak{N}$ with $|l|\leq N',|j-j_0|\leq N'$ is ($\mathcal{A}(u,\theta),N$)-strongly-good, then it is $(\mathcal{A}_{N',j_0}(\epsilon,\lambda,u,\theta),N)$-good.
\end{lemm}
\begin{proof}
Let $\mathfrak{K}=\mathfrak{W}\times\mathfrak{J}$ with
\begin{align*}
\mathfrak{W}:=[-N',N']^{\nu}\cap\mathbf{Z}^\nu,
\quad\mathfrak{J}:=\Big(j_0+\Big\{\sum\limits_{p=1}^{r}j_{k}\textbf{w}_k:j_k\in[-N',N']\Big\}\Big)\cap\Gamma_{+}(\boldsymbol M).
\end{align*}

If $\mathfrak{n}=(l,j)\in\mathfrak{N}$ with $|l|\leq N',|j-j_0|\leq N'$ is regular for $\mathcal{A}(\epsilon,\lambda,u,\theta)$, then Definition \ref{def4} and \eqref{E2.50} verify that it is  $(\mathcal{A}_{N',j_0}(\epsilon,\lambda,u,\theta),N)$-good.

If $\mathfrak{n}=(l,j)\in\mathfrak{N}$ with $|l|\leq N',|j-j_0|\leq N'$ is ($\mathcal{A}(u,\theta),N$)-strongly-regular, and  set $l:=(l_k)_{1\leq k\leq\nu}$, $\mathfrak{r}_k:=(j_0)_{k}-N',\mathfrak{t}_k:=(j_0)_{k}+N',\bar{\mathfrak{r}}_k:=-N',\bar{\mathfrak{t}}_k:=N'$, then the $N$-ball of $\mathfrak{n}$ is  defined as $\mathfrak{F}_N:=\mathfrak{F}(\mathfrak{n},N):=\mathfrak{W}_{N}\times\mathfrak{J}_{N}$, where
\begin{align*}
\mathfrak{W}_{N}:=\Big(\prod_{k=1}^{\nu}W_k\Big)\cap \mathbf{Z}^{\nu},\quad \mathfrak{J}_{N}:=\Big\{\sum\limits_{k=1}^{r}\mathfrak{j}_{k}\textbf{w}_k:\mathfrak{j}_{k}\in{J}_k \Big\}\cap \Gamma_{+}(\boldsymbol M),\quad \text{with}
\end{align*}
\begin{align*}
l_k-\bar{\mathfrak{r}}_k>N,\quad\bar{\mathfrak{t}}_k-l_k>N\Rightarrow W_{k}:=&[l_k-N,l_k+N],\\
l_k-\bar{\mathfrak{r}}_k\leq N,\quad\bar{\mathfrak{t}}_k-l_k>N\Rightarrow W_{k}:=&[\bar{\mathfrak{r}}_k,\bar{\mathfrak{r}}_k+2N],\\
l_k-\bar{\mathfrak{r}}_k>N,\quad\bar{\mathfrak{t}}_k-l_k\leq N\Rightarrow W_{k}:=&[\bar{\mathfrak{t}}_k-2N,\bar{\mathfrak{t}}_k];
\end{align*}
and
\begin{align*}
j_k-{\mathfrak{r}}_k>N,\quad{\mathfrak{t}}_k-j_k>N\Rightarrow J_{k}:=&[j_k-N,j_k+N],\\
j_k-{\mathfrak{r}}_k\leq N,\quad{\mathfrak{t}}_k-j_k>N\Rightarrow J_{k}:=&[{\mathfrak{r}}_k,{\mathfrak{r}}_k+2N],\\
j_k-{\mathfrak{r}}_k>N,\quad{\mathfrak{t}}_k-j_k\leq N\Rightarrow J_{k}:=&[\mathfrak{t}_k-2N,\mathfrak{t}_k].\\
\end{align*}
It is obvious that $\mathrm{d}(\mathfrak{n},\mathfrak{K})>N$ and $\mathrm{diam}(\mathfrak{F}_N)\leq2N<4N$. With the help of \eqref{E1.95}, there exists $\mathfrak{n}'=(l',j')\in\mathfrak{K}$ with $\mathrm{d}(\mathfrak{n},\mathfrak{n}')\leq N$ such that
\begin{align*}
\mathfrak{F}_N=\Big((l'+[-N,N]^{\nu})\times
\Big(j'+\Big\{\sum\limits_{k=1}^{r}\mathfrak{l}_{k}\textbf{w}_k:\mathfrak{l}_k\in[-N,N]\Big\}\Big)\Big)\cap\mathfrak{N}.
\end{align*}
Since $\mathfrak{n}$ is ($\mathcal{A}(u,\theta),N$)-strongly-regular, by Definition \ref{def6}, then
\begin{align*}
|(\mathcal{A}^{\mathfrak{F}_N}_{\mathfrak{F}_N})^{-1}(\epsilon,\lambda,u,\theta)|_s=|(\mathcal{A}_{N,l',j'})^{-1}(\epsilon,\lambda,u,\theta)|_s\leq N^{\tau_2+\delta s}.
\end{align*}
\end{proof}
Lemma \ref{lemma5} establishes that if $\mathfrak{n}_{0}=(l_0,j_0)\in\mathfrak{N}$ is $(\mathcal{A}_{N',l_0,j_0}(\epsilon,\lambda,u,\theta),N)$-bad, then it it $(\mathcal{A}(u,\theta),N)$-weakly-bad for $\mathcal{A}(\epsilon,\lambda,u,\theta)$ with  $|l-l_0|\leq N'$, $|j-j_0|\leq N'$. Our goal is to get the upper bound of the number of ($\mathcal{A}(u,\theta),N$)-weakly-bad sites $(l_0,j_0)$ with $|l_0|\leq N'$.
\begin{lemm}\label{lemma6}
Let $\lambda$ be $N$-good for $\mathcal{A}(\epsilon,\lambda,u,\theta)$. For $j_{0}\in \Gamma_{+}(\boldsymbol M), \chi\in[\chi_0,2\chi_0]$, the number of ($\mathcal{A}(u,\theta),N$)-weakly-singular sites $(l_0,j_0)$ with $|l_0|\leq2N'$ does not exceed $N^{\nu+d+r+5}$.
\end{lemm}
\begin{proof}
Definition \ref{def5} implies that $\mathcal{A}_{N,l_0,j_0}(\epsilon,\lambda,u,\theta)$ is $N$-bad if  $(l_0,j_0)$  is ($\mathcal{A}(u,\theta),N$)-weakly-singular. Using the covariance property \eqref{E1.93}, we obtain that $\mathcal{A}_{N,j_0}(\epsilon,\lambda,u,\theta+\lambda{\omega}_0\cdot l_0)$ is $N$-bad,  which leads to $\theta+\lambda{\omega}_0\cdot l_0\in\mathscr{B}_{N}(j_0)$ (recall \eqref{E1.80}). By assumption that $\lambda$ is $N$-good, we have that \eqref{E2.6} holds. For $\tau>2\chi_0\nu$,  we claim that there exists at most one element $\theta+\lambda{\omega}_0\cdot l_0 $ with $|l_0|\leq2N'$ in each interval $I_q$, which implies the conclusion of the lemma by \eqref{E2.6}.

Let us check the claim. Suppose that there exists $l_0\neq l'_0$ with $|l_0|,|l'_0|\leq 2N'$ such that $\theta+\lambda{\omega}_0\cdot l_0,\theta+\lambda{\omega}_0\cdot l'_0\in I_q$. Then
\begin{align}\label{E2.7}
|\lambda{\omega}_0\cdot( l_0-l'_0)|=|(\theta+\lambda{\omega}_0\cdot l_0)-(\theta+\lambda{\omega}_0\cdot l'_0)|\leq \mathrm{meas}(I_q)\leq N^{-\tau}.
\end{align}
Moreover  \eqref{E1.96} gives that, for $\lambda\in\Lambda=[1/2,3/2]$,
\begin{align*}
|\lambda{\omega}_0\cdot l|\geq\gamma_0|l|^{-\nu},\quad\forall l\in\mathbf{Z}^\nu \backslash \{0\},
\end{align*}
which carries out
\begin{align*}
|\lambda{\omega}_0\cdot(l_0-l'_0)|\geq\frac{\gamma_0}{|l_0-l'_0|^{\nu}}\geq\frac{\gamma_0}{(4N')^{\nu}}
\stackrel{\eqref{E2.1}}{=}\frac{\gamma_0}{4^\nu}N^{-\chi\nu}.
\end{align*}
If $\tau>2\chi_0\nu$, then this leads to a contradiction to \eqref{E2.7} for $N\geq\bar{N}(\gamma_0,\nu)$ large enough.
\end{proof}
\begin{coro}\label{coro1}
Let $\lambda$ be $N$-good for $\mathcal{A}(\epsilon,\lambda,u,\theta)$. For $j_{0}\in \Gamma_{+}(\boldsymbol M)$, the number of ($\mathcal{A}(u,\theta),N$)-weakly-bad sites $(l_0,j_0)$ with $|l_0|\leq N'$ does not exceed $N^{2
\nu+d+2r+6}$.
\end{coro}
\begin{proof}
Since $|l-l_0|\leq N\Rightarrow|l|\leq N'+N$,  Lemma  \ref{lemma6} establishes that the number of ($\mathcal{A}(u,\theta),N$)-weakly-singular sites $(l,j)$ with $|l-l_0|\leq N,|j-j_0|\leq N$ is bounded from above by $N^{\nu+d+r+5}\times 4^rN^r$.
By Definition \ref{def6}, each $(l_0,j_0)$, which is ($\mathcal{A}(u,\theta),N$)-weakly-bad, is included in some $N$-ball centered at an ($\mathcal{A}(u,\theta),N$)-weakly-singular site. Moreover each of these balls contain at most $4^\nu N^{\nu}$ sites of the form $(l,j_0)$. Hence the number of ($\mathcal{A}(u,\theta),N$)-weakly-bad sites is at most $4^{\nu+r}N^{{\nu+d+2r+5}}\times N^{\nu}$.
\end{proof}
Let us estimate the spatial components of the ($\mathcal{A}(u,\theta),N$)-weakly-bad sites for $\mathcal{A}(\epsilon,\lambda,u,\theta)$ with $|(l,j-j_0)|\leq N'$.
\begin{defi}
Denote by $\{\mathfrak{n}_k,k\in[0,L]\cap\mathbf{N}\}$ a sequence of site with $\mathfrak{n}_k\neq \mathfrak{n}_{k'}$, $\forall k\neq k'$. For $\tilde{B}\geq2$, we call $\{\mathfrak{n}_k,k\in[0,\tilde{L}]\cap\mathbf{N}\}$ a $\tilde{B}$-chain of length $\tilde{L}$ with $|\mathfrak{n}_{k+1}-\mathfrak{n}_{k}|\leq {\tilde{B}}$, $\forall k=0,\cdots,\tilde{L}-1.$
\end{defi}
\begin{lemm}
For $\epsilon$ small enough, there exists $\mathcal{C}(\nu,d,r)>0$ such that, $\forall\theta\in\mathbf{R}$ and $\forall N\in\mathbf{N}^{+}$, any ${\tilde{B}}$-chain of the ($\mathcal{A}(u,\theta),N$)-weakly-bad sites for $\mathcal{A}(\epsilon,\lambda,u,\theta)$ with $|(l,j-j_0)|\leq N'$ has length
$\tilde{L}\leq {(\tilde{B}N)}^{\mathcal{C}(\nu,d,r)}$.
\end{lemm}
\begin{proof}
Here we exploit that $\mathfrak{n}\in\mathfrak{N}$ is singular if it is ($\mathcal{A}(u,\theta),N$)-weakly-bad. Denote by $\{\mathfrak{n}_k,k\in[0,\tilde{L}]\cap\mathbf{N}\}$ a ${\tilde{B}}$-chain of singular sites. Then
\begin{align}\label{E2.21}
\max\{|l_{k+1}-l_k|,|j_{k+1}-j_k|\}\leq {\tilde{B}},\quad\forall k=0,\cdots,\tilde{L}-1.
\end{align}
It follows from  Definition \ref{def2} and the definition of $\lambda_j$ that
\begin{align}\label{E2.28}
&|-(\lambda{\omega}_0\cdot l_k+\theta)^2+(\|j_{k}+\rho\|^2-\|\rho\|^2)^2+m|<\tilde{\Theta}+1.
\end{align}
In fact, Definition \ref{def2} shows that
\[|-(\lambda{\omega}_0\cdot l+\theta)^2+\lambda^2_j+m-\epsilon \bar{m}|\leq \tilde{\Theta}.\]
 Clearly, we obtain that $|\epsilon \bar{m}|\leq1$ if $\epsilon$ is small enough, which lead to \eqref{E2.28}. With the help of \eqref{E2.28}, we deduce
\begin{align*}
&|-(\lambda{\omega}_0\cdot l_k+\theta)^2+(\|j_{k}+\rho\|^2-\|\rho\|^2)^2|<\tilde{\Theta}+1+m\\
\Rightarrow&|-\lambda{\omega}_0\cdot l_k-\theta+\|j_{k}+\rho\|^2-\|\rho\|^2|<\sqrt{\tilde{\Theta}+1+m}\\
&\text{or}\quad|\lambda{\omega}_0\cdot l_k+\theta+\|j_{k}+\rho\|^2-\|\rho\|^2|<\sqrt{\tilde{\Theta}+1+m}.
\end{align*}
Then one of the following $\theta$-independent inequalities holds:
\begin{align*}
&\left|\pm(\lambda{\omega}_0\cdot( l_{k+1}-l_k))+(\|j_{k+1}+\rho\|^2\pm\|j_{k}+\rho\|^2)\right|<2\left(\sqrt{\tilde{\Theta}+1+m}+\|\rho\|^2\right),
\end{align*}
which leads to
\begin{align*}
\left|\|j_{k+1}+\rho\|^2 \pm\|j_{k}+\rho\|^2 \right|<2\left(\sqrt{\tilde{\Theta}+1+m}+\|\rho\|^2+\tilde{B}\right).
\end{align*}
Combining this with the inequality
\[\left|\|j_{k+1}+\rho\|^2-\|j_{k}+\rho\|^2\right|\leq\|j_{k+1}+\rho\|^2 +\|j_{k}+\rho\|^2\]
yields
\begin{align}\label{E2.22}
\left|\|j_k+\rho\|^2-\|j_{k_0}+\rho\|^2\right|\leq2\left(\sqrt{\tilde{\Theta}+1+m}+\|\rho\|^2+\tilde{B}\right)|k-k_0|.
\end{align}
Due to \eqref{E1.50}, \eqref{E2.21}, \eqref{E2.22} and  the equality
\[(j_{k_0}+\rho)\cdot(j_k-j_{k_0})=\frac{1}{2}\left(\|j_k+\rho\|^2-\|j_{k_0}+\rho\|^2-\|j_k-j_{k_0}\|^2\right),\]
we obtain
\begin{align}\label{E2.23}
\left|(j_{k_0}+\rho)\cdot(j_k-j_{k_0})\right|&\leq\left(\sqrt{\tilde{\Theta}+1+m}+\|\rho\|^2+\tilde{B}\right)|k-k_0|+(b^2_2/2)|k-k_0|^2{\tilde{B}}^2\nonumber\\
&\leq\left(\sqrt{\tilde{\Theta}+1+m}+\|\rho\|^2+b^2_2+1\right)|k-k_0|^2{\tilde{B}}^2.
\end{align}

Define the following subspace of $\mathbf{R}^{r}$ by
\begin{equation*}
\mathscr{F}:=\mathrm{span}_{\mathbf{R}}\{j_{k}-j_{k'}:~k,k'=0,\cdots,\tilde{L}\}=\mathrm{span}_{\mathbf{R}}\{j_{k}-j_{k_0}:~k=0,\cdots,\tilde{L}\}.
\end{equation*}
Let $\mathfrak{t}$ be the dimension of $\mathscr{F}$. Denote by $\zeta_1,\cdots,\zeta_{\mathfrak{t}}$ a basis of $\mathscr{F}$. It is clear that $\mathfrak{t}\leq r$.

Case1. For all $k_0\in[0,\tilde{L}]\cap\mathbf{N}$, we have
\begin{equation*}
\mathscr{F}_{k_0}:=\mathrm{span}_{\mathbf{R}}\{j_{k}-j_{k_0}:~|k-k_0|\leq \tilde{L}^{\upsilon},~ k=0,\cdots,\tilde{L}\}=\mathscr{F}.
\end{equation*}
Formula \eqref{E2.21} indicates that
\begin{align}\label{E2.24}
|\zeta_{p}|=|j_p-j_{k_0}|\leq|p-k_0|{\tilde{B}}\leq \tilde{L}^{\upsilon}{\tilde{B}},\quad p=0,\cdots,\tilde{L}.
\end{align}
Let $\Pi_{\mathscr{F}}$ denote the orthogonal projection on $\mathscr{F}$. Then
\begin{equation}\label{E2.27}
\Pi_{\mathscr{F}}(j_{k_0}+\rho)=\sum\limits_{p=1}^{\mathfrak{t}}z_p\zeta_p
\end{equation}
for some $z_p\in\mathbf{R},p=1,\cdots,\mathfrak{t}$. Hence we get
\begin{equation*}
\Pi_{\mathscr{F}}(j_{k_0}+\rho)\cdot \zeta_{p'}=\sum\limits_{p=1}^{\mathfrak{t}}z_p\zeta_p\cdot \zeta_{p'}.
\end{equation*}
Based on above fact, we consider the linear system $Qz=y$, where
\begin{align*}
Q=(Q_{pp'})_{p,p'=1,\cdots,\mathfrak{t}}\quad \text{with}\quad Q_{pp'}=\zeta_p\cdot\zeta_{p'},\quad y_{p'}=\Pi_{\mathscr{F}}(j_{k_0}+\rho)\cdot \zeta_{p'}=(j_{k_0}+\rho)\cdot \zeta_{p'}.
\end{align*}
It follows from \eqref{E2.23}-\eqref{E2.24} that
\begin{align}\label{E2.25}
|y_{p'}|\leq\left(\sqrt{\tilde{\Theta}+1+m}+\|\rho\|^2+b^2_2+1\right)(\tilde{L}^{\upsilon}\tilde{B})^2, \quad|Q_{pp'}|\leq(\tilde{L}^{\upsilon}{\tilde{B}})^2.
\end{align}
In addition, by formula \eqref{E1.56}, we verify
\begin{align}\label{E2.26}
\mathfrak{z}^{\mathfrak{t}}\det(Q)\in\mathbf{Z},\quad\text{namely}\quad\mathfrak{z}^{\mathfrak{t}}|\det(Q)|\geq1.
\end{align}
Let $Q^*$ be the adjoint matrix of $Q$. It follows from  Hadamard inequality that
\begin{equation*}
\left|(Q^*)_{pp'}\right|\leq\prod\limits_{\mathfrak{p}\neq p,1\leq\mathfrak{p}\leq \mathfrak{t}}\Big(\sum\limits_{\mathfrak{p}'\neq p',1\leq\mathfrak{p}'\leq \mathfrak{t}}|Q_{\mathfrak{p}\mathfrak{p}'}|^2\Big)^{1/2},
\end{equation*}
which leads to
\begin{align*}
\left|(Q^*)_{pp'}\right|\stackrel{\eqref{E2.25}}{\leq}(\mathfrak{t}-1)^{\frac{\mathfrak{t}-1}{2}}{(\tilde{L}^{\upsilon}{\tilde{B}})}^{2(\mathfrak{t}-1)}.
\end{align*}
Based on above inequality, \eqref{E2.25}-\eqref{E2.26} and Cramer's rule,   we can obtain
\begin{align*}
|z_p|\leq\sum\limits_{p'=1}^{\mathfrak{t}}|Q^{-1}_{pp'}y_{p'}|\leq \mathfrak{z}^{\mathfrak{t}}{\mathfrak{t}}^{\mathfrak{t}}\left(\sqrt{\tilde{\Theta}+1+m}+\|\rho\|^2+b^2_2+1\right)
(\tilde{L}^{\upsilon}{\tilde{B}})^{2\mathfrak{t}}.
\end{align*}
Combining this with formulae \eqref{E2.24}-\eqref{E2.27} derives
\begin{align*}
\Pi_{\mathscr{F}}(j_{k_0}+\rho)\leq \mathfrak{t}|z_p||\zeta_{p}|\leq \mathfrak{z}^{\mathfrak{t}}{\mathfrak{t}}^{\mathfrak{t}+1}\left(\sqrt{\tilde{\Theta}+1+m}+\|\rho\|^2+b^2_2+1\right)
(\tilde{L}^{\upsilon}{\tilde{B}})^{2\mathfrak{t}+1}.
\end{align*}
As a consequence
\begin{align*}
|j_{k_1}-j_{k_2}|&=|(j_{k_1}-j_{k_0})-(j_{k_2}-j_{k_0})|=|\Pi_{\mathscr{F}}(j_{k_1}+\rho)-\Pi_{\mathscr{F}}(j_{k_2}+\rho)|\\
&\leq2\mathfrak{z}^{r}r^{r+1}(\sqrt{\tilde{\Theta}+1+m}+\|\rho\|^2+b^2_2+1)(\tilde{L}^{\upsilon}{B})^{2r+1}.
\end{align*}
Counted without multiplicity,  the number of $j_k\in\Gamma^{+}(\boldsymbol M)$  is bounded from above by
\[4^r\left(2\mathfrak{z}^{r}r^{r+1}\left(\sqrt{\tilde{\Theta}+1+m}+\|\rho\|^2+b^2_2+1\right)(\tilde{L}^{\upsilon}{\tilde{B}})^{2r+1}\right)^r,\]
namely
\begin{align}\label{E2.29}
\sharp\left\{j_k,0\leq k\leq \tilde{L}\right\}\leq 2^{3r}\mathfrak{z}^{r^2}r^{r(r+1)}\left(\sqrt{\tilde{\Theta}+1+m}+\|\rho\|^2+b^2_2+1\right)^{r}(\tilde{L}^{\upsilon}{\tilde{B}})^{r(2r+1)}.
\end{align}
For each $k_0\in[0,\tilde{L}]$, Corollary \ref{coro1} shows that the number of $k\in[0,\tilde{L}]$ such that $j_k=j_{k_0}$ does not exceed $N^{2
\nu+d+2r+6}$. Hence, in view of \eqref{E2.29},  the following  holds:
\begin{align}\label{E2.47}
\tilde{L}\leq2^{3r}\mathfrak{z}^{r^2}r^{r(r+1)}\left(\sqrt{\tilde{\Theta}+1+m}+\|\rho\|^2+b^2_2+1\right)^{r}(\tilde{L}^{\upsilon}{\tilde{B}})^{r(2r+1)}N^{2\nu+d+2r+6}.
\end{align}
If $\upsilon<\frac{1}{2r(2r+1)}$, then \eqref{E2.47}  derives
\begin{align*}
\tilde{L}^{\frac{1}{2}}&\leq2^{3r}\mathfrak{z}^{r^2}r^{r(r+1)}\left(\sqrt{\tilde{\Theta}+1+m}+\|\rho\|^2+b^2_2\right)^{r}{\tilde{B}}^{r(2r+1)}N^{2\nu+d+2r+6}\\
&\Rightarrow
\tilde{L}\leq 2^{6r}\mathfrak{z}^{2r^2}r^{2r(r+1)}\left(\sqrt{\tilde{\Theta}+1+m}+\|\rho\|^2+b^2_2\right)^{2r}{\tilde{B}}^{2r(2r+1)}N^{2(2\nu+d+2r+6)}.
\end{align*}

Case2. If there exist some $k'_0\in[0,\tilde{L}]\cap\mathbf{N}$ such that $\dim\mathscr{F}_{k'_0}\leq r-1$, for $k_0\in \tilde{\mathfrak{I}}$, then we consider
\begin{align*}
\mathscr{F}^1_{k_0}:=\mathrm{span}_{\mathbf{R}}\left\{j_{k}-j_{k_0}:|k-k_0|<\tilde{L}^\upsilon_1,~k \in \tilde{\mathfrak{I}}\right\}=\mathrm{span}_{\mathbf{R}}\left\{j_{k}-j_{k_0}:~k \in \tilde{\mathfrak{I}}\right\}
\end{align*}
where
\[\tilde{L}_1=\tilde{L}^{\upsilon}, \quad \tilde{\mathfrak{I}}:=\{k:|k-k'_0|<\tilde{L}^{\upsilon}\}\cap([0,\tilde{L}]\cap\mathbf{N}).\]
The upper bound of $\tilde{L}_1$ can be proved by the same method as employed on $\tilde{L}$, namely
\begin{align*}
\tilde{L}_1=\tilde{L}^\upsilon\leq2^{6r}\mathfrak{z}^{2r^2}r^{2r(r+1)}\left(\sqrt{\tilde{\Theta}+1+m}+\|\rho\|^2+b^2_2\right)^{2r}
{\tilde{B}}^{2r(2r+1)}N^{2(2\nu+d+2r+6)}.
\end{align*}
In addition the iteration is carried out at most $r$ steps owing to the fact $\mathfrak{t}\leq r$. Hence
\begin{align*}
\tilde{L}_r=\tilde{L}^{r\upsilon}\leq2^{6r}\mathfrak{z}^{2r^2}r^{2r(r+1)}\left(\sqrt{\tilde{\Theta}+m}+\|\rho\|^2+b^2_2\right)^{2r}
{\tilde{B}}^{2r(2r+1)}N^{2(2\nu+d+2r+6)}.
\end{align*}
Hence there exists some constant $\mathcal{C}(\nu,d,r)>0$ such that $\tilde{L}\leq {(\tilde{B}N)}^{\mathcal{C}(\nu,d,r)}.$
\end{proof}
In addition, the following equivalence relation is defined.
\begin{defi}\label{def8}
We say that $\tilde{\mathfrak{x}}\equiv \tilde{\mathfrak{y}}$ if there is a $N^2$-chain $\{\mathfrak{n}_k,k\in[0,\tilde{L}]\cap\mathbf{N}\}$ connecting $\tilde{\mathfrak{x}}$ to $\tilde{\mathfrak{y}}$, namely
$\mathfrak{n}_0=\tilde{\mathfrak{x}},\mathfrak{n}_{\tilde{L}}=\tilde{\mathfrak{y}}$.
\end{defi}
 Let us state the following proposition.
\begin{prop}\label{pro3}
If we suppose
\begin{align*}
\mathrm{(1)}~ \lambda~\text{is}~N\text{-good for}~\mathcal{A}~, \quad\mathrm{(2)}~\tau>2\chi_{0}\nu,
\end{align*}
then there exist $C_1:=C_1(\nu,d,r)\geq2$ and $\hat{N}:=\hat{N}(\nu,d,r,\gamma_0,m,\mathfrak{z},\tilde{\Theta},\rho,b_2)$ such that, $\forall N\geq \hat{N}$, $\forall\theta\in\mathbf{R}$,  the ($\mathcal{A}(u,\theta),N$)-weakly-bad sites for $\mathcal{A}(\epsilon,\lambda,u,\theta)$ with $|l|\leq N',|j-j_0|\leq N'$ admits a partition $\cup_{\alpha}\mathfrak{O}_{\alpha}$, where
\begin{equation*}
\mathrm{diam}(\mathfrak{O}_{\alpha})\leq N^{C_1},\quad\mathrm{d}(\mathfrak{O}_{\alpha},\mathfrak{O}_{\beta})>N^2, \quad \forall\alpha\neq\beta.
\end{equation*}
\end{prop}
\begin{proof}
Let ${\tilde{B}}=N^2$. By Definition \ref{def8}, the equivalence relation induces that a partition of the ($\mathcal{A}(u,\theta),N$)-weakly-bad sites for $\mathcal{A}(\epsilon,\lambda,u,\theta)$ with $|l|\leq N',|j-j_0|\leq N'$ satisfies
\begin{equation*}
\mathrm{diam}(\mathfrak{O}_{\alpha})\leq \tilde{L}\tilde{{B}}\leq N^{C_1},\quad\mathrm{d}(\mathfrak{O}_{\alpha},\mathfrak{O}_{\beta})>N^2, \quad\forall\alpha\neq\beta,
\end{equation*}
where $C_1:=C_1(\nu,d,r)=3\mathcal{C}(\nu,d,r)+2$.
\end{proof}
Thus the assumption $(\mathrm{A3})$ in Proposition \ref{pro1} holds by Proposition \ref{pro3} for $j_0=0,\theta=0$.
\subsection{Measure and ``complexity'' estimates}\label{sec:3.3}
We define
\begin{align}\label{E2.53}
\mathscr{B}_{N}^{0}(j_0):=&\mathscr{B}_{N}^{0}(j_{0};\epsilon,\lambda,u):=\left\{\theta\in\mathbf{R}:\|
\mathcal{A}_{N,j_{0}}^{-1}(\epsilon,\lambda,u,\theta)\|_{0}>N^{\tau}\right\}\nonumber\\
=&\left\{\theta\in\mathbf{R}:\exists~\text{an~eigenvalue~of}~\mathcal{A}_{N,j_{0}}(\epsilon,\lambda,u,\theta)~\text{with~
 modulus~less~than}~N^{-\tau}\right\},
 \end{align}
where $\parallel\cdot\parallel_{0}$ is the operator $L^{2}$-norm.
Moreover we also define
\begin{align}
\mathscr{G}_{N}^{0}(u):=\Bigg\{&\lambda\in\Lambda:\forall
j_{0}\in\Gamma_{+}(\boldsymbol M),~\mathscr{B}_{N}^{0}(j_0)\subset
\bigcup_{q=1}^{N^{\nu+d+r+5}}I_{q},~
\text{where}~I_{q}=I_q(j_0)~\nonumber\\
&\text{are~disjoint~intervals~with~measure}~\mathrm{meas}{(I_{q})}\leq
 N^{-\tau}\Bigg\}.\label{E3.13}
\end{align}
It follows from \eqref{E1.101}, \eqref{E1.97}, \eqref{E2.48}, \eqref{E1.14}, \eqref{E1.99}, $\|u\|_{s_1}\leq1$ and the definitions of $P_{N,j_0}, H_{N,j_0}$ that
\begin{align*}
|\mathcal{A}_{N,j_{0}}(\epsilon,\lambda,u,\theta)-\mathrm{Diag}(\mathcal{A}_{N,j_{0}}(\epsilon,\lambda,u,\theta))|_{s_1-\varrho}
\leq C(s_1)\|P_{N,j_0}(\bar{V}+(\mathrm{D}F)(u))_{|H_{N,j_0}}\|_{s_1}\leq {C'(s_1)}.
\end{align*}
Using the above fact and Proposition \ref{pro1} for $N':=N,\mathcal{A}:=\mathcal{A}_{N,j_{0}}(\epsilon,\lambda,u,\theta)$, if $\|\mathcal{A}^{-1}_{N,j_{0}}(\lambda,\epsilon,\theta)\|_0\leq N^{\tau}$, then one has
\begin{align*}
|\mathcal{A}^{-1}_{N,j_{0}}(\epsilon,\lambda,u,\theta)|_{s}\leq N^{\tau_2+\delta s},\quad\forall s\in[s_0,s_1-\varrho].
\end{align*}
This implies that
\begin{align*}
\forall\theta\notin\mathscr{B}_{N}^{0}(j_0)\Rightarrow \mathcal{A}_{N,j_{0}}(\epsilon,\lambda,u,\theta)~\text{is~N-good},
\end{align*}
which leads to $\mathscr{B}_{N}(j_0)\subset\mathscr{B}_{N}^{0}(j_0)$.
\begin{lemm}
Let $\hat{\mu}_{k}(M_1),\hat{\mu}_{k}(M_2)$ be eigenvalues of $M_1,M_2$ respectively, where $M_p,p=1,2$ are self-adjoint matrices of the same dimension, then $\hat{\mu}_{k}(M_p),p=1,2$ are ranked in nondecreasing order with
\begin{align}\label{E3.1}
|\hat{\mu}_{k}(M_1)-\hat{\mu}_{k}(M_2)|\leq\|M_1-M_2\|_0.
\end{align}
\end{lemm}
\begin{lemm}\label{lemma12}
Let $M(\eta)$ be a family of self-adjoint matrices in $\mathcal{M}_{\mathfrak{E}}^{\mathfrak{E}}$ with
$C^{1}$  depending on the parameter $\eta\in\mathscr{H}\subset\mathbf{R}$. Assume $\partial_{\eta}M(\eta)\geq \mathfrak{b}\mathrm{I}
$ for some $\mathfrak{b}>0$, then, for all $\mathfrak{a}>0$,
the Lebesgue measure
\begin{align*}
\mathrm{meas}\left(\{\eta\in\mathscr{H}:\|M^{-1}(\eta)\|_{0}\geq\mathfrak{a}^{-1}\}\right)\leq2\sharp\mathfrak{E}\mathfrak{a}\mathfrak{b}^{-1},
\end{align*}
where $\sharp\mathfrak{E}$ denotes the cardinality of the set $\mathfrak{E}$. Furthermore one has
\begin{align*}
\left\{\eta\in\mathscr{H}:\|M^{-1}(\eta)\|_{0}\geq\mathfrak{a}^{-1}\right\}\subset\bigcup_{1\leq q\leq\sharp\mathfrak{E}}I_{q},\quad\text{with}\quad\mathrm{meas}(I_{q})\leq2\mathfrak{ab}^{-1}.
\end{align*}
$\|M^{-1}(\eta)\|_0:=\infty$ if $M(\eta)$ is not invertible.
\end{lemm}
\begin{proof}
The proof is given by Lemma 5.1 of \cite{Berti2012nonlinearity}.
\end{proof}
Letting $\mathfrak{v}\geq1$ and
\begin{align}\label{E2.52}
N\geq2\|\rho\|,
\end{align}
 by \eqref{E1.50} and the definition of $\lambda_j$,  simple calculation yields
\begin{align}
&\lambda_{j}^{2}>\mathfrak{v}^2(\mathfrak{v}-1)^{2}N^{4}\quad\text{if}\quad|j|>{\mathfrak{v}b^{-1}_1 N}{},\label{E3.4}\\
&\lambda_{j}^{2}\leq4\mathfrak{v}^2(\mathfrak{v}+1)^{2}(b_2/b_1)^4N^{4}\quad\text{if}\quad|j|\leq{\mathfrak{v}b^{-1}_1 N},\label{E3.5}
\end{align}
where $b_1,b_2$ are given in \eqref{E1.50}. In addition, we  assume that
\begin{align}\label{E3.6}
N\geq\bar{N}(V,\nu,d,\rho)>0~\text{large~enough, \quad and }\epsilon\kappa^{-1}_{0}(\|\mathcal{T}''\|_0+\|\partial_{\lambda}\mathcal{T}''\|_0)\leq c
\end{align}
for some constant $c>0$.
\begin{lemm}\label{lemma16}
$\forall j_{0}\in\Gamma_{+}(\boldsymbol M)$, with $|j_{0}|>\frac{b_1+3}{b_{1}}N$, $\forall\lambda\in\Lambda$, we have
\begin{align*}
\mathscr{B}_{N}^{0}(j_0)\subset\bigcup_{q=1}^{N^{\nu+d+2}}I_{q}, ~\text{where}~I_{q}=I_q(j_0)~\text{~are~intervals~with}~\mathrm{meas}(I_{q})\leq
 N^{-\tau}.
 \end{align*}
\end{lemm}
\begin{proof}
It follows from \eqref{E3.1} that all the eigenvalues $\hat{\mu}_{l,j,p}(\theta),p=1,\cdots,\mathfrak{d}_j$ of  $\mathcal{A}_{N,j_{0}}(\lambda,\epsilon,\theta)$ satisfy
\begin{align*}
\hat {\mu}_{l,j,p}(\theta)=-(\lambda\omega_0\cdot l+\theta)^{2}+\lambda_{j}^{2}+O(\|V\|_0+\epsilon\|\mathcal{T}''\|_0).
\end{align*}
The fact $|j-j_0|\leq N$ gives that $|j|>\frac{3N}{b_1}$, which leads to $\lambda_{j}^{2}>36N^{4}$ owing to \eqref{E3.4}. Combining this with $|\lambda\omega_0|\leq\frac{3}{2}$ and $|l|\leq N$ yields
\begin{align*}
-(\lambda\omega_0\cdot l+\theta)^{2}+\lambda_{j}^{2}>-(\frac{3}{2}+|\theta|)^{2}+36N^{4}>15N^{2},\quad\forall|\theta|\leq3N.
 \end{align*}
Therefore, by means of \eqref{E3.6} and \eqref{E1.99}, we deduce $\hat{\mu}_{l,j,p}(\theta)\geq N^{2}$, which implies
\begin{align*}
\mathscr{B}_{N}^{0}(j_0)\cap[-3N,3N]=\emptyset.
\end{align*}
Based on above fact, denote by
\begin{align*}
\mathscr{B}_{N}^{0,+}(j_0)=\mathscr{B}_{N}^{0}(j_0)\cap(3N,+\infty),\quad \mathscr{B}_{N}^{0,-}(j_0)=\mathscr{B}_{N}^{0}(j_0)\cap(-\infty,-3N).
\end{align*}
We restrict our attention for $\theta>3N$. Define
\begin{align*}
\mathfrak{E}:=\left\{{(l,j,p)}\in\mathfrak{N}\times\mathbf{N}:|l|\leq N, |j-j_0|\leq N, p\leq \mathfrak{d}_j\right\}.
\end{align*}
It follows from the inequality $\mathfrak{d}_j\leq \|j+\rho\|^{d-r}$ that $\sharp\mathfrak{E}\leq \mathcal{C}N^{\nu+d}$. In addition,
\begin{align*}
\partial_{\theta}(-\mathcal{A}_{N,j_{0}}(\varepsilon,\lambda,\theta))=\mathrm{diag}_{|l|\leq N,|j-j_{0}|\leq N}2(\lambda\omega_0\cdot l+\theta)\mathrm{I}_{\mathfrak{d}_{j}}\geq (6N-3N)\mathrm{I}=3N\mathrm{I}.
 \end{align*}
Applying Lemma \ref{lemma12},  for $\mathfrak{a}=N^{-\tau},\mathfrak{b}=3N$ and $\sharp\mathfrak{E}\leq \mathcal{C}N^{\nu+d}$, we have
\begin{align*}
\mathscr{B}_{N}^{0,+}(j_0)\subset\bigcup_{q=1}^{N^{\nu+d+1}}I_{q},
\end{align*}
where the intervals $I_{q}=I_q(j_0)$  satisfy  $\mathrm{meas}(I_{q})\leq2N^{-\tau}(3N)^{-1}\leq
N^{-\tau}$.

The proof on $\mathscr{B}_{N}^{0,-}$ can apply  the similar step as above. For the sake of convenience, we omit the process.

Therefore, we have
\begin{align*}
\mathscr{B}_{N}^{0}(j_0)\subset\bigcup_{q=1}^{N^{\nu+d+2}}I_{q},~\text{where}~I_{q}=I_q(j_0)~\text{are~intervals~with}~\mathrm{meas}(I_{q})\leq
 N^{-\tau}.
\end{align*}
\end{proof}
Now consider the case $j_0\leq\frac{b_1+3}{b_{1}}N$. We have to study the measure of the set
\begin{align*}
\mathscr{B}_{2,N}^{0}(j_0):=\mathscr{B}_{2,N}^{0}(j_{0};\epsilon,\lambda,u):=\left\{\theta\in\mathbf{R}:\|\mathcal{A}_{N,j_{0}}^{-1}(\lambda,\epsilon,u,\theta)\|_{0}>{N^{\tau}}/{2}\right\}.
 \end{align*}
With the help of  the upper bound of $\mathrm{meas}(\mathscr{B}_{2,N}^{0}(j_0))$, a complexity estimate for $\mathscr{B}_{N}^{0}(j_0)$ can be  obtained.
\begin{lemm}\label{lemma13}
$\forall j_{0}\in\Gamma_{+}(\boldsymbol M)$, with $|j_{0}|\leq\frac{b_1+3}{b_{1}}N$, $\forall\lambda\in\Lambda$, we have
\begin{align*}
\mathscr{B}_{2,N}^{0}(j_0)\subset[-\mathfrak{s}N^2,\mathfrak{s}N^2]\quad \text{with}\quad\mathfrak{s}=2((2b_1+4)^2+1)(b_2/b_1)^2.
\end{align*}
\end{lemm}
\begin{proof}
If $|\theta|>\mathfrak{s}N^2$, then one has
\begin{align*}
|\lambda\omega_0\cdot l+\theta|\geq|\theta|-|\lambda\omega_0\cdot l|>\mathfrak{s}N^2-\frac{3}{2}N>2(2b_1+4)^2(b_2/b_1)^2N^2.
\end{align*}
Since $|j-j_0|\leq N\Rightarrow |j|\leq\frac{2b_1+3}{b_1}N$, by \eqref{E3.5}, it gives
\begin{align*}
\lambda_{j}^{2}\leq(2(2b_1+3)(2b_1+4))^2(b_2/b_1)^4N^{4}.
\end{align*}
Then, due to \eqref{E1.99} and \eqref{E3.6} , all the eigenvalues $\hat{\mu}_{l,j,p}(\theta),p=1,\cdots,\mathfrak{d}_j$ of  $\mathcal{A}_{N,j_{0}}(\epsilon,\lambda,u,\theta)$ satisfy that, for all $|\theta|>\mathfrak{s}N^2$,
\begin{align*}
\hat{\mu}_{l,j,p}(\theta)=-(\lambda\omega_0\cdot l+\theta)^{2}+\lambda_{j}^{2}+O(\|V\|_0+\epsilon\|\mathcal{T}''\|_0)\leq-(b_2/b_1)^4N^{4},
\end{align*}
which implies the conclution of the lemma.
\end{proof}
\begin{lemm}\label{lemma15}
Let $\mathfrak{h}=\mathrm{meas}(\mathscr{B}_{2,N}^{0}(j_0))$. $\forall j_{0}\in\Gamma_{+}(\boldsymbol M)$, with $|j_{0}|\leq\frac{b_1+3}{b_{1}}N$, $\forall\lambda\in\Lambda$, the following holds
\begin{align*}
\mathscr{B}_{N}^{0}(j_0)\subset\bigcup_{q=1}^{\mathfrak{C}\mathfrak{h}N^{\tau+1}}I_{q},
\end{align*}
where $I_{q}=I_q(j_0)$ are intervals with $\mathrm{meas}(I_{q})\leq N^{-\tau}.$
\end{lemm}
\begin{proof}
For brevity, we write $\mathcal{A}_{N,j_{0}}(\theta):=\mathcal{A}_{N,j_{0}}(\epsilon,\lambda,u,\theta)$. If $\theta\in\mathscr{B}_{N}^{0}(j_0)$, then there exists an eigenvalue $\hat{\mu}_{l_1,j_1,p_1}(\mathcal{A}_{N,j_{0}}(\theta))$ of $\mathcal{A}_{N,j_{0}}(\theta)$ with
\begin{align*}
|\hat{\mu}_{l_1,j_1,p_1}(\mathcal{A}_{N,j_{0}}(\theta))|\leq N^{-\tau}.
\end{align*}
Consequently, we have that, for $|\Delta \theta|\leq 1$,
\begin{align*}
\|\mathcal{A}_{N,j_{0}}(\theta+\Delta \theta)-\mathcal{A}_{N,j_{0}}(\theta)\|_0
=&\|\mathrm{Diag}_{|l|\leq N,|j-j_0|\leq N}((\lambda\omega_0\cdot l+\theta)^{2}-(\lambda\omega_0\cdot l+\theta+\Delta\theta)^{2})\mathrm{I}_{\mathfrak{d}_j}\|_0\\
=&|(2\lambda\omega_{0}+2\theta+\Delta\theta)(\Delta\theta)|
\leq(4N+2|\theta|+1)|\Delta\theta|.
\end{align*}
If $(4N+2|\theta|+1)|\Delta\theta|\leq N^{-\tau}$, then it follows from \eqref{E3.1} that
\begin{align*}
|\hat{\mu}_{l_1,j_1,p_1}(\mathcal{A}_{N,j_{0}}(\theta+\Delta\theta))-\hat{\mu}_{l_1,j_1,p_1}(\mathcal{A}_{N,j_{0}}(\theta))|\leq N^{-\tau},
\end{align*}
which gives rise to
\begin{align*}
\hat{\mu}_{l_1,j_1,p_1}(\mathcal{A}_{N,j_{0}}(\theta+\Delta\theta))\leq2N^{-\tau}.
\end{align*}
Hence $\theta+\Delta\theta\in \mathscr{B}_{2,N}^{0}(j_0)$, that is
\begin{align*}
(4N+2|\theta|+1)|\Delta\theta|\leq N^{-\tau}\Rightarrow\theta+\Delta\theta\in \mathscr{B}_{2,N}^{0}(j_0).
\end{align*}
By Lemma \ref{lemma13}, we have to guarantee
\begin{align*}
(4N+2\mathfrak{s}N+1)|\Delta\theta|)\leq N^{-\tau}\quad \text{with}\quad\mathfrak{s}=2((2b_1+4)^2+1)(b_2/b_1)^2.
\end{align*}
This indicates $|\Delta\theta|\leq c(\mathfrak{s}) N^{-(\tau+1)}$, which carries out
\begin{align*}
[\theta-c(\mathfrak{s})N^{-(\tau+1)},\theta+c(\mathfrak{s})N^{-(\tau+1)}]\subset \mathscr{B}_{2,N}^{0}(j_0).
\end{align*}
Therefore $\mathscr{B}_{N}^{0}(j_0)$ is included in an union of intervals $\mathscr{J}_{p}$ with disjoint interiors, namely
\begin{align*}
\mathscr{B}_{N}^{0}(j_0)\subset\bigcup_{p}\mathscr{J}_{p}\subset\mathscr{B}_{2,N}^{0}(j_0)\quad\text{with}\quad\mathrm{meas}(\mathscr{J}_p)\geq2c(\mathfrak{s})N^{-(\tau+1)}.
\end{align*}
Now we decompose each $\mathscr{J}_{p}$ as an union of non-overlapping intervals with
\begin{align}\label{E3.7}
\frac{1}{2}c(\mathfrak{s})N^{-(\tau+1)}\leq\mathrm{ meas}(I_q)\leq c(\mathfrak{s})N^{-(\tau+1)}.
\end{align}
Thus we get
\begin{align*}
\mathscr{B}_{N}^{0}(j_0)\subset\bigcup_{q=1,\cdots,q_0}I_{q}\subset\mathscr{B}_{2,N}^{0}(j_0),
\end{align*}
where $I_q$ satisfies \eqref{E3.7}. Since $I_q$ does not overlap, we deduce
\begin{align*}
\frac{1}{2}c(\mathfrak{s})N^{-(\tau+1)}q_0\leq\sum\limits_{q=1}^{q_0}\mathrm{meas}(I_q)\leq \mathrm{meas}(\mathscr{B}_{2,N}^{0}(j_0))=:\mathfrak{h},
\end{align*}
which gives $q_0\leq \mathfrak{C}\mathfrak{h}N^{\tau+1}$.
\end{proof}
Define
\begin{align*}
\tilde{\mathscr{B}}_{2,N}^{0}(j_0):=\tilde{\mathscr{B}}_{2,N}^{0}(j_{0};\epsilon,u):=\left\{(\lambda,\theta)\in
\Lambda\times\mathbf{R}:\|\mathcal{A}_{N,j_{0}}^{-1}(\epsilon,\lambda,u,\theta)\|_{0}>{N^{\tau}}/{2}\right\}.
\end{align*}
\begin{lemm}\label{lemma14}
$\forall j_{0}\in\Gamma_{+}(\boldsymbol M)$, with $|j_{0}|\leq\frac{b_1+3}{b_{1}}N$, $\forall\lambda\in\Lambda$, there exists some constant $\mathfrak{C}_0>0$ such that
\begin{align}\label{E3.8}
\mathrm{meas}(\tilde{\mathscr{B}}_{2,N}^{0}(j_0))\leq\mathfrak{C}_0N^{-\tau+\nu+d+2}.
\end{align}
\end{lemm}
\begin{proof}
Let
\begin{align*}
\xi=\frac{1}{\lambda^{2}},\quad\zeta=\frac{\theta}{\lambda},\quad(\xi,\zeta)\in[{4}/{9},4]\times[-2\mathfrak{s}N^2,2\mathfrak{s}N^2]
\end{align*}
with $\mathfrak{s}=2((2b_1+4)^2+1)(b_2/b_1)^2$. Thus we consider the following self-adjoint matrices
\begin{align*}
L(\xi,\zeta):=\lambda^{-2} \mathcal{A}_{N,j_{0}}(\epsilon,\lambda,u,\theta)
=&\mathrm{diag}_{|l|\leq N,|j-j_{0}|\leq N}\left((-({\omega}_0\cdot l+\zeta)^{2})\mathrm{I}_{\mathfrak{d}_{j}}\right)\\
&+\xi\check{P}_{N_0,j_0}(\Delta^2+V(\boldsymbol x))_{|\check{H}_{N_0,j_0}}-\epsilon\xi \mathcal{T}''(\epsilon,{1}/{\sqrt{\xi}}).
\end{align*}
Formula \eqref{E1.84} implies
$\check{P}_{N_0,j_0}(\Delta^2+V(\boldsymbol x))_{|\check{H}_{N_0,j_0}}\geq \kappa_0\mathrm{I}$, which leads to
\begin{align*}
\partial_{\xi}L(\xi,\zeta)&=\check{P}_{N_0,j_0}(\Delta^2+V(\boldsymbol x))_{|\check{H}_{N_0,j_0}}-\epsilon\mathcal{T}''-\frac{\epsilon}{2}\xi^{-1/2}
\partial_{\lambda}\mathcal{T}''\stackrel{\eqref{E3.6}}{\geq}{\kappa_0}/{2}.
\end{align*}
Combining this with Lemma \ref{lemma12} for $\mathfrak{a}=8N^{-\tau},\mathfrak{b}=\kappa_0/2,\sharp\mathfrak{E}\leq \mathcal{C}N^{\nu+d}$, for each $\zeta\in[-2\mathfrak{s}N^2,2\mathfrak{s}N^2]$, we obtain
\begin{align*}
\mathrm{meas}(\{\xi\in[{4}/{9},4]:\|L^{-1}(\xi,\zeta)\|_{0}>{N^{\tau}}/{8}\})\leq 2\mathcal{C}N^{\nu+d}8N^{-\tau}{2}/{\kappa_0}\leq \mathcal{C}'N^{-\tau+\nu+d}.
\end{align*}
Integrating in $\zeta\in[-2\mathfrak{s}N^2,2\mathfrak{s}N^2]$ yields $\|L^{-1}(\xi,\zeta)\|_{0}\leq N^{\tau}/{8}$
except for $\xi$ in a set of measure $O(N^{-\tau+\nu+d+2})$. In addition
\begin{align*}
\left|\det\left(\frac{\partial(\xi,\zeta)}{\partial(\lambda,\theta)}\right)\right|=\frac{1}{\lambda^4}\geq\frac{1}{6}.
\end{align*}
Combining this with $\mathcal{A}_{N,j_{0}}(\epsilon,\lambda,u,\theta)=\lambda^{2}L(\xi,\eta)$ yields
\begin{align*}
\|\mathcal{A}_{N,j_{0}}^{-1}(\lambda,\epsilon,u,\theta)\|_{0}\leq\lambda^{-2}\|L^{-1}(\xi,\zeta)\|_{0}\leq\lambda^{-2}N^{\tau/8}\leq N^{\tau}/2,
\end{align*}
for all $(\lambda,\theta)\in\Lambda\times\mathbf{R}$ except for $\lambda$ in a set of measure $O(N^{-\tau+\nu+d+2})$.
\end{proof}
Define
\begin{align}\label{E3.9}
\mathscr{U}_{N}(u):=\left\{\|\mathcal{A}^{-1}_{N}(\epsilon,\lambda,u)\|_0\leq N^{\tau}\right\},
\end{align}
where $\mathcal{A}^{-1}(\epsilon,\lambda,u)$, $\mathcal{A}^{-1}_{N}(\epsilon,\lambda,u)$ are defined in \eqref{E1.72}, \eqref{E2.51} respectively. A similar computation as the proof of  Lemma \ref{lemma14} verifies
\begin{align}\label{E3.10}
\mathrm{meas}(\Lambda\setminus\mathscr{U}_{N})\leq N^{-\tau+\nu+d+3}\leq N^{-1}
\end{align}
for $\tau\geq\nu+d+4$. Moreover define the following set
\begin{align}\label{E3.11}
\mathscr{C}_{N}(j_{0})=\left\{\lambda\in\Lambda: \mathrm{meas}({\mathscr{B}}_{2,N}^{0}(j_0))\geq{\mathfrak{C}^{-1}}N^{-\tau+\nu+d+r+3}\right\},
\end{align}
where $\mathfrak{C}$ is given in Lemma \ref{lemma15}.
\begin{lemm}\label{lemma17}
$\forall j_{0}\in\Gamma_{+}(\boldsymbol M)$, with $|j_{0}|\leq\frac{b_1+3}{b_{1}}N$, $\forall\lambda\in\Lambda$, one has $\mathrm{meas}(\mathscr{C}_{N}(j_{0}))=O(N^{-r-1})$.
\end{lemm}
\begin{proof}
Fubini Theorem yeilds that, for $\mathfrak{x}=\mathfrak{C}^{-1}N^{-\tau+\nu+d+r+3}$,
\begin{align*}
\mathrm{meas}(\tilde{\mathscr{B}}_{2,N}^{0}(j_0))=&\int_{\frac{1}{2}}^{\frac{3}{2}}\mathrm{d}\lambda~
\mathrm{meas}(\mathscr{B}_{2,N}^{0}(j_0))\\
\geq &\mathfrak{x}~\mathrm{meas}\{\lambda\in\Lambda:\mathrm{meas}(\mathscr{B}_{2,N}^{0}(j_0))\geq\mathfrak{x}\},
\end{align*}
which leads to
\begin{align*}
\mathrm{meas}\left\{\lambda\in\Lambda:\mathrm{meas}(\mathscr{B}_{2,N}^{0}(j_0))\geq\mathfrak{x}\right\}\leq\frac{1}{\mathfrak{x}}
\mathrm{meas}(\tilde{\mathscr{B}}_{2,N}^{0}(j_0))
\stackrel{\eqref{E3.8}}{\leq}\frac{1}{\mathfrak{x}}\mathfrak{C}_0N^{-\tau+\nu+d+2}
={\mathfrak{C}_0}{\mathfrak{C}}{N^{-r-1}}.
\end{align*}
\end{proof}
\begin{prop}\label{pro2}
Assume  \eqref{E3.6} holds. There exists some constant ${\mathfrak{C}}_1>0$ such that the set $\Lambda\setminus\mathcal {G}_{N}^{0}$ has measure
\begin{align}\label{E3.77}
\mathrm{meas}(\Lambda\setminus\mathscr{G}_{N}^{0}(u))\leq{\mathfrak{C}}_1N^{-1},
\end{align}
where $\Lambda,\mathscr{G}_{N}^{0}(u)$ are defined in \eqref{E1.8}, \eqref{E3.13} respectively.
\end{prop}
\begin{proof}
By definition \eqref{E3.11}, $\forall\lambda\notin\mathscr{C}_{N}(j_{0})$, we have
\begin{align*}
\mathrm{meas}(\mathscr{B}_{2,N}^{0}(j_0))\leq \mathfrak{C}^{-1}N^{-\tau+\nu+d+r+3}.
\end{align*}
Combining this with Lemma \ref{lemma15} deduces that, $\forall\lambda\notin\mathscr{C}_{N}(j_{0})$, $\forall j_{0}\in\Lambda_{+}(\boldsymbol M)$ with $|j_{0}|\leq\frac{b_1+3}{b_{1}}N$,
\begin{align}\label{E3.12}
\mathscr{B}_{N}^{0}(j_0)\subset\bigcup_{q=1}^{\mathfrak{{C}}\mathfrak{h}N^{\tau+1}}I_{q}
 \subset\bigcup_{q=1}^{N^{\nu+d+r+4}}I_{q}\quad\text{with}\quad\mathrm{meas}({I_q})\leq N^{-\tau}.
\end{align}
It follows from \eqref{E3.12} and Lemma \ref{lemma16} that, $\forall\lambda\notin\mathscr{C}_{N}(j_{0})$, $\forall j_{0}\in\Gamma_{+}(\boldsymbol M)$,
\begin{align*}
\mathscr{B}_{N}^{0}(j_0)\subset\bigcup_{q=1}^{N^{\nu+d+r+5}}I_{q},~\text{where}~I_{q}=I_q(j_0)~\text{~are~intervals~with}~\mathrm{meas}({I_q})\leq N^{-\tau}.
\end{align*}
Hence, for all $ j_{0}\in\Gamma_{+}(\boldsymbol M)$ with $|j_{0}|\leq\frac{b_1+3}{b_{1}}N$, we have
\begin{align*}
\Lambda\setminus\mathscr{G}_{N}^{0}(u)\subset\bigcup_{|j_{0}|\leq{(b_1+3)}{b^{-1}_{1}}N}\mathscr{C}_{N}(j_{0}).
\end{align*}
Applying Lemma \ref{lemma17} yields
\begin{align*}
\mathrm{ meas}(\Lambda\setminus\mathscr {G}_{N}^{0}(u))\leq\sum_{|j_{0}|\leq{(b_1+3)}{b^{-1}_{1}}N}\mathrm{meas}(\mathscr{C}_{N}(j_{0}))=O(N^{-1}).
\end{align*}
\end{proof}

\subsection{Nash-Moser iteration}\label{sec:3.4}
We first give the following iterative theorem. From now on, we fix
\begin{align}
&\delta:={1}/{4}, \quad\tau_1:=3\nu+d+1,\quad\chi_0:=3C_1+9, \label{E3.22}\\
&\tau:=\max\{\tau_1+3,2\chi_0\nu+1\}=\max\{3\nu+d+4,2\chi_0\nu+1\},\quad\tau_2:=3\tau+2(\nu+r)+(\nu+d),\label{E3.33}\\
&s_0:=\nu+d,\quad s_1:=12\chi_0(\tau+(\nu+r)+(\nu+d)),\quad s_2:=12\tau_2+8s_1+12,\label{E3.34}\\
&\text{and}\quad \sigma =\tau_2+3\delta s_1+3.\label{E3.42}
\end{align}
\begin{rema}\label{rema1}
Formulae \eqref{E3.22}-\eqref{E3.34} satisfy $\tau\geq\max\{\tau_1+1,\nu+d+4\}$ (see \eqref{E3.10}), assumptions \eqref{E2.44}-\eqref{E2.46} and $\mathrm{(2)}$ in Propositions \ref{pro1}  and  \ref{pro3} respectively. The choice of $\tau_1$ is seen in Lemma \ref{lemma18}. Moreover the choices of $\delta,s_1$ satisfy $\delta s_1\geq\varrho/2$, where $\varrho$ is given in Lemma \ref{lemma11}.
\end{rema}
Setting $\gamma>0$, we restrict $\lambda$ to the set
\begin{align}
\mathscr{U}:=&\left\{\lambda\in\Lambda:\|((-{(\lambda{\omega}_0\cdot l)}^2)\mathrm{I}_{\mathfrak{d}_j}+\check{P}_{N_0,0}(\Delta^2+V(\boldsymbol x))_{|\check{H}_{N_0,0}})^{-1}\|_{L^2_{x}}\leq \gamma^{-1}N^{\tau_1}_0,\quad\forall|l|\leq N_0\right\}\nonumber\\
=&\left\{\lambda\in\Lambda:|-(\lambda{\omega}_0\cdot l)^2+\hat{\lambda}_{j,p}|\geq\gamma N^{-\tau_1}_0,\quad\forall|l|\leq N_0,\forall|j|\leq N_0,1\leq p\leq\mathfrak{d}_j\right\},\label{E1.9}
\end{align}
where $\hat{\lambda}_{j,p},p=1,\cdots,\mathfrak{d}_j$ are eigenvalues of $\check{P}_{N_0,0}(\Delta^2+V(\boldsymbol x))_{|\check{H}_{N_0,0}}$.
\begin{theo}\label{theo1}
There exist $\bar{\mathfrak{c}},\bar{\gamma}$ (depending on $\nu,d,r,V,\gamma_0,\kappa_0$) such that if
\begin{align}\label{E3.15}
N_0\geq16\gamma^{-1},\quad \gamma\in(0,\bar{\gamma}),\quad \epsilon_0N_0^{s_2}\leq \bar{\mathfrak{c}},
\end{align}
then there is a sequence $(u_n)_{n\geq0}$ of $C^1$ maps $u_n:[0,\epsilon_0)\times\Lambda\rightarrow H^{s_1}$ satisfying
\\
$(F1)_{n\geq0}$ $u_n\in H_{N_n}$, $u_{n}(0,\lambda)=0$, $\|u_n\|_{s_1}\leq1$, $\|\partial_{\lambda}u_n\|_{s_1}\leq \bar{C}N_{0}^{\tau_1+s_1+1}\gamma^{-1}$.
\\
$(F2)_{n\geq1}$ $\|u_k-u_{k-1}\|_{s_1}\leq N^{-\sigma-1}_{k}$, $\|\partial_{\lambda}(u_k-u_{k-1})\|_{s_1}\leq N_{k}^{-\frac{1}{2}}$, $\forall 1\leq k\leq n$,.
\\
$(F3)_{n\geq1}$ one has
\begin{align*}
\|u-u_{n-1}\|_{s_1}\leq N_{n}^{-\sigma}\Rightarrow\bigcap_{k=1}^{n}\mathscr{G}_{N_{k}}^{0}(u_{k-1})\subset\mathscr{G}_{N_{n}}(u),
\end{align*}
where $\mathscr{G}^0_{N}(u),\mathscr{G}_{N}(u)$ are defined in \eqref{E3.13}, \eqref{E1.81} respectively.
\\
$(F4)_{n\geq0}$ Set
\begin{align}\label{E3.17}
\mathscr{D}_{n}:=\bigcap_{k=1}^{n}\mathscr{U}_{N_{k}}(u_{k-1})\cap\bigcap_{k=1}^{n}\mathscr{G}_{N_{k}}^{0}(u_{k-1})\cap\mathscr{U},
\end{align}
where $\mathscr{U}_{N}(u),\mathscr{U}$ are defined in \eqref{E3.9}, \eqref{E1.9} respectively. If $\lambda\in \mathcal{B}(\mathscr{D}_{n},N^{-\sigma}_n)$, then $u_{n}(\epsilon,\lambda)$ solves the equation
\begin{equation*}\label{E3.21}
P_{N_n}(L_{\lambda}u-\epsilon F(u))=0.\tag{$P_{N_n}$}
\end{equation*}
\\
$(F5)_{n\geq0}$ Letting $\mathscr{A}_{n}:=\|u_n\|_{s_2}$ and $\mathscr{A}'_{n}:=\|\partial_{\lambda}u_n\|_{s_2}$,  we have
\begin{align*}
(1)~\mathscr{A}_{n}\leq N^{2\tau_2+2\delta s_1+2}_n, \quad\quad\quad\quad (2)~\mathscr{A}'_{n}\leq N_{n}^{4\tau_2+2s_1+6}.
\end{align*}
The sequence $(u_n)_{n\geq0}$ converges in ${s_1}$-norm to a map
\begin{align*}
u(\epsilon,\cdot)\in C^{1}(\Lambda;H^{s_1}) \quad\text{with}\quad u(0,\lambda)=0.
\end{align*}
Moreover
\begin{align}\label{E3.20}
\mathscr{D}_{\epsilon}:=\bigcap_{n\geq0}\mathscr{D}_{n}
\end{align}
is the Cantor-like set, and for all $\lambda\in\mathscr{D}_{\epsilon}$, $u(\epsilon,\lambda)$ is a solution of \eqref{E1.2} with $\omega=\lambda\omega_0$.
\end{theo}
\begin{rema}
The case $(F2)_{n=0}$ for $u_{-1}:=0$ is seen as $(F1)_{0}$.
\end{rema}
Step1: Initialization. Let us check that $(F1)_{0},(F4)_{0},(F5)_{0}$ hold.
In the first step of iteration, the equation $(P_{N_0})$ is written as the following form
\begin{align}\label{E1.12}
\mathfrak{L}_{N_0}u=\epsilon P_{N_0}F(u)\quad\text{with}\quad\mathfrak{L}_{N_0}:=P_{N_0}(L_\lambda)_{|{H}_{N_0}},\quad \forall u\in{H}_{N_0},
\end{align}
where $L_\lambda$ is given by \eqref{E1.75}.
\begin{lemm}\label{lemma4}
For all $\lambda\in\mathcal{B}(\mathscr{U},2N^{-\sigma}_0)$, the operator $\mathfrak{L}_{N_0}$ is invertible with
\begin{align}\label{E1.13}
\|\mathfrak{L}^{-1}_{N_0}\|_{s_1}\leq2c^{s_1}_2\gamma^{-1}N^{\tau_1+s_1}_0.
\end{align}
\end{lemm}
\begin{proof}
For all $\lambda\in\mathcal{B}(\mathscr{U},2N^{-\sigma}_0)$, there exists $\lambda_1\in\mathscr{U}$ such that $|\lambda_1-\lambda|\leq2N^{-\sigma}_0$. If $N_0\geq16\gamma^{-1}\geq2$, then it follows \eqref{E1.8} and \eqref{E1.9} that, for all $\sigma\geq\tau_1+3$,
\begin{align}
|-(\lambda{\omega}_0\cdot l)^2+\hat{\lambda}_{j,k}|\geq&\gamma N^{-\tau_1}_0-|(\lambda_1{\omega}_0\cdot l)^2-(\lambda{\omega}_0\cdot l)^2|\nonumber\\
\geq&\gamma N^{-\tau_1}_0-|\lambda_1-\lambda||\lambda_1+\lambda||{\omega}_0|^2|l|^2
\geq\gamma N^{-\tau_1}_0-2N^{-\sigma}_04|{\omega}_0|^2N^2_0\nonumber\\
\geq&\frac{\gamma}{2}N_0^{-\tau_1}.\label{E1.90}
\end{align}
This gives that
$\|\mathfrak{L}^{-1}_{N_0}\|_0\leq2\gamma^{-1}N^{\tau_1}_0$.
Combining this with formula \eqref{E1.10}, we get the conclusion of the lemma.
\end{proof}
%
Then solving equation \eqref{E1.12} is reduced to the fixed point problem $u=U_0(u)$, where
\begin{align}\label{E1.17}
U_0:{H}_{N_0}\rightarrow {H}_{N_0},\quad u\mapsto\epsilon\mathfrak{L}^{-1}_{N_0}P_{N_0}F(u).
\end{align}
\begin{lemm}\label{lemma1}
If $\epsilon\gamma^{-1}N_0^{\tau_1+s_1+\sigma}\leq \mathfrak{c}_0(s_1)$ is small enough, $\forall\lambda\in\mathcal{B}(\mathscr{U},2N^{-\sigma}_0)$, the map $U_0$ is a contraction in $\mathcal{B}(0,\rho_0):=\left\{u\in H_{N_0}:\|u\|_{s_1}\leq\rho_0:=N_0^{-\sigma}\right\}$.
\end{lemm}
\begin{proof}
If $\epsilon\gamma^{-1}N_0^{\tau_1+s_1+\sigma}\leq \mathfrak{c}_0(s_1)$ is small enough, owing to \eqref{E1.14}, \eqref{E1.13} and $\|u\|_{s_1}\leq1$, then it shows
\begin{align*}
\|U_0(u)\|_{s_1}\leq&2\epsilon c^{s_1}_2\gamma^{-1}N_0^{\tau_1+s_1}\|F(u)\|_{s_1}\leq2\epsilon c^{s_1}_2\gamma^{-1}N_0^{\tau_1+s_1}C(s_1)(1+\|u\|_{s_1})\leq N^{-\sigma}_0.
\end{align*}
In addition, by \eqref{E1.14}, \eqref{E1.13} and $\|u\|_{s_1}\leq1$, we have that, for $\epsilon\gamma^{-1}N_0^{\tau_1+s_1+\sigma}\leq \mathfrak{c}_0(s_1)$ small enough,
\begin{align}\label{E1.18}
\|(\mathrm{D}U_0)(u)\|_{s_1}&=\epsilon\|\mathfrak{L}^{-1}_{N_0}P_{N_0}(\mathrm{D}F)(u)_{|H_{N_0}}\|_{s_1}
\leq2\epsilon c^{s_1}_2\gamma^{-1}N_0^{\tau_1+s_1}\|P_{N_0}(\mathrm{D}{F})(u)_{|H_{N_0}}\|_{s_1}\nonumber\\
&\leq2\epsilon c^{s_1}_2\gamma^{-1}N_0^{\tau_1+s_1}C(s_1)(1+\|u\|_{s_1})\leq1/2.
\end{align}
Thus the map $U_0$ is a contraction in $\mathcal{B}(0,\rho_0)$.
\end{proof}
Denote by $\tilde{u}_0$ the unique solution of equation \eqref{E1.12} in $\mathcal{B}(0,\rho_0)$. The map $U_0$ (see \eqref{E1.17}) has $u=0$ as a fixed point for $\epsilon=0$ . By uniqueness we deduce $\tilde{u}_0(0,\lambda)=0$. The implicit function theorem implies that $\tilde{u}_0(\epsilon,\cdot)\in C^1(\mathcal{B}(\mathscr{U},2N_0^{-\sigma});H_{N_0})$ with
\begin{align}\label{E1.87}
\partial_{\lambda}\tilde{u}_0=-\mathfrak{L}^{-1}_{N_0}(\epsilon,\lambda,\tilde{u}_0)(\partial_\lambda\mathfrak{L}_{N_0})\tilde{u}_0,
\end{align}
where
\begin{align}\label{E1.88}
\mathfrak{L}_{N_0}(\epsilon,\lambda,\tilde{u}_0):=\mathfrak{L}_{N_0}-\epsilon P_{N_0}((\mathrm{D}F)(\tilde{u}_0))_{|H_{N_0}},
\quad\partial_\lambda\mathfrak{L}_{N_0}=P_{N_0}(2\lambda({\omega}_0\cdot\partial_{\varphi})^2)_{|{H}_{N_0}}.
\end{align}
Formulae \eqref{E1.18} and \eqref{E1.13} give that $\mathfrak{L}_{N_0}(\epsilon,\lambda,\tilde{u}_0)$ is invertible with
\begin{align*}
\|\mathfrak{L}^{-1}_{N_0}(\epsilon,\lambda,\tilde{u}_0)\|_{s_1}=\|(\mathrm{I}-\epsilon \mathfrak{L}^{-1}_{N_0}P_{N_0}((\mathrm{D}F)(\tilde{u}_0))_{|H_{N_0}})^{-1}\mathcal{L}^{-1}_{N_0}\|_{s_1}\leq2\|\mathfrak{L}^{-1}_{N_0}\|_{s_1}
\leq4c^{s_1}_2\gamma^{-1}N^{\tau_1+s_1}_0.
\end{align*}
Consequently, applying $\|\tilde{u}_0\|_{s_1+2}\stackrel{\eqref{E1.10}}{\leq} c^2_2N^2_0\|\tilde{u}_0\|_{s_1}\leq c^2_2N^{2-\sigma}_0$, we derive that, for all $\sigma\geq2$,
\begin{align}\label{E1.20}
\|\partial_{\lambda}\tilde{u}_0\|_{s_1}\leq&4c^{s_1}_2\gamma^{-1}N^{\tau_1+s_1}_0\|(\partial_\lambda\mathfrak{L}_{N_0})\tilde{u}_0\|_{s_1}
\leq12c^{s_1}_2|{\omega}_0|^2\gamma^{-1}N^{\tau_1+s_1}_0\|\tilde{u}_0\|_{s_1+2}\nonumber\\
\leq&12c^{s_1+2}_2\gamma^{-1}N^{\tau_1+s_1+2-\sigma}_0\leq12c^{s_1+2}_2\gamma^{-1}N^{\tau_1+s_1}_0.
\end{align}
Define a $C^{\infty}$ cut-off function $\psi_0:\Lambda\rightarrow[0,1]$ as
\begin{equation}\label{E1.19}
\psi_0:=
\begin{cases}
1\quad\quad\text{if}\quad\lambda\in\mathcal{B}(\mathscr{U},N_0^{-\sigma}),\\
0\quad\quad\text{if}\quad\lambda\in\mathcal{B}(\mathscr{U},2N_0^{-\sigma})
\end{cases}
\quad\text{with}\quad
|\partial_{\lambda}\psi_0|\leq12c^{s_1+2}_2N^{\sigma}_0.
\end{equation}
Then ${u}_0:=\psi_0\tilde{u}_0\in C^1(\Lambda;H_{N_0})$. It follows from $\|\tilde{u}_0\|_{s_1}\leq\rho_0$ (see Lemma \ref{lemma1}), \eqref{E1.20}-\eqref{E1.19} that, for all $\sigma\geq2$,
\begin{align}
\|{u}_0\|_{s_1}&\leq|\psi_0|\|\tilde{u}_0\|_{s_1}\leq N^{-\sigma}_0,\label{E104}\\
\|\partial_{\lambda}{u}_0\|_{s_1}&\leq|\partial_{\lambda}\psi_0|\|\tilde{u}_0\|_{s_1}+|\psi_0|\|\partial_{\lambda}\tilde{u}_0\|_{s_1}\leq
12c^{s_1+2}_2N^{\sigma}_0N^{-\sigma}_0+12c^{s_1+2}_2\gamma^{-1}N^{\tau_1+s_1}_0\nonumber\\
&\leq12c^{s_1+2}_2\gamma^{-1}N^{\tau_1+s_1+1}_0\label{E105}.
\end{align}
%
Formula \eqref{E104}-\eqref{E105} show that $(F1)_{0}$ is satisfied.  Next, we verify the property $(F4)_{0}$. Since $\mathscr{D}_{0}=\mathscr{U}$, by the definition of $\psi_0$ (see \eqref{E1.19}), the $u_0(\epsilon,\lambda)$ solves the equation $(P_0)$ for all $\lambda\in \mathcal{B}(\mathscr{D}_{0},2N^{-\sigma}_0)$.

Let us show that the upper bounds of $u_0$, $\partial_{\lambda}u_0$ on ${s_2}$-norm. Applying the equality $u_0=U_0(u_0)$, \eqref{E1.10}, \eqref{E1.17}, \eqref{E1.13}, \eqref{E1.14}, \eqref{E3.15}, \eqref{E3.33} and the inequality $\|u_0\|_{s_1}\leq1$, we derive
\begin{align}\label{E1.91}
\|u_0\|_{s_2}=&\|U_0(u_0)\|_{s_2}\leq c_2^{s_2-s_1}N_0^{s_2-s_1}\|U_0(u_0)\|_{s_1}\leq\epsilon c_2^{s_2}N_0^{s_2-s_1}2\gamma^{-1}N_0^{\tau_1+s_1}C(s_1)(1+\|u_0\|_{s_1})\nonumber\\
\leq& N_0^{\tau_1+2}\leq N_0^{2\tau_2+2\delta s_1+2}.
\end{align}
For all $\lambda\in\mathcal{B}(\mathscr{D}_{0},2N^{-\sigma}_0)$, if $\tau\geq\tau_1+1$ ($N_0\geq2\gamma^{-1}$), then \eqref{E1.90} infers
\begin{align*}
|-(\lambda{\omega}_0\cdot l)^2+\hat{\lambda}_{j,p}|\geq\frac{\gamma}{2}N_0^{-\tau_1}\geq N_{0}^{-\tau},
\end{align*}
which indicates that, for $N=N_0,\theta=0,j_0=0$,
\begin{align}\label{E2.0}
|\mathfrak{L}^{-1}_{N_0}|_s=|(P_{N_0}(L_\lambda)_{|{H}_{N_0}})^{-1}|_{s}\stackrel{\eqref{E1.83},\eqref{E1.77}}{\leq}\frac{1}{2}N^{\tau_2+\delta s}_0,\quad\forall s\in[s_0,s_2].
\end{align}
Consequently, for all $s_1\geq s_0+\varrho$ and $\|u_0\|_{s_1}\leq1$, we verify
\begin{align*}
|\mathfrak{L}_{N_0}^{-1}|_{s_0}|\epsilon P_{N_0}((\mathrm{D}F)(u_0))_{|H_{N_0}} |_{s_0}\stackrel{\eqref{E3.15}}{\leq}\frac{\epsilon}{2}N^{\tau_2+\delta s_0}_0C(s_0)(1+\|u_0\|_{s_0+\varrho})\leq\frac{1}{2}.
\end{align*}
Hence, for all $s_1\geq s_0+\varrho$ and $\varrho/2\leq\delta s_1$, it follows from Lemma \ref{lemma2}, \eqref{E1.22}, \eqref{E1.10}, \eqref{E2.0}, \eqref{E1.14}, \eqref{E3.34}, \eqref{E1.91}, \eqref{E3.15}, that, for all $s\in[s_1,s_2]$,
\begin{align*}
|\mathfrak{L}^{-1}_{N_0}(\epsilon,\lambda,u_0)|_{s}\stackrel{\eqref{E1.92}}{\leq}&C(s)(|\mathfrak{L}_{N_0}^{-1}|_{s}+|\mathfrak{L}_{N_0}^{-1}|^2_{s_0}|\epsilon P_{N_0}((\mathrm{D}F)(u_0))_{|H_{N_0}} |_{s})\\
{\leq}& C(s)(|\mathfrak{L}_{N_0}^{-1}|_{s}+|\mathfrak{L}_{N_0}^{-1}|^2_{s_0}\|\epsilon P_{N_0}((\mathrm{D}F)(u_0))_{|H_{N_0}} \|_{s+\varrho})\\
{\leq}& C(s)(\frac{1}{2}N^{\tau_2+\delta s}_0+\epsilon\frac{1}{4} N^{2(\tau_2+\delta s_0)}_0c^{\varrho}_2N^{\varrho}_0\|P_{N_0}((\mathrm{D}F)(u_0))_{|H_{N_0}} \|_{s})\\
{\leq}&C(s)(\frac{1}{2}N_0^{\tau_2+\delta s}+\epsilon \frac{1}{4}N_0^{2(\tau_2+\delta s_0)}c^{\varrho}_2N^{2\delta s_1}_0C(s_2)(1+\|u_0\|_{s_2}))\\
{\leq}&N_0^{\tau_2+\delta s}.
\end{align*}
Combining this with \eqref{E1.87}, \eqref{E1.25}, \eqref{E1.88}, \eqref{E1.10}, \eqref{E1.91}, \eqref{E3.34}, $\delta=\frac{1}{4}$ and $\|u_0\|_{s_1}\leq1$ yields
\begin{align}\label{E1.105}
\|\partial_{\lambda}u_0\|_{s_2}&\leq C(s_2)\left(|\mathfrak{L}^{-1}_{N_0}(\epsilon,\lambda,u_0)|_{s_1}\|(\partial_\lambda\mathfrak{L}_{N_0})u_0\|_{s_2}
+|\mathfrak{L}^{-1}_{N_0}(\epsilon,\lambda,u_0)|_{s_2}\|(\partial_\lambda\mathfrak{L}_{N_0})u_0\|_{s_1}\right)\nonumber\\
&\leq 3C(s_2)c^2_2\left(N_0^{\tau_2+\delta s_1}N^2_0\|u_0\|_{s_2}+N_{0}^{\tau_2+\delta s_2}N^2_0\|u_0\|_{s_1}\right)\nonumber\\
&\leq N_{0}^{4\tau_2+2s_1+6}.
\end{align}
Formulae \eqref{E1.91} and \eqref{E1.105} give that $(F5)_{0}$ holds.

Step 2: assumption. Assume that we have get a solution $u_{n}\in C^1(\Lambda;H_{N_{n}})$ of \eqref{E3.21} and that properties  $(F1)_{k}$-$(F5)_{k}$ hold for all $k\leq n$.

Step 3: iteration. Our goal is to find a solution $u_{n+1}\in C^1(\Lambda;H_{N_{n+1}})$ of $(P_{N_{n+1}})$ and to prove the statements $(F1)_{n+1}$-$(F5)_{n+1}$. Denote by
\begin{equation*}
\tilde{u}_{n+1}=u_{n}+h\quad\text{with}\quad h\in H_{N_{n+1}}
\end{equation*}
a solution of $(P_{N_{n+1}})$. In addition it follows from \eqref{E3.41}, $\sigma\geq2$ (see \eqref{E3.42}) and $N_0\geq2$ that
\begin{align}\label{E3.68}
\mathcal{B}(\mathscr{D}_{n+1},2N^{-\sigma}_{n+1})\subset\mathcal{B}(\mathscr{D}_{n},2N^{-\sigma}_{n}),
\end{align}
which derives $P_{N_{n}}(L_{\lambda}u_{n}-\epsilon F(u_{n}))=0$ by $(F4)_{n}$. This implies
\begin{align}
P_{N_{n+1}}(L_{\lambda}(u_{n}+h)-\epsilon {F}(u_{n}+h))=&P_{N_{n+1}}(L_{\lambda}u_{n}-\epsilon {F}(u_{n}))+P_{N_{n+1}}(L_{\lambda}h-\epsilon ({F}(u_{n}+h)-{F}(u_{n})))\nonumber\\
=&\mathfrak{L}_{N_{n+1}}(\epsilon,\lambda,u_n)h+R_{n}(h)+r_{n},\label{E3.38}
\end{align}
where $\mathfrak{L}_{N_{n+1}}(\epsilon,\lambda,u_n)$ is defined in \eqref{E3.25} and
\begin{align}
&R_{n}(h):=-\epsilon P_{N_{n+1}}(F(u_n+h)
-{F}(u_n)-(\mathrm{D}F)(u_n)h),\label{E3.39}\\
&r_{n}:=P_{N_{n+1}}P^{\bot}_{N_{n}}(L_{\lambda}u_{n}-\epsilon F(u_{n}))=P_{N_{n+1}}P^{\bot}_{N_{n}}(\bar{V}u_n-\epsilon F(u_n)).\label{E3.40}
\end{align}
Remark that $P_{N_{n+1}}P^{\bot}_{N_{n}}(D_{\lambda}u_n)=0$ according to \eqref{E1.75}.  Our aim is to prove the linearized operators $\mathfrak{L}_{N_{n+1}}(\epsilon,\lambda,u_{n})$ (recall \eqref{E3.25}) is invertible and to give
 the tame estimates of its inverse using Proposition \ref{pro1}. In addition formulae \eqref{E1.73}, \eqref{E1.72} and \eqref{E2.51} deduce that $\mathfrak{L}_{N_{n+1}}(\epsilon,\lambda,u_{n})$ may be represented by the matrix $\mathcal{A}_{N_{n+1}}(\epsilon,\lambda,u_n)$.
We distinguish two cases.

If $2^{n+1}\leq\chi_0$, then there exists $\chi\in[\chi_0,2\chi_0]$ such that
\begin{align*}
N_{n+1}=\tilde{N}^{\chi},\quad\tilde{N}:=[N^{1/\chi_0}_{n+1}]\in(N^{1/\chi}_0,N_0).
\end{align*}

If $2^{n+1}>\chi_0$, then there exists a unique $p\in[0,n]$  such that
\begin{align*}
N_{n+1}={N}^{\chi}_{p},\quad\chi=2^{n+1-p}\in[\chi_0,2\chi_0).
\end{align*}
\begin{lemm}\label{lemma20}
Let $\mathcal{A}(\epsilon,\lambda,u,\theta)$ be defined in \eqref{E1.97}. $\forall \theta\in\mathbf{R},\forall j_0\in \Gamma_{+}(\boldsymbol M),$
$
\forall \lambda\in\bigcap_{k=1}^{n+1}\mathscr{G}^0_{N_{k}}(u_{k-1}),$  the assumption $\mathrm{(A3)}$ of Proposition \ref{pro1} may be applied to $\mathcal{A}_{N_{n+1},j_0}(\epsilon,\lambda,u_n,\theta)$.
\end{lemm}
\begin{proof}
Define $\mathfrak{K}_{N_{n+1}}:=\mathfrak{W}_{N_{n+1}}\times\mathfrak{J}_{N_{n+1}}$ with
\begin{align*}
\mathfrak{W}_{N_{n+1}}:=[-N_{n+1},N_{n+1}]^{\nu}\cap\mathbf{Z}^\nu,
\quad\mathfrak{J}_{N_{n+1}}:=\Big\{j_0+\sum\limits_{k=1}^{r}j_{k}\textbf{w}_k:j_k\in[-N_{n+1},N_{n+1}]\Big\}\cap\Gamma_{+}(\boldsymbol M).
\end{align*}

Let us consider the case $2^{n+1}\leq\chi_0$. By Lemma \ref{lemma5}, if the site $\mathfrak{n}=(l,j)\in\mathfrak{K}_{N_{n+1}}$ is ($\mathcal{A}(u_n,\theta),\tilde{N}$)-strongly-good for $\mathcal{A}(\epsilon,\lambda,u_{n},\theta)$ with $|(l,j-j_0)|\leq N_{n+1}$, then it is $(\mathcal{A}_{N_{n+1},j_0}(\epsilon,\lambda,u_{n},\theta),\tilde{N})$-good. This implies
\begin{align}\label{E3.35}
\left\{(\mathcal{A}_{N_{n+1},j_0}(\epsilon,\lambda,u_n,\theta),\tilde{N})\text{-bad~ sites}\right\}\subset\Big\{(&\mathcal{A}(u_n,\theta),\tilde{N})\text{-weakly-bad}~\text{sites~of}\nonumber\\
&\mathcal{A}(\epsilon,\lambda,u_{n},\theta)~\text{with}~|(l,j-j_0)|\leq N_{n+1}\Big\}.
\end{align}
It follows from Lemma \ref{lemma19} and $(F1)_{n}$ that $\mathscr{G}_{\tilde{N}}(u_n)=\Lambda$, which shows that $\lambda\in\cap_{k=1}^{n+1}\mathscr{G}^0_{N_{k}}(u_{k-1})\subset\mathscr{G}_{\tilde{N}}(u_n)$ is $\tilde{N}$-good for $\mathcal{A}_{N_{n+1},j_0}(\epsilon,\lambda,u_n,\theta)$. Combining this with \eqref{E3.33}, \eqref{E3.35} and Proposition \ref{pro3} gives that the assumption $\mathrm{(A3)}$ of Proposition \ref{pro1} applies to $\mathcal{A}_{N_{n+1},j_0}(\epsilon,\lambda,u_n,\theta)$.

If $2^{n+1}>\chi_0$, then a simple discussion as above yields
\begin{align}\label{E3.36}
\left\{(\mathcal{A}_{N_{n+1},j_0}(\epsilon,\lambda,u_n,\theta),{N}_{p})\text{-bad~ sites}\right\}\subset\{(&\mathcal{A}(u_n,\theta),{N}_{p})\text{-weakly-bad}~\text{sites~of}\nonumber\\
&\mathcal{A}(\epsilon,\lambda,u_{n},\theta)~\text{with}~|(l,j-j_0)|\leq N_{n+1}\}.
\end{align}
If the following
\begin{align}\label{E3.37}
\bigcap_{k=1}^{n+1}\mathscr{G}^0_{N_{k}}(u_{k-1})\subset\mathscr{G}_{{N}_{p}}(u_n)
\end{align}
holds, by \eqref{E3.33}, \eqref{E3.36} and Proposition \ref{pro3},  we have that the assumption $\mathrm{(A3)}$ of Proposition \ref{pro1} applies to $\mathcal{A}_{N_{n+1},j_0}(\epsilon,\lambda,u_n,\theta)$.

 Let us verify formula \eqref{E3.37}. In fact, for $p=0$, it follows from Lemma \ref{lemma19} and $(F1)_{n}$ that $\mathscr{G}_{{N}_0}(u_n)=\Lambda$. Hence it is clear that \eqref{E3.37} holds for $p=0$. For $p\geq1$, one has
\begin{align*}
\|u_n-u_{p-1}\|_{s_1}\leq\sum\limits_{k=p}^{n}\|u_{k}-u_{k-1}\|_{s_1}\stackrel{(F2)_{k}}{\leq}\sum\limits_{k=p}^{n}N^{-\sigma-1}_{k}\leq N^{-\sigma}_{p}\sum\limits_{k=p}^{n}N^{-1}_{k}\leq N^{-\sigma}_{p},
\end{align*}
which leads to
\begin{align*}
\bigcap_{k=1}^{n+1}\mathscr{G}^0_{N_{k}}(u_{k-1})\subset\bigcap_{k=1}^{p}\mathscr{G}^0_{N_{k}}(u_{k-1})\stackrel{(F3)_{p}}\subset\mathscr{G}_{{N}_{p}}(u_n).
\end{align*}
\end{proof}
\begin{lemm}\label{lemma22}
For all $\lambda\in\mathcal{B}(\mathscr{D}_{n+1},2N^{-\sigma}_{n+1})$, the operator $\mathfrak{L}_{N_{n+1}}(\epsilon,\lambda,u_n)$ is invertible with
\begin{align}\label{E3.43}
|\mathfrak{L}^{-1}_{N_{n+1}}(\epsilon,\lambda,u_n)|_{s_1}\leq N^{\tau_2+\delta s_1}_{n+1},\quad|\mathfrak{L}^{-1}_{N_{n+1}}(\epsilon,\lambda,u_n)|_{s_2}\leq C''(s_2)N^{\tau_2+\delta s_2}_{n+1}.
\end{align}
\end{lemm}
\begin{proof}
The operator $\mathfrak{L}_{N_{n+1}}(\epsilon,\lambda,u_{n})$ is represented by the  matrix $\mathcal{A}_{N_{n+1}}(\epsilon,\lambda,u_n)$.
Let $\lambda\in\mathscr{D}_{n+1}$ (recall \eqref{E3.17}), which leads to $\lambda\in\mathscr{U}_{N_{n+1}}(u_{n})$. Definition \eqref{E3.9} gives that $\mathcal{A}_{N_{n+1}}(\epsilon,\lambda,u_n)$ is invertible with
\begin{align}\label{E3.44}
\|\mathcal{A}^{-1}_{N_{n+1}}(\epsilon,\lambda,u_n)\|_0\leq N^{\tau}_{n+1}.
\end{align}
In addition formulae \eqref{E1.22}, \eqref{E1.14},  \eqref{E1.99} and $(F1)_{n}$ deduce
\begin{align}\label{E3.45}
\left|\mathcal{A}_{N_{n+1}}(\epsilon,\lambda,u_n)-\mathrm{Diag}(\mathcal{A}_{N_{n+1}}(\epsilon,\lambda,u_n))\right|_{s_1-\varrho}
\leq C(s_1)(\|V\|_{s_1}+\epsilon\|(\mathrm{D}F)(u_{n})\|_{s_1})\leq C'(V).
\end{align}
Under \eqref{E3.44}-\eqref{E3.45} and Lemma \ref{lemma20} for $\theta=0,j_0=0$, the assumptions $\mathrm{(A1)}$-$\mathrm{(A3})$ in Proposition \ref{pro1} are satisfied. If $2^{n+1}\leq\chi_0$ ($\mathrm{resp}.$ $2^{n+1}>\chi_0$), combining this with Remark \ref{rema1} yields that Proposition \ref{pro1} is applied to $\mathcal{A}:=\mathcal{A}_{N_{n+1}}(\epsilon,\lambda,u_n)$ with
\begin{align*}
&\mathfrak{A}:=[-N_{n+1},N_{n+1}]^{\nu}\times\Big(\Big\{j_0+\sum\limits_{p=1}^{r}j_{k}\textbf{w}_k:j_k\in[-N_{n+1},N_{n+1}]\Big\}
\cap\Gamma_{+}(\boldsymbol M)\Big),\\
& N':=N_{n+1}, ~ N:=\tilde{N}~(\mathrm{resp.}~N:=N_p).
\end{align*}
Hence, for $s=s_1$, $\delta s_1\geq\varrho/2$, it follows from \eqref{E2.43}, \eqref{E1.22}, \eqref{E1.10}, \eqref{E1.99}, \eqref{E1.14}, \eqref{E3.41} and $(F1)_{n}$ that, for all $\lambda\in\mathscr{D}_{n+1}$,
\begin{align}
|\mathcal{A}^{-1}_{N_{n+1}}(\epsilon,\lambda,u_n)|_{s_1}\leq&\frac{1}{4}N^{\tau_2}_{n+1}(N^{\delta s_1}_{n+1}+C(s_1)(\|V\|_{s_1+\varrho}+\epsilon\|(\mathrm{D}F)(u_n)\|_{s_1+\varrho}))\nonumber\\
\leq&\frac{1}{4}N^{\tau_2}_{n+1}\left(N^{\delta s_1}_{n+1}+C'+\epsilon C'(s_1)(1+c^{\varrho}_2N^{\varrho}_{n}\|u_n\|_{s_1})\right)\nonumber\\
\leq&\frac{1}{2}N^{\tau_2+\delta s_1}_{n+1}.\label{E3.46}
\end{align}
Moreover for $s=s_2$, $\delta s_1\geq\varrho/2$, by \eqref{E2.43}, \eqref{E1.22}, \eqref{E1.10}, \eqref{E1.99}, \eqref{E1.14}, \eqref{E3.41}, $\delta=1/4$, \eqref{E3.34} and $(F5)_{n}$, we have that, for all $\lambda\in\mathscr{D}_{n+1}$,
\begin{align}
|\mathcal{A}^{-1}_{N_{n+1}}(\epsilon,\lambda,u_n)|_{s_2}\leq&\frac{1}{4}N^{\tau_2}_{n+1}(N^{\delta s_2}_{n+1}+C(s_2)(\|V\|_{s_2+\varrho}+\epsilon\|(\mathrm{D}F)(u_n)\|_{s_2+\varrho}))\nonumber\\
\leq&\frac{1}{4}N^{\tau_2}_{n+1}\left(N^{\delta s_2}_{n+1}+C'+\epsilon C'(s_2)(1+c^\varrho_2N^{\varrho}_{n}\|u_n\|_{s_2})\right)\nonumber\\
\leq&\frac{1}{2}N^{\tau_2+\delta s_2}.\label{E3.47}
\end{align}
Formulae \eqref{E3.46}-\eqref{E3.47} give that, for all $\lambda\in\mathscr{D}_{n+1}$,
\begin{align}\label{E3.48}
|\mathfrak{L}^{-1}_{N_{n+1}}(\epsilon,\lambda,u_n)|_{s}\leq\frac{1}{2} N^{\tau_2+\delta s}_{n+1},\quad s=s_1,s_2.
\end{align}
For all  $\lambda'\in\mathcal{B}(\mathscr{D}_{n+1},2N^{-\sigma}_{n+1})$, there exists some $\lambda\in\mathscr{D}_{n+1}$ such that $|\lambda'-\lambda|<2N^{-\sigma}_{n+1}$. It is obvious that
\begin{align*}
\mathfrak{L}_{N_{n+1}}(\epsilon,\lambda',u_n(\epsilon,\lambda'))=\mathfrak{L}_{N_{n+1}}(\epsilon,\lambda,u_n(\epsilon,\lambda))+\mathscr{R},
\end{align*}
where
\begin{align*}
\mathscr{R}=&\mathfrak{L}_{N_{n+1}}(\epsilon,\lambda',u_n(\epsilon,\lambda'))-\mathfrak{L}_{N_{n+1}}(\epsilon,\lambda,u_n(\epsilon,\lambda))\\
=&P_{N_{n+1}}((\lambda'+\lambda)(\lambda'-\lambda)(\omega_0\cdot \partial_\varphi)^2)_{|H_{N_{n+1}}}\\
&-\epsilon P_{N_{n+1}}((\mathrm{D}F)(u_{n}(\epsilon,\lambda'))-(\mathrm{D}F)(u_{n}(\epsilon,\lambda)))_{|H_{N_{n+1}}}.
\end{align*}
It follows from \eqref{E1.8}, \eqref{E1.62}, \eqref{E3.41}, $\sigma\geq2$ (see \eqref{E3.42}) and  $N_0\geq2$ that
\begin{align}\label{E3.49}
|P_{N_{n+1}}((\lambda'+\lambda)(\lambda'-\lambda)(\omega_0\cdot \partial_\varphi)^2)_{|H_{N_{n+1}}}|_{s}\leq 8c^2_2N^{-\sigma+2}_{n+1}\quad\text{for}\quad s=s_1,s_2.
\end{align}
In addition, applying \eqref{E1.22}, \eqref{E1.10}, \eqref{E1.86}, \eqref{E3.33}-\eqref{E3.34}, \eqref{E3.15}, $(F_1)_{n}$ and $\delta s_1\geq\varrho/2$, we deduce
\begin{align}
|\epsilon P_{N_{n+1}}((\mathrm{D}F)(u_{n}(\epsilon,\lambda'))-(\mathrm{D}F)(u_{n}(\epsilon,\lambda)))_{|H_{N_{n+1}}}|_{s_1}\leq &C(s_1)N^{-\sigma+\delta s_1}_{n+1}\label{E3.50},\\
|\epsilon P_{N_{n+1}}((\mathrm{D}F)(u_{n}(\epsilon,\lambda'))-(\mathrm{D}F)(u_{n}(\epsilon,\lambda)))_{|H_{N_{n+1}}}|_{s_2}\leq &C(s_2)N^{4\tau_2+2s_1+6}_{n}N^{-\sigma+\delta s_1}_{n+1}.\label{E3.51}
\end{align}
As a consequence
\begin{align*}
|\mathfrak{L}^{-1}_{N_{n+1}}(\epsilon,\lambda,u_n(\epsilon,\lambda))|_{s_1}|\mathscr{R}|_{s_1}\stackrel{\eqref{E3.48}\text{-}\eqref{E3.50}}{\leq}
C'(s_2)N^{\tau_2+\delta s_1}_{n+1}N^{-\sigma+\delta s_1+2}_{n+1}
\stackrel{\eqref{E3.42}}{\leq}1/2.
\end{align*}
Hence Lemma \ref{lemma2} gives that $\mathfrak{L}_{N_{n+1}}(\epsilon,\lambda',u_n(\epsilon,\lambda'))$ is invertible with
\begin{align*}
|\mathfrak{L}^{-1}_{N_{n+1}}(\epsilon,\lambda',u_n(\epsilon,\lambda'))|_{s_1}
\stackrel{\eqref{E1.33}}{\leq}&2|\mathfrak{L}^{-1}_{N_{n+1}}(\epsilon,\lambda,u_n(\epsilon,\lambda))|_{s_1}\stackrel{\eqref{E3.48}}{\leq}N^{\tau_2+\delta s_1}_{n+1},\\
|\mathfrak{L}^{-1}_{N_{n+1}}(\epsilon,\lambda',u_n(\epsilon,\lambda'))|_{s_2}
\stackrel{\eqref{E1.92}}{\leq}&C(s_2)(|\mathfrak{L}^{-1}_{N_{n+1}}(\epsilon,\lambda,u_n(\epsilon,\lambda))|_{s_2}
+|\mathfrak{L}^{-1}_{N_{n+1}}(\epsilon,\lambda,u_n(\epsilon,\lambda))|^2_{s_1}|\mathscr{R}|_{s_2})\\
\stackrel{\eqref{E3.48}\text{-}\eqref{E3.49},\eqref{E3.51}}{\leq}&C(s_2)(\frac{1}{2}N^{\tau_2+\delta s_2}_{n+1}+\frac{1}{4}N^{2(\tau_2+\delta s_1)}_{n+1}C'(s_2)N^{4\tau_2+2s_1+6}_{n}N^{-\sigma+\delta s_1}_{n+1})\\
\stackrel{\eqref{E3.41},\eqref{E3.34},\delta=\frac{1}{4}}\leq &C''(s_2)N^{\tau_2+\delta s_2}_{n+1}.
\end{align*}
The proof of the lemma is completed.
\end{proof}
Then, owing to Lemma \ref{lemma22}, solving the equation $(P_{N_{n+1}})$ (see also \eqref{E3.38}) is reduced to the fixed point problem $h=U_{n+1}(h)$ with
\begin{align}\label{E3.52}
U_{n+1}:{H}_{N_{n+1}}\rightarrow {H}_{N_{n+1}},\quad h\mapsto-\mathfrak{L}^{-1}_{N_{n+1}}(\epsilon,\lambda,u_n)(R_{n}(h)+r_{n}),
\end{align}
where $R_{n}(h),r_{n}$ are defined in \eqref{E3.39}-\eqref{E3.40}.
\begin{lemm}\label{lemma21}
For all $\lambda\in\mathcal{B}(\mathscr{D}_{n+1},2N^{-\sigma}_{n+1})$, the map $U_{n+1}$ is a contraction in
\begin{align*}
\mathcal{B}(0,\rho_{n+1}):=\left\{h\in H_{N_{n+1}}:\|h\|_{s_1}\leq\rho_{n+1}:=N_{n+1}^{-\sigma-1}\right\}.
\end{align*}
Moreover the unique fixed point $\tilde{h}_{n+1}(\epsilon,\lambda)$ of $U_{n+1}$ satisfies
\begin{align}\label{E3.54}
\|\tilde{h}_{n+1}\|_{s_1}\leq2K(s_2)N^{\tau_2+\delta s_1}_{n+1}N^{-(s_2-s_1)}_{n}\mathscr{A}_{n}.
\end{align}
where $\mathscr{A}_{n}$ is seen in $(F5)_{n}$.
\end{lemm}
\begin{proof}
For all $\lambda\in\mathcal{B}(\mathscr{D}_{n+1},2N^{-\sigma}_{n+1})$, it follows from  \eqref{E3.39}-\eqref{E3.40}, \eqref{E1.11}, \eqref{E1.14}, \eqref{E103}-\eqref{E1.99} and $(F5)_n$ that
\begin{align}
\|r_n\|_{s_1}+\|R_{n}(h)\|_{s_1}\leq& c^{-(s_2-s_1)}_1N^{-(s_2-s_1)}_{n}(\|\bar{V}u_n\|_{s_2}+\epsilon\|F(u_{n})\|_{s_2})+\epsilon C(s_1)\|h\|^2_{s_1}\nonumber\\
\leq&C(s_2)c^{-(s_2-s_1)}_1N^{-(s_2-s_1)}_{n}(1+\mathscr{A}_{n})+\epsilon C(s_1)\|h\|^2_{s_1}\label{E3.55}\\
\leq&2C(s_2)c^{-(s_2-s_1)}_1N^{-(s_2-s_1)+2\tau_2+2\delta s_1+2}_{n}+\epsilon C(s_1)\|h\|^2_{s_1}.\label{E3.56}
\end{align}
Letting $\|h\|_{s_1}\leq\rho_{n+1}$, by \eqref{E3.52}, \eqref{E1.25}, \eqref{E3.43}, \eqref{E3.56} and \eqref{E3.41}, we check that, for some constants $K(s_1),K(s_2)>0$,
\begin{align*}
\|U_{n+1}(h)\|_{s_1}\leq K(s_2)N^{2(\tau_2+\delta s_1+1)}_{n+1}N^{-(s_2-s_1)}_{n}+\epsilon K(s_1)N^{\tau_2+\delta s_1}_{n+1}\rho^2_{n+1}.
\end{align*}
Formulae \eqref{E3.34}-\eqref{E3.42} imply that, for $N\geq N_0(s_2)$ large enough,
\begin{align}\label{E3.57}
K(s_2)N^{2(\tau_2+\delta s_1+1)}_{n+1}N^{-(s_2-s_1)}_{n}\leq\rho_{n+1}/2,\quad\epsilon K(s_1)N^{\tau_2+\delta s_1}_{n+1}\rho_{n+1}\leq1/2,
\end{align}
which leads to $\|U_{n+1}(h)\|_{s_1}\leq \rho_{n+1}$.
Moreover differentiating \eqref{E3.52} with respect to $h$ yields
\begin{align*}
\mathrm{D}U_{n+1}(h)[v]\stackrel{\eqref{E3.39}\text{-}\eqref{E3.40}}=\epsilon \mathfrak{L}^{-1}_{N_{n+1}}(\epsilon,\lambda,u_n)P_{N_{n+1}}((\mathrm{D}F)(u_{n}+h)[v]-(\mathrm{D}F)(u_{n})[v]).
\end{align*}
Using \eqref{E1.25}, \eqref{E3.43}, \eqref{E1.15}, $(F1)_{n}$ and \eqref{E3.57},  we deduce
\begin{align*}
\|\mathrm{D}U_{n+1}(h)[v]\|_{s_1}\leq&\epsilon C(s_1)N^{\tau_2+\delta s_1}_{n}\|P_{N_{n+1}}((\mathrm{D}F)(u_{n}+h)[v]-(\mathrm{D}F)(u_{n})[v])\|_{s_1}\\
\leq&\epsilon K(s_1)N^{\tau_2+\delta s_1}_{n+1}\rho_{n+1}\|v\|_{s_1}\leq\frac{1}{2}\|v\|_{s_1}.
\end{align*}
Hence $U_{n+1}$ is a contraction in $\mathcal{B}(0,\rho_{n+1})$. Let $\tilde{h}_{n+1}(\epsilon,\lambda)$ denote by the unique fixed point of $U_{n+1}$. In addition, by means of  \eqref{E1.25}, \eqref{E3.43}, \eqref{E3.52}, \eqref{E3.55}, \eqref{E3.41} and $\|\tilde{h}_{n+1}\|\leq\rho_{n+1}$,  we obtain
\begin{align*}
\|\tilde{h}_{n+1}\|_{s_1}\leq K(s_2)N^{\tau_2+\delta s_1}_{n+1}N^{-(s_2-s_1)}_{n}\mathscr{A}_{n}+\epsilon K(s_1)N^{\tau_2+\delta s_1}_{n+1}\rho_{n+1}\|\tilde{h}_{n+1}\|_{s_1}.
\end{align*}
Combining this with \eqref{E3.57} gives that \eqref{E3.54} holds.
\end{proof}

Denote
\begin{align}\label{E3.58}
\mathscr{Q}_{n+1}(\epsilon,\lambda,h):=P_{N_{n+1}}(L_{\lambda}(u_{n}+h)-\epsilon {F}(u_{n}+h))=0,\quad \forall h\in H_{n+1}.
\end{align}
Lemma \ref{lemma21} indicates that for all $\lambda\in\mathcal{B}(\mathscr{D}_{n+1},2N^{-\sigma}_{n+1})$, $\tilde{h}_{n+1}$ is a solution of equation \eqref{E3.58}. Since $u_{n}(0,\lambda)\stackrel{(F1)_{n}}{=}0$, by uniqueness, we derive
\begin{align*}
\tilde{h}_{n+1}(0,\lambda)=0,\quad \forall\lambda\in\mathcal{B}(\mathscr{D}_{n+1},2N^{-\sigma}_{n+1}).
\end{align*}
Let us verify the upper bound of $\tilde{h}_{n+1}$ in high norm.
\begin{lemm}\label{lemma23}
For all $\lambda\in\mathcal{B}(\mathscr{D}_{n+1},2N^{-\sigma}_{n+1})$, one has
\begin{align}\label{E3.59}
\|\tilde{h}_{n+1}\|_{s_2}\leq K(s_2)N^{\tau_2+\delta s_1}_{n+1}\mathscr{A}_n.
\end{align}
\end{lemm}
\begin{proof}
Applying \eqref{E3.52}, \eqref{E1.25} and  \eqref{E3.43} yields
\begin{align*}
\|\tilde{h}_{n+1}\|_{s_2}=&\|\mathfrak{L}^{-1}_{N_{n+1}}(\epsilon,\lambda,u_n)(R_{n}(\tilde{h}_{n+1})+r_{n})\|_{s_2}\\
\leq &C(s_2)(N^{\tau_2+\delta s_1}_{n+1}(\|R_{n}(\tilde{h}_{n+1})\|_{s_2}+\|r_{n}\|_{s_2})+C''({s_2})N^{\tau_2+\delta s_2}_{n+1}(\|R_{n}(\tilde{h}_{n+1})\|_{s_1}+\|r_{n}\|_{s_1})).
\end{align*}
Let us check the upper bounds of $\|R_{n}(\tilde{h}_{n+1})\|_{s}+\|r_{n}\|_{s}$ for $s=s_1,s_2$. It follows from \eqref{E3.39}, \eqref{E103}, \eqref{E3.54}, $\|\tilde{h}_{n+1}\|_{s_1}\leq \rho_{n+1}$, $(F5)_{n}$, \eqref{E1.16} and \eqref{E3.57} that
\begin{align}\label{E3.60}
\|R_{n}(\tilde{h}_{n+1})\|_{s_1}\leq C(s_2)N^{-(s_2-s_1)}_{n}\mathscr{A}_n,\quad \|R_{n}(\tilde{h}_{n+1})\|_{s_2}\leq C(s_2)(\epsilon \rho^2_{n+1}\mathscr{A}_n+\epsilon\rho_{n+1}\|\tilde{h}_{n+1}\|_{s_2}).
\end{align}
Moreover, using \eqref{E3.40}, \eqref{E1.11}, \eqref{E1.14}, \eqref{E1.99} and $(F5)_{n}$, we deduce
\begin{align}\label{E3.61}
\|r_{n}\|_{s_1}\leq C(s_2)N^{-(s_2-s_1)}_{n}\mathscr{A}_{n},\quad\|r_n\|_{s_2}\leq C(s_2)\mathscr{A}_{n}.
\end{align}
Consequently, by \eqref{E3.60}-\eqref{E3.61}, \eqref{E3.41}, \eqref{E3.34}-\eqref{E3.42} and the definition of $\rho_{n+1}$, we get that, for $\epsilon$ small enough,
\begin{align*}
\|\tilde{h}_{n+1}\|_{s_2}\leq& C(s_2)(N^{\tau_2+\delta s_1}_{n+1}C(s_2)(2\mathscr{A}_{n}+\epsilon\rho_{n+1}\|\tilde{h}_{n+1}\|_{s_2})+C''(s_2)N^{\tau_2+\delta s_2}_{n+1}(2C(s_2)N^{-(s_2-s_1)}_{n}\mathscr{A}_{n}))\\
\leq&C'''(s_2)N^{\tau_2+\delta s_1}_{n+1}\mathscr{A}_{n}+\epsilon C(s_2)N^{\tau_2+\delta s_1-\sigma-1}_{n+1}\|\tilde{h}_{n+1}\|_{s_2}\\
\leq&C'''(s_2)N^{\tau_2+\delta s_1}_{n+1}\mathscr{A}_{n}+\frac{1}{2}\|\tilde{h}_{n+1}\|_{s_2},
\end{align*}
which leads to $\|\tilde{h}_{n+1}\|_{s_2}\leq 2C'''(s_2)N^{\tau_1+\delta s_1}_{n+1}\mathscr{A}_{n}$.
\end{proof}
Next, let us estimate the derivatives of $\tilde{h}_{n+1}$ with respect to $\lambda$.
\begin{lemm}
For all $\lambda\in\mathcal{B}(\mathscr{D}_{n+1},2N^{-\sigma}_{n+1})$, one has $\tilde{h}_{n+1}(\epsilon,\cdot)\in\mathcal{B}(\mathscr{D}_{n+1},2N^{-\sigma}_{n+1})$ with
\begin{align}\label{E3.62}
\|\partial_{\lambda}\tilde{h}_{n+1}\|_{s_1}\leq N^{-1}_{n+1},\quad\|\partial_{\lambda}\tilde{h}_{n+1}\|_{s_2}\leq N^{\tau_2+\delta s_1+1}_{n+1}(N^{\tau_2+\delta s_1+2}_{n+1}\mathscr{A}_{n}+\mathscr{A}'_{n}).
\end{align}
\end{lemm}
\begin{proof}
Lemma \ref{lemma21} shows that for all $\lambda\in\mathcal{B}(\mathscr{D}_{n+1},2N^{-\sigma}_{n+1})$, $\tilde{h}_{n+1}$ is a solution of equation \eqref{E3.58}. Applying \eqref{E3.58} and \eqref{E3.25} yields
\begin{align}\label{E3.67}
\mathrm{D}_{h}\mathscr{Q}_{n+1}(\epsilon,\lambda,\tilde{h}_{n+1})=\mathfrak{L}_{N_{n+1}}(\epsilon,\lambda,u_n+\tilde{h}_{n+1})
=\mathfrak{L}_{N_{n+1}}(\epsilon,\lambda,u_n)+\mathscr{R}',
\end{align}
where $\mathscr{R}'=-\epsilon P_{N_{n+1}}((\mathrm{D}F)(u_{n}+\tilde{h}_{n+1})-(\mathrm{D}F)(u_{n}))_{|H_{N_{n+1}}}$. It follows from  \eqref{E1.22}, \eqref{E1.86}, \eqref{E1.10}, $\delta s_1\geq\varrho/2$, $(P1)_n$, $(F1)_{n}$, $(F5)_{n}$,   $\|\tilde{h}_{n+1}\|\leq\rho_{n+1}<1$ and \eqref{E3.59} that
\begin{align}
|\mathscr{R}'|_{s_1}\leq C(s_1)N^{2\delta s_1}_{n+1}\rho_{n+1}, \quad|\mathscr{R}'|_{s_2}\leq C(s_2)N^{\tau_2+3\delta s_1}_{n+1}\mathscr{A}_{n}.\label{E3.64}
\end{align}
Using \eqref{E3.43}, \eqref{E3.64} together with \eqref{E3.42}, we deduce
\begin{align*}
|\mathfrak{L}^{-1}_{N_{n+1}}(\epsilon,\lambda,u_n)|_{s_1}|\mathscr{R}'|_{s_1}\leq1/2.
\end{align*}
Combining this with Lemma \ref{lemma2} yields
\begin{align}
|\mathfrak{L}^{-1}_{N_{n+1}}(\epsilon,\lambda,u_n+\tilde{h}_{n+1})|_{s_1}\stackrel{\eqref{E1.33}}{\leq}&
2|\mathfrak{L}^{-1}_{N_{n+1}}(\epsilon,\lambda,u_n)|_{s_1}\stackrel{\eqref{E3.43}}{\leq}2N^{\tau_2+\delta s_1}_{n+1},\label{E3.65}\\
|\mathfrak{L}^{-1}_{N_{n+1}}(\epsilon,\lambda,u_n+\tilde{h}_{n+1})|_{s_1}\stackrel{\eqref{E1.92}}{\leq}&C(s_2)
(|\mathfrak{L}^{-1}_{N_{n+1}}(\epsilon,\lambda,u_n)|_{s_2}
+|\mathfrak{L}^{-1}_{N_{n+1}}(\epsilon,\lambda,u_n)|^2_{s_1}|\mathscr{R}'|_{s_2})\nonumber\\
\stackrel{\eqref{E3.43},\eqref{E3.64},(F5)_n,\eqref{E3.34},\delta=\frac{1}{4},\eqref{E3.41}}{\leq}&\tilde{C}(s_2)N^{\tau_2+\delta s_2}_{n+1}.\label{E3.66}
\end{align}
Therefore the implicit function theorem establishes $\tilde{h}_{n+1}\in C^1(\mathcal{B}(\mathscr{D}_{n+1},2N^{-\sigma}_{n+1});H_{N_{n+1}})$,
 which then infers
\begin{equation*}
\partial_{\lambda}\mathscr{Q}_{n+1}(\epsilon,\lambda,\tilde{h}_{n+1})+
\mathrm{D}_{h}\mathscr{Q}_{n+1}(\epsilon,\lambda,\tilde{h}_{n+1})\partial_{\lambda}\tilde{h}_{n+1}=0.
\end{equation*}
Consequently, using the fact that $u_n$ solves \eqref{E3.21} deduced by \eqref{E3.68} and $(F4)_n$, we deduce
\begin{align*}
\partial_{\lambda}\tilde{h}_{n+1}\stackrel{\eqref{E3.58},\eqref{E3.67}}{=}-\mathfrak{L}^{-1}_{N_{n+1}}(\epsilon,\lambda,u_n+\tilde{h}_{n+1})
\partial_{\lambda}\mathscr{Q}_{n+1}(\epsilon,\lambda,\tilde{h}_{n+1}),
\end{align*}
where
\begin{align*}
\partial_{\lambda}\mathscr{Q}_{n+1}(\epsilon,\lambda,\tilde{h}_{n+1})=&P_{N_{n+1}}((\partial_{\lambda}L_{\lambda})\tilde{h}_{n+1})+
P_{N_{n+1}}P^{\bot}_{N_{n}}(\bar{V}\partial_{\lambda}u_{n})-\epsilon P_{N_{n+1}}P^{\bot}_{N_{n}}((\mathrm{D}F)(u_{n})\partial_{\lambda}u_{n})\\
&-\epsilon P_{N_{n+1}}((\mathrm{D}F)(u_{n}+\tilde{h}_{n+1})\partial_{\lambda}u_{n}-(\mathrm{D}F)(u_n)\partial_{\lambda}u_{n}).
\end{align*}
Remark that $P_{N_{n+1}}P^{\bot}_{N_{n}}(D_{\lambda}u_n)=0$ according to \eqref{E1.75}. To establish \eqref{E3.62}, we have to verify the upper bounds of $\partial_{\lambda}\mathscr{Q}_{n+1}(\epsilon,\lambda,\tilde{h}_{n+1})$ in $s_1,s_2$-norms. It follows from \eqref{E1.75}, \eqref{E1.10}-\eqref{E1.11}, \eqref{E3.54}, \eqref{E1.99}, $(F1)_{n}$, $(F5)_{n}$, $\|\tilde{h}_{n+1}\|_{s_1}\leq\rho_{n+1}<1$, \eqref{E1.86}-\eqref{E1.15}, \eqref{E3.15} and \eqref{E3.59} that
\begin{align*}
\|\partial_{\lambda}\mathscr{Q}_{n+1}(\epsilon,\lambda,\tilde{h}_{n+1})\|_{s_1}\leq &C(s_2)(N^{\tau_2+\delta s_1+2}_{n+1}N^{-(s_2-s_1)}_{n}\mathscr{A}_{n}+N^{-(s_2-s_1)}_{n}\mathscr{A}'_{n}),\\
\|\partial_{\lambda}\mathscr{Q}_{n+1}(\epsilon,\lambda,\tilde{h}_{n+1})\|_{s_2}\leq& C(s_2)(N^{\tau_2+\delta s_1+2}_{n+1}\mathscr{A}_{n}+\mathscr{A}'_{n}),
\end{align*}
which give rise to
\begin{align*}
\|\partial_{\lambda}\tilde{h}_{n+1}\|_{s_1}\stackrel{\eqref{E1.25},\eqref{E3.65}}{\leq}&C'(s_2)N^{\tau_2+\delta s_1}_{n+1}(N^{\tau_2+\delta s_1+2}_{n+1}N^{-(s_2-s_1)}_{n}\mathscr{A}_{n}+N^{-(s_2-s_1)}_{n}\mathscr{A}'_{n})\\
\stackrel{\eqref{E3.41},\eqref{E3.34},(F5)_{n},\delta=1/4}\leq &N^{-1}_{n+1},
\end{align*}
\begin{align*}
\|\partial_{\lambda}\tilde{h}_{n+1}\|_{s_1}\stackrel{\eqref{E1.25},\eqref{E3.65}\text{-}\eqref{E3.66},\eqref{E3.41},\eqref{E3.34},\delta=1/4}{\leq}N^{\tau_2+\delta s_1+1}_{n+1}(N^{\tau_2+\delta s_1+2}_{n+1}\mathscr{A}_{n}+\mathscr{A}'_{n}).
\end{align*}
\end{proof}
Define a $C^{\infty}$ cut-off function $\psi_{n+1}:\Lambda\rightarrow[0,1]$ as
\begin{align}\label{E3.69}
\psi_{n+1}:=
\begin{cases}
1\quad\quad&\text{if}\quad\lambda\in\mathcal{B}(\mathscr{D}_{n+1},N^{-\sigma}_{n+1}),\\
0\quad\quad&\text{if}\quad\lambda\in\mathcal{B}(\mathscr{D}_{n+1},2N^{-\sigma}_{n+1}),
\end{cases}
\quad\text{with}\quad |\partial_{\lambda}\psi_{n+1}|\leq CN^{\sigma}_{n+1}.
\end{align}
Then we define ${h}_{n+1}$ as
\begin{align}\label{E3.70}
{h}_{n+1}(\epsilon,\lambda):
=\begin{cases}
\psi_{n+1}\tilde{h}_{n+1}(\epsilon,\lambda)\quad&\text{if}\quad\lambda\in\mathcal{B}(\mathscr{D}_{n+1},2N^{-\sigma}_{n+1}),\\
0\quad&\text{if}\quad\lambda\notin\mathcal{B}(\mathscr{D}_{n+1},2N^{-\sigma}_{n+1}).
\end{cases}
\end{align}
Hence, by \eqref{E3.69}-\eqref{E3.70}, \text{Lemma }\ref{lemma21}, \eqref{E3.59}, \eqref{E3.62}, $(F5)_{n}$, \eqref{E3.42} and $\delta=1/4$, one has that $h_{n+1}\in C^1(\Lambda;H_{N_{n+1}})$ with
\begin{align}
\|h_{n+1}\|_{s_1}&\leq|\psi_{n+1}|\|\tilde{h}_{n+1}\|_{s_1}\leq\|\tilde{h}_{n+1}\|_{s_1}{\leq} N^{-\sigma-1}_{n+1},\label{E3.71}\\
\|h_{n+1}\|_{s_2}&\leq|\psi_{n+1}|\|\tilde{h}_{n+1}\|_{s_2}\leq\|\tilde{h}_{n+1}\|_{s_2}{\leq} K(s_2)N^{2\tau_2+2\delta s_1+1}_{n+1},\label{E3.72}\\
\|\partial_{\lambda}h_{n+1}\|_{s_1}&\leq|\partial_{\lambda}\psi_{n+1}|\|\tilde{h}_{n+1}\|_{s_1}+|\psi_{n+1}|\|\partial_{\lambda}\tilde{h}_{n+1}\|_{s_1}
\leq N^{-1/2}_{n+1}\label{E3.73},\\
\|\partial_{\lambda}h_{n+1}\|_{s_2}&\leq|\partial_{\lambda}\psi_{n+1}|\|\tilde{h}_{n+1}\|_{s_2}+|\psi_{n+1}|\|\partial_{\lambda}\tilde{h}_{n+1}\|_{s_2}
{\leq}K'(s_2)N_{n+1}^{3\tau_2+\frac{5}{4} s_1+4}.\label{E3.74}
\end{align}
Moreover it is clear that $h_{n+1}(0,\lambda)=0$. As a consequence, we define $u_{n+1}(\epsilon,\cdot)\in C^{1}(\Lambda;H_{N_{n+1}})$ as
\begin{align}\label{E3.75}
u_{n+1}:=u_n+h_{n+1},
\end{align}
where $h_{n+1}$ is given by \eqref{E3.70}. Let us check that properties  $(F1)_{n+1}$-$(F5)_{n+1}$ hold. It follows from \eqref{E104}-\eqref{E105}, $(F_2)_n$, \eqref{E3.71} and \eqref{E3.73} that
\begin{align}
\|u_{n+1}\|_{s_1}\leq&\|u_0\|_{s_1}+\sum\limits_{k=1}^{n+1}\|h_{k}\|_{s_1}
{\leq}N^{-\sigma}_0+\sum\limits_{k=1}^{n+1}N_{k}^{-\sigma-1}
\leq\frac{1}{2}+N^{-\sigma}_{1}\sum\limits_{k=1}^{n+1}N^{-1}_{k}\leq1,\label{E3.83}\\
\|\partial_{\lambda}u_{n+1}\|_{s_1}\leq&\|\partial_{\lambda}u_0\|_{s_1}+\sum\limits_{k=1}^{n+1}\|\partial_{\lambda}h_{k}\|_{s_1}\leq12c^{s_1+2}_2\gamma^{-1}N^{\tau_1+s_1+1}_0+\sum\limits_{k=1}^{n+1}N_{k}^{-1/2}\leq \bar{C}\gamma^{-1}N^{\tau_1+s_1+1}_0.\label{E3.84}
\end{align}
Therefore property $(F1)_{n+1}$ holds. It is straightforward that property $(F2)_{n+1}$ holds according to \eqref{E3.71}, \eqref{E3.73} and \eqref{E3.75}. By the definitions of $\psi_{n+1},h_{n+1}$ (see \eqref{E3.69}-\eqref{E3.70}), we have that $h_{n+1}=\tilde{h}_{n+1}$ on $\mathcal{B}(\mathscr{D}_{n+1},N^{-\sigma}_{n+1})$. Moreover $h_{n+1}$ solves equation $\mathscr{Q}_{n+1}(\epsilon,\lambda,h)$ (see \eqref{E3.58}), which leads to that $u_{n+1}$ solves equation $(P_{N_{n+1}})$, namely property $(F4)_{n+1}$. Let us check that property $(F5)_{n+1}$ holds. Indeed, the definitions of $\mathscr{A}_{n+1}$ and $\mathscr{A}'_{n+1}$ establish that
\begin{align*}
\mathscr{A}_{n+1}\stackrel{\eqref{E3.75}}{\leq}&\mathscr{A}_{n}+\|h_{n+1}\|_{s_2}\stackrel{(F5)_{n},\eqref{E3.72}}{\leq}N^{2\tau_2+2\delta s_1+2}_{n}+K(s_2)N^{2\tau_2+2\delta s_1+1}_{n+1}\stackrel{\eqref{E3.41}}{\leq}N^{{2\tau_2+2\delta s_1+2}}_{n+1},\\
\mathscr{A}'_{n+1}\stackrel{\eqref{E3.75}}{\leq}&\mathscr{A}'_{n}+\|\partial_{\lambda}h_{n+1}\|_{s_2}\stackrel{(F5)_{n},\eqref{E3.74}}{\leq}
N^{4\tau_2+2 s_1+6}_{n}+K'(s_2)N^{3\tau_2+\frac{5}{4} s_1+4}_{n+1}\leq N^{4\tau_2+2 s_1+6}_{n+1}.
\end{align*}
Now, we are denoted to prove property $(F3)_{n}$.
\begin{lemm}
Property $(F3)_{n+1}$ holds.
\end{lemm}
\begin{proof}
Firstly, if $\|u-u_n\|_{s_1}\leq N^{-\sigma}_{n+1}$, then we claim that
\begin{align}\label{E3.79}
\|\mathcal{A}^{-1}_{N_{n+1},j_0}(\epsilon,\lambda,u_n,\theta)\|_{0}\leq N^{\tau}_{n+1}
\Rightarrow \mathcal{A}_{N_{n+1},j_0}(\epsilon,\lambda,u,\theta)
\text{ is } N_{n+1}\text{-good},
\end{align}
where $\mathcal{A}(\epsilon,\lambda,u,\theta)$ is defined in \eqref{E1.97}. This implies that
\begin{align*}
\mathscr{B}_{N_{n+1}}(j_0;\epsilon,\lambda,u)\subset\mathscr{B}_{N_{n+1}}^{0}(j_0;\epsilon,\lambda,u_n),\quad \forall j_0\in\Gamma_{+}(\boldsymbol M),\forall \theta\in\mathbf{R}
\end{align*}
by  definitions \eqref{E1.80} and \eqref{E2.53}. Hence, using \eqref{E1.81}, \eqref{E3.13},  we establish
\begin{align*}
&\lambda\in\mathscr{G}_{N_{n+1}}^{0}(u_n)\Rightarrow\lambda\in\mathscr{G}_{N_{n+1}}(u),\quad\text{that is }\\
&\|u-u_n\|_{s_1}\leq N^{-\sigma}_{n+1}\text{ and }\lambda\in\bigcap_{k=1}^{n+1}\mathscr{G}_{N_{k}}^{0}(u_{k-1})\Rightarrow \lambda\in\mathscr{G}_{N_{n+1}}(u).
\end{align*}
Let us prove the claim \eqref{E3.79}. It follows from \eqref{E1.22}, \eqref{E1.14}, \eqref{E1.99} and $(F1)_{n}$ that
\begin{align}
\left|\mathcal{A}_{N_{n+1},j_0}(\epsilon,\lambda,u_n,\theta)-\mathrm{Diag}(\mathcal{A}_{N_{n+1},j_0}(\epsilon,\lambda,u_n,\theta))\right|_{s_1-\varrho}
\leq& C(s_1)(\|V\|_{s_1}+\epsilon\|(\mathrm{D}F)(u_{n})\|_{s_1})\nonumber\\
\leq &C'(V)=:\Upsilon.\label{E3.80}
\end{align}
Combining this with the fact $\|\mathcal{A}^{-1}_{N_{n+1},j_0}(\epsilon,\lambda,u_n,\theta)\|_{0}\leq N^{\tau}_{n+1}$ and Lemma \ref{lemma20} yields that the assumptions $\mathrm{(A1)}$-$\mathrm{(A3)}$ in Proposition \ref{pro1} are satisfied. If $2^{n+1}\leq \chi_0$ $(\mathrm{resp}.~2^{n+1}> \chi_0)$, then we apply Proposition \ref{pro1} to $\mathcal{A}:=\mathcal{A}_{N_{n+1}}(\epsilon,\lambda,u_n,\theta)$ with
\begin{align*}
&\mathfrak{A}:=[-N_{n+1},N_{n+1}]^{\nu}\times\Big(\Big\{j_0+\sum\limits_{p=1}^{r}j_{k}\textbf{w}_k:j_k\in[-N_{n+1},N_{n+1}]\Big\}
\cap\Gamma_{+}(\boldsymbol M)\Big),\\
&N':=N_{n+1} \quad N:=\tilde{N}~(\mathrm{resp}.~N:=N_p).
\end{align*}
Hence, for all $s\in[s_0,s_1-\varrho]$, it follows from  \eqref{E2.43}, \eqref{E1.22}, \eqref{E1.14}, \eqref{E1.99}, and $(F1)_{n}$ that
\begin{align}
|\mathcal{A}^{-1}_{N_{n+1},j_0}(\epsilon,\lambda,u_n,\theta)|_{s}\leq\frac{1}{4}N^{\tau_2}_{n+1}(N^{\delta s}_{n+1}+C(s_1)(\|V\|_{s_1}+\epsilon\|(\mathrm{D}F)(u_n)\|_{s_1}))
\leq\frac{1}{2}N^{\tau_2+\delta s}_{n+1}.\label{E3.81}
\end{align}
Moreover
\begin{align*}
\mathcal{A}_{N_{n+1},j_0}(\epsilon,\lambda,u,\theta):=\mathcal{A}_{N_{n+1},j_0}(\epsilon,\lambda,u_n,\theta)+\mathcal{R} \quad
\end{align*}
with
$\mathcal{R}=\mathcal{A}_{N_{n+1},j_0}(\epsilon,\lambda,u,\theta)-\mathcal{A}_{N_{n+1},j_0}(\epsilon,\lambda,u_n,\theta)$. It is clear that
\begin{align*}
|\mathcal{R}|_{s_1-\varrho}\stackrel{\eqref{E1.97},\eqref{E1.22}}{\leq}C(s_1)\|(\mathrm{D}F)(u)-(\mathrm{D}F)(u_n)\|_{s_1}
\stackrel{\eqref{E1.86},(F1)_{n}}{\leq}C'(s_1)\|u-u_{n}\|_{s_1}{\leq}C'(s_1)N^{-\sigma}_{n+1},
\end{align*}
which carries out
\begin{align*}
|\mathcal{A}^{-1}_{N_{n+1},j_0}(\epsilon,\lambda,u_n,\theta)|_{s_1-\varrho}|\mathcal{R}|_{s_1-\varrho}\stackrel{\eqref{E3.81},\eqref{E3.42}}{\leq}1/2.
\end{align*}
By Lemma \ref{lemma2}, we obtain that, for all $s\in[s_0,s_1-\varrho]$,
\begin{align*}
|\mathcal{A}^{-1}_{N_{n+1},j_0}(\epsilon,\lambda,u,\theta)|_{s}\stackrel{\eqref{E1.33},\eqref{E3.81}}{\leq}N^{\tau_2+\delta s}_{n+1},
\end{align*}
which leads to that $\mathcal{A}_{N_{n+1},j_0}(\epsilon,\lambda,u,\theta)$ is $N_{n+1}$-good.
\end{proof}

\subsection{Measure estimate}\label{sec:3.5}

\begin{lemm}\label{lemma18}
The complementary of the set $\mathscr{U}$ defined in \eqref{E1.9} has that, for some constant $\mathfrak{C}_2>0$,
\begin{align}\label{E3.14}
\mathrm{meas}(\Lambda\setminus \mathscr{U})\leq\mathfrak{C}_{2}\gamma.
\end{align}
\end{lemm}
\begin{proof}
Definition \eqref{E1.9} gives
\begin{align*}
\Lambda\setminus\mathscr{U}\subset\bigcup_{|l|,|j|\leq N_0,1\leq p\leq \mathfrak{d}_j}\mathscr{U}_{l,j,p}\quad\text{with}\quad\mathscr{U}_{l,j,p}:=\left\{\lambda\in\Lambda:|-(\lambda{\omega}_0\cdot l)^2+\hat{\lambda}_{j,p}|\leq\gamma N^{-\tau_1}_0\right\},
\end{align*}
where $\hat{\lambda}_{j,p},p=1,\cdots,\mathfrak{d}_j$ are eigenvalues of $\check{P}_{N_0,0}(\Delta^2+V(\boldsymbol x))_{|\check{H}_{N_0,0}}$. Formula \eqref{E1.84} implies $\hat{\lambda}_{j,p}\geq\kappa_0$. If $\gamma N^{-\tau_1}_0<\kappa_0$, then $\mathscr{U}_{0,j,p}=\emptyset$. For $l\neq0$, letting $\xi=\lambda^2$, we have
\begin{align*}
\mathfrak{f}_{l,j,p}(\xi)=\xi(\omega_0\cdot l)^2-\hat{\lambda}_{j,p}.
\end{align*}
Using \eqref{E1.96} and $|l|\leq N_0$, we deduce
\begin{align*}
\partial_{\xi}\mathfrak{f}_{l,j,p}(\xi)=(\omega_0\cdot l)^2\geq4\gamma^2_0N^{-2\nu}_0.
\end{align*}
Thus one has
\begin{align*}
|\xi_1-\xi_2|\leq\frac{2\gamma N^{-\tau_1}_0}{4\gamma^2_0N^{-2\nu}_0}=\frac{\gamma}{2\gamma^2_0}N^{-\tau_1+2\nu}_0,
\end{align*}
which carries out
\begin{align*}
\mathrm{meas}(\mathscr{U}_{l,j,p})\leq\frac{\gamma}{2\gamma^2_0}N^{-\tau_1+2\nu}_0.
\end{align*}
Thus, if $\tau_1-3\nu-d\geq1$, then the following holds:
\begin{align*}
\mathrm{meas}(\Lambda\setminus\mathscr{U})\leq&\sum\limits_{|l|,|j|\leq N_0,p\leq \mathfrak{d}_j}\mathrm{meas}(\mathscr{U}_{l,j,p})
\leq\frac{\gamma}{2\gamma^2_0}N^{-\tau_1+2\nu}_0\mathcal{C}N^{\nu+d}_0=O(\frac{\gamma}{2\gamma^2_0}N^{-1}_0).
\end{align*}
\end{proof}
Finally, for $\epsilon_0$ small enough, we choose
\begin{align}\label{E3.78}
\gamma=\epsilon^{\frac{1}{s_2+1}}_0,\quad N_0=32\gamma^{-1}
\end{align}
to guarantee that \eqref{E3.15} is satisfied, and that $16\gamma^{-1}\geq\max\{\kappa_0^{-\frac{1}{\tau}}+1,2\|\rho\|\}$ (recall Lemma \ref{lemma19} and \eqref{E2.52}). To apply Proposition \ref{pro2} to $\mathscr{G}_{N_{k}}^{0}(u_{k-1}),k\geq1$,  we have to check
\begin{align}\label{E3.76}
\epsilon\kappa^{-1}_{0}(\|\mathcal{T}''_{N_{k}}(u_{k-1})\|_0+\|\partial_{\lambda}\mathcal{T}''_{N_{k}}(u_{k-1})\|_0)\leq c~(\text{recall} \eqref{E3.6}),
\end{align}
where $\mathcal{T}''(u)$ is defined in \eqref{E1.100}. It is easy that
\begin{align*}
\|\mathcal{T}''_{N_{k}}(u_{k-1})\|_0\stackrel{\eqref{E1.66}}{\leq}|\mathcal{T}''_{N_{k}}(u_{k-1})|_{s_0}\stackrel{\eqref{E1.22},\eqref{E1.14}}{\leq} C(s_0)(1+\|u_{k-1}\|_{s_0+\rho})\stackrel{\eqref{E3.34},(F1)_{k-1}}\leq C'(s_0),
\end{align*}
where $ \varrho=(2\nu+d+r+1)/2$ (recall Lemma \eqref{lemma11}). Moreover
\begin{align*}
\|\partial_{\lambda}\mathcal{T}''_{N_{k}}(u_{k-1})\|_0\stackrel{\eqref{E1.66},\eqref{E1.22},\eqref{E1.14}}{\leq} C(s_0)(1+\|\partial_{\lambda}u_{k-1}\|_{s_0+\rho})\stackrel{(F1)_{k-1}}\leq C'(s_0)N^{\tau_2+s_1+1}\gamma^{-1}.
\end{align*}
Hence, by \eqref{E3.15}, we get \eqref{E3.76}. Consequently, the complement of $\mathscr{D}_{\epsilon}$ in $\Lambda $ has measure
\begin{align*}
\mathrm{meas}\left(\mathscr{D}^{c}_{\epsilon}\right)\stackrel{\eqref{E3.17},\eqref{E3.20}}{=}&\mathrm{meas}
\left(\bigcup_{k=1}^{n}(\mathscr{U}_{N_{k}}(u_{k-1}))^{c}\cup\bigcup_{k=1}^{n}(\mathscr{G}_{N_{k}}^{0}(u_{k-1}))^{c}\cup\mathscr{U}^{c}\right)\\
\leq&\sum\limits_{k\geq1}\mathrm{meas}\left((\mathscr{U}_{N_{k}}(u_{k-1}))^{c}\right)+\sum\limits_{k\geq1}\mathrm{meas}\left((\mathscr{G}_{N_{k}}^{0}(u_{k-1}))^{c}\right)
+\mathrm{meas}\left(\mathscr{U}^{c}\right)\\
\stackrel{\eqref{E3.10}, \eqref{E3.77},\eqref{E3.14}}{\leq}&\sum\limits_{k\geq1}N^{-1}_{k}+\mathfrak{C}_1
\sum\limits_{k\geq1}N^{-1}_{k}+\mathfrak{C}_2\gamma\stackrel{\eqref{E3.78}}{\leq}\tilde{\mathfrak{C}}\epsilon^{\frac{1}{s_1+1}}_0=O(\gamma).
\end{align*}

\section{Appendix}\label{sec:4}

\subsection{Proof of lemma \ref{lemma3}}
Before proving the lemma,  we have to give some definitions.  We call that the site $j\in{E}$ is regular if $|-(\lambda{\omega}_0\cdot l+\theta)^2+\lambda^2_j+m|\geq\Theta$, where $A^{j}_{j}=(-(\lambda{\omega}_0\cdot l+\theta)^2+\lambda^2_j+m)\mathrm{I}_{\mathfrak{d}_j}$.
Otherwise $j$ is singular.
Let ${R},{S}$ denote
the following sets
\begin{align*}
{R}:=\{j\in{E}~|~j~\text{is}~\text{regular}\}, \quad{S}:=\{j\in{E}~|~j~\text{is}~\text{singular}\}.
\end{align*}
It is straightforward that ${E}={R}+{S}$. For fixed $l\in\mathbf{Z}^{\nu},$ $\theta\in\mathbf{R}$, let $A$ represent the following linear operator:
\begin{equation*}
(-{(\lambda{\omega}_0\cdot l+\theta)}^2)\mathrm{I}_{\mathfrak{d}_j}+\check{P}_{N_0,j_0}(\Delta^2+V(\boldsymbol x))_{|\check{H}_{N_0,j_0}}.
\end{equation*}
Abusing the notations, we wtite $A^{{E}}_{{E}}:=A,u_{{E}}:=u,h_{{E}}:=h$, where $A^{E}_{E}\in\mathcal{M}^{E}_{E},u_{{E}},h_{E}\in H^s_{E}$. Consider the Cramer system
\begin{align}\label{E1.26}
A^{{E}}_{{E}}u_{{E}}=h_{{E}}.
\end{align}
\begin{rema}
We choose $\Theta$ satisfying $\Theta^{-1}\|V\|_{s_0+\nu+r+\rho}\leq \mathfrak{c}_1(s_0)$ for some constant $\mathfrak{c}_1(s_0)>0$.
\end{rema}
\begin{proof}

Denote ${E}:=\{~j=j_0+\sum^{r}_{p=1}j_p\mathrm{\textbf{w}}_p|~j_p\in[-N,N]~\}\cap \Gamma_{+}(\boldsymbol M)$.
\\
\textbf{The~first~reduction:} There exist $M^{{E}}_{{R}},N^{{E}}_{{R}}\in\mathcal{M}^{{E}}_{{R}}$  such that
\begin{align*}
A^{{E}}_{{E}}u_{{E}}=h_{{E}} \Rightarrow u_{{R}}+M^{{E}}_{{R}}u_{{E}}=N^{{E}}_{{R}}h_{{E}}.
\end{align*}
Moreover, $M^{{E}}_{{R}},N^{{E}}_{{R}}$ satisfy \eqref{E1.30}-\eqref{E1.32}.

In fact, for $j\in {E}$ is regular, we obtain
\begin{align*}
A^{j}_{j}u_j+A^{{E}\backslash\{j\}}_ju_{{E}\backslash\{j\}}=h_j\Rightarrow u_j+ (A^{j}_{j})^{-1}A^{{E}\backslash\{j\}}_ju_{{E}\backslash\{j\}}= (A^{j}_{j})^{-1}h_j.
\end{align*}
From \eqref{E1.23}-\eqref{E1.24}, \eqref{E1.22}, it yields that
\begin{align}
|(A^{j}_{j})^{-1}A^{{E}\backslash\{j\}}_j|_{s_0+\nu+r}\leq& C|(A^{j}_{j})^{-1}|_{s_0+\nu+r}|A^{{E}\backslash\{j\}}_j|_{s_0+\nu+r}\leq C|(A^{j}_{j})^{-1}|_{s_0+\nu+r}\|{V}\|_{s_0+\nu+r+\varrho}\nonumber\\
\leq& C\Theta^{-1}\|{V}\|_{s_0+\nu+r+\varrho},\label{E1.27}\\
|(A^{j}_{j})^{-1}A^{{E}\backslash\{j\}}_j|_{s+\nu+r}\leq& \frac{1}{2}|(A^{j}_{j})^{-1}|_{s+\nu+r}|A^{{E}\backslash\{j\}}_j|_{s_0}+\frac{C(s)}{2}|(A^{j}_{j})^{-1}|_{s_0}|A^{{E}\backslash\{j\}}_j|_{s+\nu+r}\nonumber\\
\leq&C(s)(\Theta^{-1}\|{V}\|_{s_0+\nu+r+\varrho}+\Theta^{-1}\|{V}\|_{s+\nu+r+\varrho}).\label{E1.28}
\end{align}
Define
\begin{equation}\label{E1.29}
M^{j'}_{j}:=
\begin{cases}
((A^{j}_{j})^{-1}A^{{E}\backslash\{j\}}_j)^{j'}_j,\quad &j'\in {E}\backslash\{j\},\\
0,\quad&j'=j
\end{cases}
\quad\text{and}\quad
N^{j'}_{j}:=
\begin{cases}
((A^{{j}}_{{j}})^{-1})^{j'}_j,\quad &j'={j},\\
0,\quad&j'\in {E}\backslash\{j\}.
\end{cases}
\end{equation}
Then \eqref{E1.26} becomes
\begin{align*}
u_j+M^{{E}}_ju_{{E}}=N^{{E}}_jh_{{E}}.
\end{align*}
In addition, it follows from \eqref{E1.21}, \eqref{E1.27}-\eqref{E1.29} and the definition of the set ${R}$ that
\begin{align}
&|M^{{E}}_{{R}}|_{s_0}\leq K_1|M^{{E}}_j|_{s_{0}+\nu+r}\leq C'\Theta^{-1}\|{V}\|_{s_0+\nu+r+\varrho},\label{E1.30}\\
&|M^{{E}}_{{R}}|_{s}\leq K_1|M^{{E}}_j|_{s+\nu+r}\leq C'(s)(\Theta^{-1}\|V\|_{s_0+\nu+r+\varrho}+\Theta^{-1}\|{V}\|_{s+\nu+r+\varrho}),\label{E1.31}\\
&|N^{{E}}_{{R}}|_{s_0}\leq K_1|N^{{E}}_j|_{s_{0}+\nu+r}\leq K_1\Theta^{-1},\quad|N^{{E}}_{{R}}|_{s}\leq K_1|N^{{E}}_j|_{s+\nu+r}\leq K_1\Theta^{-1}.\label{E1.32}
\end{align}
\\
\textbf{The~second~reduction:} For $\Theta^{-1}\|{V}\|_{s_0+\nu+r+\rho}\leq \mathfrak{c}_1(s_0)$ small enough, there exist $\tilde{{M}}^{{S}}_{{R}}\in\mathcal{M}^{{S}}_{{R}}$, $\tilde{N}^{{E}}_{{R}}\in\mathcal{M}^{{E}}_{{R}}$ satisfying \eqref{E1.38}-\eqref{E1.40} such that
\begin{align}\label{E1.41}
A^{{E}}_{{E}}u_{{E}}=h_{{E}} \Rightarrow u_{{R}}=\tilde{{M}}^{{S}}_{{R}} u_{{S}}+\tilde{N}^{{E}}_{{R}} h_{{E}}.
\end{align}

In fact, since ${E}={R}+{S}$, then
\begin{equation*}
u_{{R}}+M^{{E}}_{{R}}u_{{E}}=N^{{E}}_{{R}}h_{{E}}\Rightarrow
u_{{R}}+M^{{R}}_{{R}}u_{{R}}+M^{{S}}_{{R}}u_{{S}}=N^{{E}}_{{R}}h_{{E}},
\end{equation*}
namely,
\begin{equation}\label{E1.35}
(\mathrm{I}^{{R}}_{{R}}+M^{{R}}_{{R}})u_{{R}}+M^{{S}}_{{R}}u_{{S}}=N^{{E}}_{{R}}h_{{E}}.
\end{equation}
For $\Theta^{-1}\|{V}\|_{s_0+\nu+r+\rho}\leq \mathfrak{c}_1(s_0)$ small enough, formula \eqref{E1.30} shows
\begin{equation*}
|(\mathrm{I}^{{R}}_{{R}})^{-1}|_{s_0}|M^{{R}}_{{R}}|_{s_0}\leq C'\Theta^{-1}\|{V}\|_{s_0+\nu+r+\varrho}\leq1/2.
\end{equation*}
 Then it follows from Lemma \ref{lemma2} that $(\mathrm{I}^{{R}}_{{R}}+M^{{R}}_{{R}})$ is invertible with
\begin{align}
&|(\mathrm{I}^{{R}}_{{R}}+M^{{R}}_{{R}})^{-1}|_{s_0}\stackrel{\eqref{E1.33}}{\leq}2,\label{E1.36}\\
|(\mathrm{I}^{{R}}_{{R}}+M^{{R}}_{{R}})^{-1}|_{s}\stackrel{\eqref{E1.92}}{\leq}&C(s)
(|(\mathrm{I}^{{R}}_{{R}})^{-1}|_{s}+|(\mathrm{I}^{{R}}_{{R}})^{-1}|_{s_0}|M^{{R}}_{{R}}|_s)\leq C''(s)(1+\Theta^{-1}\|{V}\|_{s+\nu+r+\varrho})\label{E1.37}.
\end{align}
As a consequence equation \eqref{E1.35} is reduced to
\begin{equation*}
u_{{R}}=\tilde{{M}}^{{S}}_{{R}} u_{{S}}+\tilde{N}^{{E}}_{{R}} h_{{E}}
\end{equation*}
where
\begin{equation*}
\tilde{M}^{{S}}_{{R}}=-(\mathrm{I}^{{R}}_{{R}}+M^{{R}}_{{R}})^{-1}M^{{S}}_{{R}},
\quad\tilde{N}^{{E}}_{{R}}=(\mathrm{I}^{{R}}_{{R}}+M^{{R}}_{{R}})^{-1}N^{{E}}_{{R}}.
\end{equation*}
For $\Theta^{-1}\|{V}\|_{s_0+\nu+r+\rho}\leq \mathfrak{c}_1(s_0)$ small enough, applying \eqref{E1.23}-\eqref{E1.24}, \eqref{E1.30}-\eqref{E1.32} and \eqref{E1.36}-\eqref{E1.37} yields that
\begin{align}
|\tilde{M}^{{S}}_{{R}}|_{s_0}&\leq 2C|M^{{S}}_{{R}}|_{s_0}\leq C,\label{E1.38}\\
|\tilde{M}^{{S}}_{{R}}|_{s}&\leq\frac{1}{2}|(\mathrm{I}^{{R}}_{{R}}+M^{{R}}_{{R}})^{-1}|_{s_0}
|M^{{S}}_{{R}}|_{s}+\frac{C(s)}{2}|(\mathrm{I}^{{R}}_{{R}}+M^{{R}}_{{R}})^{-1}|_{s}
|M^{{S}}_{{R}}|_{s_0}\nonumber\\
&\leq C'''(s)(1+\Theta^{-1}\|{V}\|_{s+\nu+r+\varrho}),\label{E1.39}\\
|\tilde{N}^{{E}}_{{R}}|_{s_0}&\leq C'\Theta^{-1},\quad |\tilde{N}^{{E}}_{{R}}|_{s}\leq C'''(s)\Theta^{-1}(1+\Theta^{-1}\|{V}\|_{s+\nu+r+\varrho}).\label{E1.40}
\end{align}\\
\textbf{The~third~reduction:} There exist $\hat{M}^{{S}}_{{E}}\in\mathcal{M}^{{S}}_{{E}}$, $\hat{N}^{{E}}_{{E}}\in\mathcal{M}^{{E}}_{{E}}$ satisfying \eqref{E1.46}-\eqref{E1.49} such that
\begin{align*}
A^{{E}}_{{E}}u_{{E}}=h_{{E}} \Rightarrow \hat{M}^{{S}}_{{E}}u_{{S}}=\hat{N}^{{E}}_{{E}}h_{{E}}.
\end{align*}
Furthermore, $((A^{{E}}_{{E}})^{-1})^{{E}}_{{S}}$ is a left inverse of $\hat{M}^{{S}}_{{E}}$.\\

In fact, since ${E}={R}+{S}$, for $j\in{E}$ is regular, this holds:
\begin{align}
(A^{{E}}_{{E}}u_{{E}})_{j}=h_{j}&\Rightarrow(A^{{R}}_{{E}}u_{{R}}
+A^{{S}}_{{E}}u_{{S}})_{j}=h_{j}\stackrel{\eqref{E1.41}}{\Rightarrow}
(A^{{R}}_{{E}}(\tilde{{M}}^{{S}}_{{R}} u_{{S}}+\tilde{N}^{{E}}_{{R}} h_{{E}})+A^{{S}}_{{E}}u_{{S}})_{j}=h_{j}\nonumber\\
&\Rightarrow(\hat{M}^{{S}}_{{E}}u_{{S}})_{j}=(\hat{N}^{{E}}_{{E}}h_{{E}})_{j},\label{E1.42}
\end{align}
where
\begin{equation}\label{E1.43}
\hat{M}^{{S}}_{{E}}=A^{{R}}_{{E}}\tilde{{M}}^{{S}}_{{R}}+A^{{S}}_{{E}},\quad
\hat{N}^{{E}}_{{E}}=\mathrm{I}^{{E}}_{{E}}-A^{{R}}_{{E}}\tilde{N}^{{E}}_{{R}}.
\end{equation}
Since $j\in {R}$, formula \eqref{E1.42} infers that $\hat{M}^{{S}}_{j}=0$, which then gives that $\hat{N}^{{E}}_{j}=0$. Therefore
\begin{align}
|A^{{E}}_{{E}}|_{s_0}&=|A^{{E}}_{{S}}|_{s_0}\leq|\mathrm{Diag}(A^{{E}}_{{S}})|_{s_0}
+|A^{{E}}_{{S}}-\mathrm{Diag}(A^{{E}}_{{S}})|_{s_0}\leq\Theta+C\|{V}\|_{s_0+\varrho},\label{E1.44}\\
|A^{{E}}_{{E}}|_{s}&=|A^{{E}}_{{S}}|_{s}\leq|\mathrm{Diag}(A^{{E}}_{{S}})|_{s}
+|A^{{E}}_{{S}}-\mathrm{Diag}(A^{{E}}_{{S}})|_{s}\leq\Theta+C(s)\|{V}\|_{s+\varrho}.\label{E1.45}
\end{align}
It follows from \eqref{E1.23}-\eqref{E1.24}, \eqref{E1.38}-\eqref{E1.40} and \eqref{E1.43}-\eqref{E1.45} that
\begin{align}
&|\hat{M}^{{S}}_{{E}}|_{s_0}=|\hat{M}^{{S}}_{{S}}|_{s_0}\leq C|A^{{R}}_{{S}}|_{s_0}|\tilde{{M}}^{{S}}_{{R}}|_{s_0}+|A^{{S}}_{{S}}|_{s_0}\leq C'(\Theta+\|{V}\|_{s_0+\varrho}),\label{E1.46}\\
&|\hat{M}^{{S}}_{{E}}|_{s}=|\hat{M}^{{S}}_{{S}}|_{s}\leq \frac{1}{2}|A^{{R}}_{{S}}|_{s_0}|\tilde{{M}}^{{S}}_{{R}}|_{s}+\frac{C(s)}{2}
|A^{{R}}_{{S}}|_{s}|\tilde{{M}}^{{S}}_{{R}}|_{s_0}+|A^{{S}}_{{S}}|_{s}\leq C(s,\Theta)(1+\|{V}\|_{s+\nu+r+\varrho}),\label{E1.47}\\
&|\hat{N}^{{E}}_{{E}}|_{s_0}=|\hat{N}^{{E}}_{{S}}|_{s_0}\leq|\mathrm{I}^{{E}}_{{S}}|_{s_0}+C|A^{{R}}_{{S}}|_{s_0}
|\tilde{N}^{{E}}_{{R}}|_{s_0}\leq C',\label{E1.48}\\
&|\hat{N}^{{E}}_{{E}}|_{s}=|\hat{N}^{{E}}_{{S}}|_{s}\leq|\mathrm{I}^{{E}}_{{S}}|_{s_0}+\frac{1}{2}|A^{{R}}_{{S}}|_{s_0}
|\tilde{N}^{{E}}_{{R}}|_{s}+\frac{C(s)}{2}|A^{{R}}_{{S}}|_{s}
|\tilde{N}^{{E}}_{{R}}|_{s_0}\leq C(s,\Theta)(1+\Theta^{-1}\|{V}\|_{s+\nu+r+\varrho}).\label{E1.49}
\end{align}
In addition, by \eqref{E1.43}, it is obvious that
\begin{equation*}
((A^{{E}}_{{E}})^{-1})^{{E}}_{{S}}\hat{M}^{{S}}_{{E}}=((A^{{E}}_{{E}})^{-1})^{{S}}_{{S}}\hat{M}^{{S}}_{{S}}=
((A^{{E}}_{{E}})^{-1})^{{S}}_{{S}}(A^{{R}}_{{S}}\tilde{{M}}^{{S}}_{{R}}+A^{{S}}_{{S}})=\mathrm{I}^{{S}}_{{S}}.
\end{equation*}
This shows that $((A^{{E}}_{{E}})^{-1})^{{E}}_{{S}}$ is a left inverse of $\hat{M}^{{S}}_{{E}}$.

If $\delta<1$, then there exists some constant $C(r)>0$ such that ${S}$ (the set of singular sites) admits a partition with
\begin{equation}\label{E1.60}
\mathrm{diam}(\Omega_\alpha)\leq LB\leq B^{1+C(r)}=N^{\frac{\delta}{2}},\quad\mathrm{d}(\Omega_{\alpha},\Omega_{\beta})>N^{\frac{\delta}{2(1+C(r))}}, \quad\forall\alpha\neq\beta.
\end{equation}
Therefore we have the following result:
\\
\textbf{The~final~reduction:} Define $X^{{S}}_{{E}}\in\mathcal{M}^{{S}}_{{E}}$ by
\begin{equation}\label{E1.63}
X^{j}_{j'}:=
\begin{cases}
\hat{M}^{j}_{j'}\quad&\text{if}\quad(j,j')\in\bigcup_{\alpha}(\Omega_{\alpha}\times\hat{\Omega}_{\alpha}),\\
0&\text{if}\quad(j,j')\notin\bigcup_{\alpha}(\Omega_{\alpha}\times\hat{\Omega}_{\alpha}),
\end{cases}
\end{equation}
where $\hat{\Omega}_{\alpha}:=\left\{j\in{E}: \mathrm{d}(j,\Omega_{\alpha})\leq\frac{1}{4}N^{\frac{\delta}{2(1+C(r))}}\right\}.$
The definition of $\hat{\Omega}_{\alpha}$ together with \eqref{E1.60} may indicate $\hat{\Omega}_{\alpha}\cap\hat{\Omega}_{\beta}=\emptyset,~\forall\alpha\neq\beta$.

\textbf{Step1}: Let us claim that $X^{{S}}_{{E}}\in\mathcal{M}^{{S}}_{{E}}$ has a left inverse $Y^{{E}}_{{S}}\in\mathcal{M}^{{E}}_{{S}}$ with
\begin{align}\label{E1.68}
\|Y^{{E}}_{{S}}\|_{0}\leq2N^{\tau}.
\end{align}
Define $Z^{{S}}_{{E}}:=\hat{M}^{{S}}_{{E}}-X^{{S}}_{{E}}$, which then gives that $X^{{S}}_{{E}}=\hat{M}^{{S}}_{{E}}-Z^{{S}}_{{E}}$. Definition \eqref{E1.63} implies that
\begin{align*}
Z^{j}_{j'}=0\quad\text{if}\quad \mathrm{d}(j,j')\leq{N^{\frac{\delta}{2(1+C(r))}}}/{4}.
\end{align*}
Combining this with \eqref{E1.61}, \eqref{E1.47}, \eqref{E1.99}, for all $s_1\geq s_0+\nu+r+\varrho$, we obtain
\begin{align}
|Z^{S}_{E}|_{s_0}\leq&(N^{\frac{\delta}{2(1+C(r))}}/4)^{-(s_1-\nu-r-\varrho-s_0)}|Z^{S}_{E}|_{s_1-\nu-r-\varrho}
\leq4^{s_1}(N^{\frac{\delta}{2(1+C(r))}})^{-(s_1-\nu-r-\varrho-s_0)}|\hat{M}^{{S}}_{{E}}|_{s_1-\nu-r-\varrho}\nonumber\\
\leq&4^{s_1}(N^{\frac{\delta}{2(1+C(r))}})^{-(s_1-\nu-r-\varrho-s_0)}C(s_1,\Theta)(1+\|{V}\|_{s_1})\quad\\
\leq& C'(s_1,\Theta)(N^{\frac{\delta}{2(1+C(r))}})^{-(s_1-\nu-r-\varrho-s_0)},\label{E1.64}\\
|Z^{S}_{E}|_{s}&\leq|\hat{M}^{{S}}_{{E}}|_{s}\leq C(s,\Theta)(1+\|{V}\|_{s+\nu+r+\varrho}).\label{E1.65}
\end{align}
%
In addition, for $N\geq \tilde{N}(s_1,V)$ large enough and $s_1>\frac{2(1+C(r))}{\delta}\tau+(s_0+\nu+r+\varrho)$, formulae  \eqref{E1.66} and \eqref{E1.64} establish
\begin{align*}
\|((A^{{E}}_{{E}})^{-1})^{{E}}_{{S}}\|_{0}\|Z^{S}_{E}\|_{0}
\leq\|((A^{{E}}_{{E}})^{-1})^{{E}}_{{S}}\|_{0}|Z^{S}_{E}|_{s_0}
\leq C(s_1,\Theta)N^{\tau}(N^{\frac{\delta}{2(1+C(r))}})^{-(s_1-\nu-r-\varrho-s_0)}\leq1/2.
\end{align*}
%
Lemma \ref{lemma2} verifies that $X^{S}_{E}$ has a left inverse $Y^{E}_{S}\in\mathcal{M}^{E}_{S}$ with
\begin{align*}
\|Y^{E}_{S}\|_{0}\stackrel{\eqref{E1.34}}{\leq} 2\|((A^{{E}}_{{E}})^{-1})^{{E}}_{{S}}\|_{0}\leq 2N^{\tau}.
\end{align*}

\textbf{Step2}: Define $\tilde{Y}^{E}_{S}$ by
\begin{equation}\label{E1.67}
\tilde{Y}^{j'}_{j}:=
\begin{cases}
Y^{j'}_{j}\quad&\text{if}\quad(j',j)\in\bigcup_{\alpha}(\Omega_{\alpha}\times\hat{\Omega}_{\alpha})\\
0&\text{if}\quad(j',j)\notin\bigcup_{\alpha}(\Omega_{\alpha}\times\hat{\Omega}_{\alpha}).
\end{cases}
\end{equation}
If the fact $({Y}^{E}_{S}-\tilde{Y}^{E}_{S})X^{S}_{E}=0$ holds, then $\tilde{Y}^{E}_{S}$ is a left inverse of $X^{S}_{E}$ with
\begin{align}\label{E1.69}
|\tilde{Y}^{E}_{S}|_{s}\leq C(s)N^{\frac{\delta}{2}(s+\nu+r)+\tau}.
\end{align}
Let us prove above fact. For $j\in {S}=\bigcup_{\alpha}\Omega_{\alpha}$, there is $\alpha$ such that $j\in\Omega_{\alpha}$. Moreover, since $({Y}^{E}_{S}-\tilde{Y}^{E}_{S})^{j''}_{j}=Y^{j''}_{j}-\tilde{Y}^{j''}_{j}=0$ {for} $j''\in\hat{\Omega}_{\alpha}$, we get
\begin{align*}
(({Y}^{E}_{S}-\tilde{Y}^{E}_{S})X^{S}_{E})^{j'}_{j}=\sum\limits_{j''\notin\hat{\Omega}_{\alpha}}({Y}^{E}_{S}-\tilde{Y}^{E}_{S})^{j''}_{j}X^{j'}_{j''}.
\end{align*}

If $j'\in\Omega_{\alpha}$, then definition \eqref{E1.63} implies that $X^{j'}_{j''}=0$, which shows that $(({Y}^{E}_{S}-\tilde{Y}^{E}_{S})X^{S}_{E})^{j'}_{j}=0$.

If $j'\in\Omega_{\beta}$ with $\alpha\neq\beta$, for $j''\notin\hat{\Omega}_{\beta}$, then $X^{j'}_{j''}=0$ owing to \eqref{E1.63}. As a consequence
\begin{align*}
(({Y}^{E}_{S}-\tilde{Y}^{E}_{S})X^{S}_{E})^{j'}_{j}&=\sum\limits_{j''\in\hat{\Omega}_{\beta}}({Y}^{E}_{S}-\tilde{Y}^{E}_{S})^{j''}_{j}X^{j'}_{j''}
\stackrel{\eqref{E1.67}}{=}\sum\limits_{j''\in\hat{\Omega}_{\beta}}{Y}^{j''}_{j}X^{j'}_{j''}
\stackrel{\eqref{E1.63}}{=}\sum\limits_{j''\in E}{Y}^{j''}_{j}X^{j'}_{j''}\\
&=(Y^{E}_{S}X^{S}_{E})^{j'}_{j}=(\mathrm{I}^{S}_{S})^{j'}_{j}=0.
\end{align*}
Since $\mathrm{diam}(\hat{\Omega}_{\alpha})<2N^{\frac{\delta}{2}}$ (see \eqref{E1.60}), by definition \eqref{E1.67}, we obtain that $\tilde{Y}^{j'}_{j}=0$ for all $|j-j'|\geq2N^{\frac{\delta}{2}}$. Then \eqref{E1.62} infers
\begin{align*}
|\tilde{Y}^{E}_{S}|_{s}\leq C(s)N^{\frac{\delta}{2}(s+\nu+r)}\|\tilde{Y}^{E}_{S}\|_{0}\stackrel{\eqref{E1.68},\eqref{E1.67}}{\leq}C(s)N^{\frac{\delta}{2}(s+\nu+r)+\tau}.
\end{align*}

\textbf{Step3}: For $N\geq \tilde{N}(s_1,V)$ large enough and $s_1>(1+C(r))(s_0+\nu+r)+\frac{2(1+C(r))}{\delta}\tau+(s_0+\nu+r+\varrho)$, it follows from \eqref{E1.64} and \eqref{E1.69} that
\begin{align*}
|\tilde{Y}^{E}_{S}|_{s_0}|Z^{S}_{E}|_{s_0}\leq CN^{\frac{\delta}{2}(s_0+\nu+r)+\tau}C(s_1,\Theta)(N^{\frac{\delta}{2(1+C(r))}})^{-(s_1-\nu-r-\varrho-s_0)}\leq1/2.
\end{align*}
Combining this with the equality $\hat{M}^{{S}}_{{E}}=X^{{S}}_{{E}}+Z^{{S}}_{{E}}$ and Lemma \ref{lemma2} establishes that $\hat{M}^{{S}}_{{E}}$ has a left inverse ${^{[-1]}}(\hat{M}^{{S}}_{{E}})$ with
\begin{align}
|^{[-1]}(\hat{M}^{{S}}_{{E}})|_{s_0}\stackrel{\eqref{E1.33}}{\leq}&2|\tilde{Y}^{E}_{S}|_{s_0}\stackrel{\eqref{E1.69}}{\leq}2CN^{\frac{\delta}{2}(s_{0}+\nu+r)+\tau},\label{E1.70}\\
|^{[-1]}(\hat{M}^{{S}}_{{E}})|_{s}\stackrel{\eqref{E1.92}}{\leq}&C(s)(|\tilde{Y}^{E}_{S}|_{s}+|\tilde{Y}^{E}_{S}|^2_{s_0}|Z^{{S}}_{{E}}|_{s})
\nonumber\\
\stackrel{\eqref{E1.69},\eqref{E1.65}}{\leq}& C'(s,\Theta)N^{\delta(s_0+\nu+r)+2\tau}(N^{\delta s}+\|{V}\|_{s+\nu+r+\varrho}).\label{E1.71}
\end{align}
Thus the system \eqref{E1.26} is equivalent to
\begin{equation*}
\begin{cases}
u_{R}=\tilde{{M}}^{{S}}_{{R}}({^{[-1]}}(\hat{M}^{{S}}_{{E}}))\hat{N}^{E}_{E}h_{E}+\tilde{N}^{{E}}_{{R}} h_{{E}},\\
u_{S}=({^{[-1]}}(\hat{M}^{{S}}_{{E}}))\hat{N}^{E}_{E}h_{E}.
\end{cases}
\end{equation*}
This implies that
\begin{align*}
((A^{E}_{E})^{-1})^{E}_{R}=\tilde{{M}}^{{S}}_{{R}}({^{[-1]}}\hat{M}^{{S}}_{{E}})\hat{N}^{E}_{E}+\tilde{N}^{{E}}_{{R}},\quad
((A^{E}_{E})^{-1})^{E}_{S}=({^{[-1]}}\hat{M}^{{S}}_{{E}})\hat{N}^{E}_{E}.
\end{align*}
From \eqref{E1.23}, \eqref{E1.38}-\eqref{E1.40}, \eqref{E1.48}-\eqref{E1.49} and \eqref{E1.70}-\eqref{E1.71}, it yields that
\begin{align*}
|((A^{E}_{E})^{-1})^E_{S}|_{s}&\leq C''(s,\Theta)N^{\delta(s_0+\nu+r)+2\tau}(N^{\delta s}+\|{V}\|_{s+\nu+r+\varrho}),\\
|((A^{E}_{E})^{-1})^E_{R}|_{s}&\leq C''(s,\Theta)N^{\delta(s_0+\nu+r)+2\tau}(N^{\delta s}+\|{V}\|_{s+\nu+r+\varrho}).
\end{align*}
The definition of $E$ and \eqref{E1.10} give
$\|{V}\|_{s+\nu+r+\varrho}{\leq}c^{\nu+r+\varrho}_{2}N^{\nu+r+\varrho}\|{V}\|_{s}$. Consequently, for $N\geq \tilde{N}(s_2,V)$ large enough, combining this with \eqref{E1.99} yields
\begin{align*}
|(A^E_{E})^{-1}|_{s}&\leq|((A^{E}_{E})^{-1})^E_{S}|_{s}+|((A^{E}_{E})^{-1})^E_{R}|_{s}\leq C'''(s,\Theta)N^{\delta(s_0+\nu+r)+2\tau+\nu+r+\varrho}(N^{\delta s}+\|{V}\|_s)\\
&\leq\frac{1}{4}N^{\delta(s_0+\nu+r)+2\tau+\nu+r+\varrho+1}(N^{\delta s}+\|{V}\|_s)\\
&\leq\frac{1}{2}N^{\tau_2+\delta s},
\end{align*}
where $\tau_2>\delta(s_0+\nu+r)+2\tau+\nu+r+\varrho+1=2\tau+\frac{7}{4}(\nu+r)+\frac{3}{4}(\nu+d)+\frac{1}{2}$.
\end{proof}
Let us verify that the fact \eqref{E1.60} holds.
\begin{defi}
Denote by $\{j_k,k\in[0,L]\cap\mathbf{N}\}$ a sequence of sites with $j_k\neq j_{k'}$, $\forall k\neq k'$. For ${B}\geq2$, we call $\{j_k,k\in[0,L]\cap\mathbf{N}\}$ a ${B}$-chain of length $L$ with $|j_{k+1}-j_{k}|\leq {B}$, $\forall k=0,\cdots,L-1.$
\end{defi}
\begin{lemm}
There exists $C(r)>0$ such that, for fixed $l\in\mathbf{Z}^{\nu}$, $\theta\in\mathbf{R}$, any ${B}$-chain of singular sites has length
$L\leq {B}^{C(r)}$.
\end{lemm}
\begin{proof}
Denote by $\{j_k,k\in[0,L]\cap\mathbf{N}\}$ a ${B}$-chain of singular sites. Then
\begin{align}\label{E1.51}
|j_{k+1}-j_k|\leq {B},\quad\forall k=0,\cdots,L-1.
\end{align}
Letting $\vartheta:=\lambda{\omega}_0\cdot l+\theta$, by the definitions of the singular site and $\lambda_j$, we give
\begin{align*}
&|-\vartheta^2+(\|j_{k}+\rho\|^2-\|\rho\|^2)^2+m|<\Theta\Rightarrow|-\vartheta^2+(\|j_{k}+\rho\|^2-\|\rho\|^2)^2|<\Theta+m\\
\Rightarrow&|-\vartheta+\|j_{k}+\rho\|^2-\|\rho\|^2|<\sqrt{\Theta+m}\quad\text{or}\quad|\vartheta+\|j_{k}+\rho\|^2-\|\rho\|^2|<\sqrt{\Theta+m}\\
\Rightarrow&\left|\|j_{k+1}+\rho\|^2+\|j_{k}+\rho\|^2\right|<2(\sqrt{\Theta+m}+\|\rho\|^2)\quad\text{or}\quad\\
&\left|\|j_{k+1}+\rho\|^2-\|j_{k}+\rho\|^2\right|<2(\sqrt{\Theta+m}+\|\rho\|^2),
\end{align*}
which leads to
\begin{align*}
\left| \|j_{k+1}+\rho\|^2 -\|j_{k}+\rho\|^2 \right|<2(\sqrt{\Theta+m}+\|\rho\|^2).
\end{align*}
This implies
\begin{align*}
\left|\|j_k+\rho\|^2-\|j_{k_0}+\rho\|^2\right|\leq2(\sqrt{\Theta+m}+\|\rho\|^2)|k-k_0|.
\end{align*}
Combining this with \eqref{E1.50}, \eqref{E1.51} and  the equality
\[(j_{k_0}+\rho)\cdot(j_k-j_{k_0})=\frac{1}{2}\left(\|j_k+\rho\|^2-\|j_{k_0}+\rho\|^2-\|j_k-j_{k_0}\|^2\right)\]
yields
\begin{align}\label{E1.53}
\left|(j_{k_0}+\rho)\cdot(j_k-j_{k_0})\right|&\leq(\sqrt{\Theta+m}+\|\rho\|^2)|k-k_0|+(b^2_2/2)|k-k_0|^2{B}^2\nonumber\\
&\leq(\sqrt{\Theta+m}+\|\rho\|^2+b^2_2)|k-k_0|^2{B}^2.
\end{align}
Define the following subspace of $\mathbf{R}^{r}$ by
\begin{equation*}
\mathscr{E}:=\mathrm{span}_{\mathbf{R}}\{j_{k}-j_{k'}:~k,k'=0,\cdots,L\}=\mathrm{span}_{\mathbf{R}}\{j_{k}-j_{k_0}:~k=0,\cdots,L\}.
\end{equation*}
Let ${r}_0$ be the dimension of $\mathscr{E}$. Denote by $\xi_1,\cdots,\xi_{{r}_0}$ a basis of $\mathscr{E}$. It is straightforward that ${r}_0\leq r$.

Case1. For all $k_0\in[0,L]\cap\mathbf{N}$, we have
\begin{equation*}
\mathscr{E}_{k_0}:=\mathrm{span}_{\mathbf{R}}\{j_{k}-j_{k_0}:~|k-k_0|\leq L^{\upsilon},~ k=0,\cdots,L\}=\mathscr{E}.
\end{equation*}
Formula \eqref{E1.51} indicates that
\begin{align}\label{E1.54}
|\xi_{p}|=|j_p-j_{k_0}|\leq|p-k_0|{B}\leq L^{\upsilon}{B},\quad p=1,\cdots,r_0.
\end{align}
Let $\Pi_{\mathscr{E}}$ denote the orthogonal projection on $\mathscr{E}$. Then
\begin{equation}\label{E1.58}
\Pi_{\mathscr{E}}(j_{k_0}+\rho)=\sum\limits_{p=1}^{r_0}z_p\xi_p
\end{equation}
for some $z_p\in\mathbf{R},p=1,\cdots,r_0$. Hence we get
\begin{equation*}
\Pi_{\mathscr{E}}(j_{k_0}+\rho)\cdot \xi_{p'}=\sum\limits_{p=1}^{r_0}z_p\xi_p\cdot \xi_{p'}.
\end{equation*}
Based on above fact, we consider the linear system $Qz=y$, where
\begin{align*}
Q=(Q_{pp'})_{p,p'=1,\cdots,r_0}\quad\text{with}\quad Q_{pp'}=\xi_p\cdot\xi_{p'},\quad y_{p'}=\Pi_{\mathscr{E}}(j_{k_0}+\rho)\cdot \xi_{p'}=(j_{k_0}+\rho)\cdot \xi_{p'}.
\end{align*}
It follows from \eqref{E1.53}-\eqref{E1.54} that
\begin{align}\label{E1.55}
|y_{p'}|\leq(\sqrt{\Theta+m}+\|\rho\|^2+b^2_2)(L^{\upsilon}B)^2, \quad|Q_{pp'}|\leq(L^{\upsilon}{B})^2.
\end{align}
Moreover formula \eqref{E1.56} verifies
\begin{align}\label{E1.57}
\mathfrak{z}^{r_0}\det(Q)\in\mathbf{Z},\quad\text{namely}\quad\mathfrak{z}^{r_0}|\det(Q)|\geq1.
\end{align}
Hadamard inequality gives
\begin{equation*}
\left|(Q^*)_{pp'}\right|\leq\prod\limits_{\mathfrak{p}\neq p,1\leq\mathfrak{p}\leq r_0}\Big(\sum\limits_{\mathfrak{p}'\neq p',1\leq\mathfrak{p}'\leq r_0}|Q_{\mathfrak{p}\mathfrak{p}'}|^2\Big)^{1/2},
\end{equation*}
where $Q^*$ is the adjoint matrix of $Q$. This establishes that
\begin{align*}
\left|(Q^*)_{pp'}\right|\leq(r_0-1)^{\frac{r_0-1}{2}}{(L^{\upsilon}{B})}^{2(r_0-1)}.
\end{align*}
Based on this and Cramer's rule,  \eqref{E1.55}-\eqref{E1.57}, we obtain that
\begin{align*}
|z_p|\leq\sum\limits_{p'=1}^{r_0}|Q^{-1}_{pp'}y_{p'}|\leq \mathfrak{z}^{r_0}{r_0}^{r_0}(\sqrt{\Theta+m}+\|\rho\|^2+b^2_2)
(L^{\upsilon}{B})^{2r_0}.
\end{align*}
Combining this with formulae \eqref{E1.54}-\eqref{E1.58}, we derive
\begin{align*}
|\Pi_{\mathscr{E}}(j_{k_0}+\rho)|\leq r_0|z_p||\xi_{p}|\leq \mathfrak{z}^{r_0}{r_0}^{r_0+1}(\sqrt{\Theta+m}+\|\rho\|^2+b^2_2)
(L^{\upsilon}{B})^{2r_0+1}.
\end{align*}
As a consequence
\begin{align*}
|j_{k_1}-j_{k_2}|&=|(j_{k_1}-j_{k_0})-(j_{k_2}-j_{k_0})|=|\Pi_{\mathscr{E}}(j_{k_1}+\rho)-\Pi_{\mathscr{E}}(j_{k_2}+\rho)|\\
&\leq2\mathfrak{z}^{r}r^{r+1}(\sqrt{\Theta+m}+\|\rho\|^2+b^2_2)(L^{\upsilon}{B})^{2r+1},
\end{align*}
which then implies
\begin{align}\label{E1.59}
L\leq4^r(2\mathfrak{z}^{r}r^{r+1}(\sqrt{\Theta+m}+\|\rho\|^2+b^2_2)(L^{\upsilon}{B})^{2r+1})^r.
\end{align}
If $\upsilon<\frac{1}{2r(2r+1)}$, then \eqref{E1.59} yields that
\begin{align*}
L^{\frac{1}{2}}&\leq2^{3r}\mathfrak{z}^{r^2}r^{r(r+1)}(\sqrt{\Theta+m}+\|\rho\|^2+b^2_2)^{r}{B}^{r(2r+1)}\\
&\Rightarrow
L\leq2^{6r}\mathfrak{z}^{2r^2}r^{2r(r+1)}(\sqrt{\Theta+m}+\|\rho\|^2+b^2_2)^{2r}{B}^{2r(2r+1)}.
\end{align*}

Case2. If there exists some $k'_0\in[0,L]\cap\mathbf{N}$ such that $\dim\mathscr{E}_{k'_0}\leq r-1$, for $k_0\in {\mathfrak{J}}$, then we consider
\begin{align*}
\mathscr{E}^1_{k_0}:=\mathrm{span}_{\mathbf{R}}\left\{j_{k}-j_{k_0}:|k-k_0|<L^\upsilon_1,~k \in {\mathfrak{J}}\right\}=\mathrm{span}_{\mathbf{R}}\left\{j_{k}-j_{k_0}:~k \in {\mathfrak{J}}\right\},
\end{align*}
where
\[L_1=L^{\upsilon}, \quad{\mathfrak{J}}:=\{k:|k-k'_0|<L^{\upsilon},~ k=0,\cdots,L\}\cap([0,L]\cap\mathbf{N}).\]
The upper bound of $L_1$ can be proved by the same method as employed on $L$, namely
\begin{align*}
L_1=L^\upsilon\leq2^{6r}\mathfrak{z}^{2r^2}r^{2r(r+1)}(\sqrt{\Theta+m}+\|\rho\|^2+b^2_2)^{2r}{B}^{2r(2r+1)}.
\end{align*}
The fact $r_0\leq r$ leads to that the iteration is carried out at most $r$ steps. Thus
\begin{align*}
L_r=L^{r\upsilon}\leq2^{6r}\mathfrak{z}^{2r^2}r^{2r(r+1)}(\sqrt{\Theta+m}+\|\rho\|^2+b^2_2)^{2r}{B}^{2r(2r+1)}\Rightarrow L\leq {B}^{C(r)}
\end{align*}
for some constant $C(r)>0$.
Let ${B}=N^{\frac{\delta}{2(1+C(r))}}$.
\begin{defi}
We say that $\mathfrak{x}\equiv \mathfrak{y}$ if there is a $N^{\frac{\delta}{2(1+C(r))}}$-chain $\{j_k,k\in[0,L]\cap\mathbf{N}\}$ connecting $\mathfrak{x}$ to $\mathfrak{y}$, namely,
$j_0=\mathfrak{x},j_L=\mathfrak{y}$.
\end{defi}
The equivalence relation induces that a partition of ${S}$ satisfies
\begin{equation*}
\mathrm{diam}(\Omega_\alpha)\leq L{B}\leq {B}^{1+C(r)}=N^{\frac{\delta}{2}},\quad\mathrm{d}(\Omega_{\alpha},\Omega_{\beta})>N^{\frac{\delta}{2(1+C(r))}}, \quad\forall\alpha\neq\beta.
\end{equation*}
\end{proof}

\subsection{Proof of Proposition \ref{pro1}}

For $\mathcal{A}\in\mathcal{M}^{\mathfrak{A}}_{\mathfrak{A}}$, define
\begin{align*}
\mathrm{Diag}(\mathcal{A}):=(\delta_{\mathfrak{n}\mathfrak{n}'}\mathcal{A}^{\mathfrak{n}'}_{\mathfrak{n}})_{\mathfrak{n},\mathfrak{n}'\in \mathfrak{A}}.
\end{align*}
Denote by ${\mathfrak{G}},{\mathfrak{B}}$
the following sets
\begin{align*}
{\mathfrak{G}}:=\{j\in\mathfrak{{A}}~|~j~\text{is}~(\mathcal{A},N)\text{-good}\}, \quad{\mathfrak{B}}:=\{j\in\mathfrak{{A}}~|~j~\text{is}~(\mathcal{A},N)\text{-bad}\}.
\end{align*}
It is clear that $\mathfrak{A}={\mathfrak{G}}\cup \mathfrak{B}$. Moreover $\mathfrak{G}=\bar{\mathfrak{R}}\cup \mathfrak{R}$, where
\[\bar{\mathfrak{{R}}}:=\{j\in{\mathfrak{G}}~|~j~\text{is}~(\mathcal{A},N)\text{-regular}\},\quad \mathfrak{R}:=\{j\in{\mathfrak{G}}~|~j~\text{is}~\text{regular}\}.\]
\begin{proof}
Abusing the notations, we wtite $\mathcal{A}^{\mathfrak{A}}_{\mathfrak{A}}:=\mathcal{A},u_{\mathfrak{A}}:=u,h_{\mathfrak{A}}:=h$, where $\mathcal{A}^{\mathfrak{A}}_{\mathfrak{A}}\in\mathcal{M}^{\mathfrak{A}}_{\mathfrak{A}},u_{\mathfrak{A}},h_{\mathfrak{A}}\in H^s_{\mathfrak{A}}$. Consider the following Cramer system
\begin{align}\label{E2.5}
\mathcal{A}^{\mathfrak{{A}}}_{\mathfrak{{A}}}u_{\mathfrak{{A}}}=h_{\mathfrak{{A}}}.
\end{align}
\textbf{The~first~reduction:} For $N\geq \bar{N}(\tilde{\Theta},\Upsilon,s_1)$ large enough, there exist $\mathcal{P}^{\mathfrak{A}}_{\mathfrak{{G}}},\mathcal{S}^{\mathfrak{A}}_{\mathfrak{{G}}}\in\mathcal{M}^{\mathfrak{A}}_{\mathfrak{{G}}}$ with
\begin{align}\label{E2.9}
|\mathcal{P}^\mathfrak{A}_{\mathfrak{G}}|_{s_0}\leq C(s_1)\tilde{\Theta}^{-1}\Upsilon,\quad|\mathcal{S}^\mathfrak{A}_{\mathfrak{G}}|_{s_0}\leq N^{\mathfrak{e}},
\end{align}
and, for all $s\geq s_0$,
\begin{align}\label{E2.10}
|\mathcal{P}^\mathfrak{A}_{\mathfrak{G}}|_{s}\leq C(s)N^{\mathfrak{e}}(N^{s-s_0}+N^{-(\nu+r)}|\mathcal{Q}|_{s+\nu+r}),\quad|\mathcal{S}^\mathfrak{A}_{\mathfrak{G}}|_{s}\leq C(s)N^{\mathfrak{e}+s-s_0}
\end{align}
such that
\begin{align}\label{E2.11}
\mathcal{A}^{\mathfrak{A}}_{\mathfrak{A}}u_{\mathfrak{A}}=h_{\mathfrak{A}} \Rightarrow u_{\mathfrak{G}}+\mathcal{P}^{\mathfrak{A}}_{\mathfrak{G}}u_{\mathfrak{A}}=\mathcal{S}^{\mathfrak{A}}_{\mathfrak{G}}h_{\mathfrak{{A}}}.
\end{align}

In fact, since $\mathfrak{n}\in {\mathfrak{A}}$ is ($\mathcal{A},N$)-regular, there exist $\mathfrak{F}\subset\mathfrak{N}$ with $\mathrm{diam}(\mathfrak{F})\leq 4N$, $\mathrm{d}(\mathfrak{n},\mathfrak{A}\backslash \mathfrak{F})\geq N$ such that $\mathcal{A}^{\mathfrak{F}}_{\mathfrak{F}} $ is $N$-good. Then we have
\begin{align*}
\mathcal{A}^{\mathfrak{A}}_{\mathfrak{A}}u_{\mathfrak{A}}=h_{\mathfrak{A}}\Rightarrow \mathcal{A}^{\mathfrak{F}}_{\mathfrak{F}}u_{\mathfrak{F}}+\mathcal{A}^{\mathfrak{A}\backslash{\mathfrak{F}}}_{\mathfrak{F}}u_{\mathfrak{A}\backslash{\mathfrak{F}}}
=h_{\mathfrak{F}}\Rightarrow u_{\mathfrak{F}}+ (\mathcal{A}^{\mathfrak{F}}_{\mathfrak{F}})^{-1}\mathcal{A}^{\mathfrak{A}\backslash{\mathfrak{F}}}_{\mathfrak{F}}u_{\mathfrak{A}\backslash{\mathfrak{F}}}= (\mathcal{A}^{\mathfrak{F}}_{\mathfrak{F}})^{-1}h_{\mathfrak{F}}.
\end{align*}
From \eqref{E1.24}, \eqref{E2.2} and $(\mathrm{A1})$, it yields that
\begin{align}
|(\mathcal{A}^{\mathfrak{F}}_{\mathfrak{F}})^{-1}\mathcal{A}^{{\mathfrak{A}}\backslash{\mathfrak{F}}}_{\mathfrak{F}}|_{s_1-\varrho}\leq& C(s_1)|(\mathcal{A}^{\mathfrak{F}}_{\mathfrak{F}})^{-1}|_{s_1-\varrho}|\mathcal{Q}|_{s_1-\varrho}\leq C(s_1)N^{\tau_2+\delta(s_1-\varrho)}\Upsilon.\label{E2.3}
\end{align}
The fact $\mathrm{diam}(\mathfrak{F})\leq4N$ shows that $((A^{\mathfrak{F}}_{\mathfrak{F}})^{-1})^{\mathfrak{n}'}_{\mathfrak{n}}=0$ for all $|\mathfrak{n}'-\mathfrak{n}|>4N$. Combining this with \eqref{E1.23}, \eqref{E1.62}, \eqref{E2.2} and $(\mathrm{A1})$ verifies
\begin{align}
|(\mathcal{A}^{\mathfrak{F}}_{\mathfrak{F}})^{-1}\mathcal{A}^{{\mathfrak{A}}\backslash{\mathfrak{F}}}_{\mathfrak{F}}|_{s+\nu+r}\leq& \frac{1}{2}|(\mathcal{A}^{\mathfrak{F}}_{\mathfrak{F}})^{-1}|_{s_0}
|\mathcal{Q}|_{s+\nu+r}+ \frac{C(s)}{2}|(\mathcal{A}^{\mathfrak{F}}_{\mathfrak{F}})^{-1}|_{s+\nu+r}|\mathcal{Q}|_{s_0}\nonumber\\
\leq&\frac{1}{2}
|(\mathcal{A}^{\mathfrak{F}}_{\mathfrak{F}})^{-1}|_{s_0}|\mathcal{Q}|_{s+\nu+r}+\frac{C(s)}{2}4^{s+\nu+r-s_0}N^{s+\nu+r-s_0}|(\mathcal{A}^{\mathfrak{F}}_{\mathfrak{F}})^{-1}|_{s_0}|\mathcal{Q}|_{s_0}\nonumber\\
\leq&\frac{1}{2}N^{\tau_2+\delta s_0}|\mathcal{Q}|_{s+\nu+r}+\frac{C(s)}{2}4^{s+\nu+r-s_0}N^{s+\nu+r-s_0}N^{\tau_2+\delta s_0}\Upsilon\nonumber\\
\leq& C'(s)N^{(\delta-1)s_0}(N^{s+\nu+r+\tau_2}\Upsilon+N^{\tau_2+s_0}|\mathcal{Q}|_{s+\nu+r}).\label{E2.4}
\end{align}
Define
\begin{equation}\label{E2.8}
\mathcal{P}^{\mathfrak{n}'}_{\mathfrak{n}}:=
\begin{cases}
((\mathcal{A}^{\mathfrak{F}}_{\mathfrak{F}})^{-1}\mathcal{A}^{\mathfrak{A}\backslash \mathfrak{F}}_{\mathfrak{F}})^{\mathfrak{n}'}_{\mathfrak{n}},\quad &\mathfrak{n}'\in {\mathfrak{A}}\backslash{\mathfrak{F}},\\
0,\quad&\mathfrak{n}'\in \mathfrak{F}
\end{cases}
\quad\text{and}\quad
\mathcal{S}^{\mathfrak{n}'}_{\mathfrak{n}}:=
\begin{cases}
((\mathcal{A}^{{\mathfrak{F}}}_{{\mathfrak{F}}})^{-1})^{\mathfrak{n}'}_\mathfrak{n},\quad &\mathfrak{n}'\in \mathfrak{F},\\
0,\quad&\mathfrak{n}'\in {\mathfrak{A}}\backslash{\mathfrak{F}}.
\end{cases}
\end{equation}
Then \eqref{E2.5} becomes
\begin{align*}
u_{\mathfrak{n}}+\sum\limits_{\mathfrak{n}'\in \mathfrak{A}}\mathcal{P}^{\mathfrak{n}'}_{\mathfrak{n}}u_{\mathfrak{n}}=\sum\limits_{\mathfrak{n}'\in \mathfrak{A}}\mathcal{S}^{\mathfrak{n}'}_{\mathfrak{n}}h_{\mathfrak{n}'}.
\end{align*}
Since $\mathrm{d}(\mathfrak{n},\mathfrak{A}\backslash \mathfrak{F})\geq N$, we have that $\mathcal{P}^{\mathfrak{n}'}_{\mathfrak{n}}=0$ for $|\mathfrak{n}-\mathfrak{n}'|\leq N$. Hence, by means of \eqref{E1.61}, \eqref{E2.3}, \eqref{E2.8}, for $s_1>\frac{1}{1-\delta}(\tau_2+\nu+r+s_0)+\varrho$, we get that, for $N\geq \bar{N}(\tilde{\Theta},s_1)$ large enough,
\begin{align*}
|\mathcal{P}^{\mathfrak{A}}_{\mathfrak{n}}|_{s_0+\nu+r}\leq N^{-(s_1-s_0-\nu-r-\varrho)}|\mathcal{P}^{\mathfrak{A}}_{\mathfrak{n}}|_{s_1-\varrho}{\leq} C(s_1)\Upsilon N^{(\delta-1)s_1+\tau_2+\nu+r+s_0+(1-\delta)\varrho}\leq C(s_1)\tilde{\Theta}^{-1}\Upsilon,
\end{align*}
which leads to
\begin{align*}
|\mathcal{P}^\mathfrak{A}_{\bar{\mathfrak{R}}}|_{s_0}\stackrel{\eqref{E1.21}}{\leq}K_1|\mathcal{P}^{\mathfrak{A}}_{\mathfrak{n}}|_{s_0+\nu+r}\leq C'(s_1)\tilde{\Theta}^{-1}\Upsilon.
\end{align*}
Letting $\mathfrak{e}:=\tau_2+\nu+r+s_0$, for $N\geq \bar{N}(\Upsilon,s_1)$ large enough, we get
\begin{align*}
|\mathcal{P}^{\mathfrak{A}}_{\mathfrak{n}}|_{s+\nu+r}\stackrel{\eqref{E2.8}}{\leq}|(\mathcal{A}^{\mathfrak{F}}_{\mathfrak{F}})^{-1}
A^{{\mathfrak{A}}\backslash{\mathfrak{F}}}_{\mathfrak{F}}|_{s+\nu+r}\stackrel{\eqref{E2.4}}\leq& C'(s)N^{(\delta-1)s_0}(N^{s+\nu+r+\tau_2}\Upsilon+N^{\tau_2+s_0}|\mathcal{Q}|_{s+\nu+r})\nonumber\\
\leq &C'(s)N^{\mathfrak{e}}(N^{s-s_0}+N^{-(\nu+r)}|\mathcal{Q}|_{s+\nu+r}),
\end{align*}
which carries out
\begin{align*}
|\mathcal{P}^\mathfrak{A}_{\bar{\mathfrak{R}}}|_{s}\stackrel{\eqref{E1.21}}{\leq}K_1|\mathcal{P}^{\mathfrak{A}}_{\mathfrak{n}}|_{s+\nu+r}\leq C''(s)N^{\mathfrak{e}}(N^{s-s_0}+N^{-(\nu+r)}|\mathcal{Q}|_{s+\nu+r}).
\end{align*}
In addition, definition \eqref{E2.8} gives that $\mathcal{S}^{\mathfrak{n}'}_{\mathfrak{n}}=0$ for $|\mathfrak{n}-\mathfrak{n}'|>4N$. As a consequence
\begin{align*}
|\mathcal{S}^\mathfrak{A}_{\bar{\mathfrak{R}}}|_{s_0}\leq K_1|\mathcal{S}^{\mathfrak{A}}_{\mathfrak{n}}|_{s_0+\nu+r}\leq K_1|(\mathcal{A}^{\mathfrak{F}}_{\mathfrak{F}})^{-1}|_{s_0+\nu+r}\stackrel{\eqref{E2.2}}{\leq}K_1N^{\tau_2+\delta( s_0+\nu+r)}\leq N^{\mathfrak{e}},
\end{align*}
which gives
\begin{align*}
|\mathcal{S}^\mathfrak{A}_{\bar{\mathfrak{R}}}|_{s}\leq K_1|\mathcal{S}^{\mathfrak{A}}_{\mathfrak{n}}|_{s+\nu+r}\stackrel{\eqref{E1.62}}{\leq}K_1(4N)^{s-s_0}|\mathcal{S}^{\mathfrak{A}}_{\mathfrak{n}}|_{s_0+\nu+r}\leq C(s)N^{\mathfrak{e}+s-s_0}.
\end{align*}
If $\mathfrak{n}\in {\mathfrak{A}}$ is regular, then a similar argument as the first reduction shown in the proof of  Lemma \ref{lemma3} yields
\begin{align*}
&|\mathcal{P}^{{\mathfrak{A}}}_{{\mathfrak{R}}}|_{s_0}\leq K_1|P^{{\mathfrak{A}}}_\mathfrak{n}|_{s_{0}+\nu+r}\leq C'\tilde{\Theta}^{-1}|\mathcal{Q}|_{s_0+\nu+r}\leq C'\tilde{\Theta}^{-1}\Upsilon,\\
&|\mathcal{P}^{{\mathfrak{A}}}_{{\mathfrak{R}}}|_{s}\leq K_1|\mathcal{P}^{{\mathfrak{A}}}_\mathfrak{n}|_{s+\nu+r}\leq C'(s)\tilde{\Theta}^{-1}(|\mathcal{Q}|_{s_0+\nu+r}+|\mathcal{Q}|_{s+\nu+r}),\\
&|\mathcal{S}^{{\mathfrak{A}}}_{{\mathfrak{R}}}|_{s_0}\leq K_1|\mathcal{S}^{{\mathfrak{A}}}_\mathfrak{n}|_{s_{0}+\nu+r}\leq K_1\tilde{\Theta}^{-1},\quad|\mathcal{S}^{{\mathfrak{A}}}_{{\mathfrak{R}}}|_{s}\leq K_1|\mathcal{S}^{{\mathfrak{A}}}_\mathfrak{n}|_{s+\nu+r}\leq K_1\tilde{\Theta}^{-1}.
\end{align*}
Thus formulae \eqref{E2.9}-\eqref{E2.10} hold.
\\
\textbf{The~second~reduction:} If  $\tilde{\Theta}$ is large enough subject to $\Upsilon$, then there exist $\tilde{\mathcal{P}}^{\mathfrak{B}}_{\mathfrak{G}}\in\mathcal{M}^{\mathfrak{B}}_{\mathfrak{G}}$, $\tilde{\mathcal{S}}^{\mathfrak{A}}_{\mathfrak{G}}\in\mathcal{M}^{\mathfrak{A}}_{\mathfrak{G}}$ with
\begin{align}\label{E2.14}
|\tilde{\mathcal{P}}^{\mathfrak{B}}_{\mathfrak{G}}|_{s_0}&\leq C(s_1)\Upsilon\tilde{\Theta}^{-1},\quad |\tilde{\mathcal{S}}^{\mathfrak{A}}_{\mathfrak{G}}|_{s_0}\leq C(s_1)N^{\mathfrak{e}},
\end{align}
and, for all $s\geq s_0$,
\begin{align}\label{E2.15}
|\tilde{\mathcal{P}}^{\mathfrak{B}}_{\mathfrak{G}}|_{s}&\leq C(s)N^{\mathfrak{e}}(N^{s-s_0}+N^{-(\nu+r)}|\mathcal{Q}|_{s+\nu+r}),~ |\tilde{\mathcal{S}}^{\mathfrak{A}}_{\mathfrak{G}}|_{s}\leq C(s)N^{2\mathfrak{e}}(N^{s-s_0}+N^{-(\nu+r)}|\mathcal{Q}|_{s+\nu+r}),
\end{align}
 such that
\begin{align}\label{E2.16}
\mathcal{A}^{\mathfrak{A}}_{\mathfrak{A}}u_{\mathfrak{A}}=h_{\mathfrak{A}} \Rightarrow u_{\mathfrak{G}}=\widetilde{\mathcal{P}}^{\mathfrak{B}}_{\mathfrak{G}} u_{\mathfrak{B}}+\widetilde{\mathcal{S}}^{\mathfrak{A}}_{\mathfrak{G}} h_{\mathfrak{A}}.
\end{align}

In fact, by means of the fact $\mathfrak{A}=\mathfrak{G}+\mathfrak{B}$ and formula \eqref{E2.11}, we infers
\begin{equation*}
u_{\mathfrak{G}}+\mathcal{P}^{\mathfrak{A}}_{\mathfrak{G}}u_{\mathfrak{A}}=\mathcal{S}^{\mathfrak{A}}_{\mathfrak{G}}h_{\mathfrak{A}}\Rightarrow
u_{\mathfrak{G}}+\mathcal{P}^{\mathfrak{G}}_{\mathfrak{G}}u_{\mathfrak{G}}+\mathcal{P}^{\mathfrak{B}}_{\mathfrak{G}}u_{\mathfrak{B}}
=\mathcal{S}^{\mathfrak{A}}_{\mathfrak{G}}h_{\mathfrak{A}},
\end{equation*}
namely,
\begin{equation}\label{E2.17}
(\mathrm{I}^{\mathfrak{G}}_{\mathfrak{G}}+\mathcal{P}^{\mathfrak{G}}_{\mathfrak{G}})u_{\mathfrak{G}}+
\mathcal{P}^{\mathfrak{B}}_{\mathfrak{G}}u_{\mathfrak{B}}=\mathcal{S}^{\mathfrak{A}}_{\mathfrak{G}}h_{\mathfrak{A}}.
\end{equation}
If  $\tilde{\Theta}$ is large enough subject to $\Upsilon$,  then we have
\begin{equation*}
|(\mathrm{I}^{\mathfrak{G}}_{\mathfrak{G}})^{-1}|_{s_0}|\mathcal{P}^{\mathfrak{G}}_{\mathfrak{G}}|_{s_0}\stackrel{\eqref{E2.9}}{\leq} 1/2.
\end{equation*}
Hence Lemma \ref{lemma2} gives that $\mathrm{I}^{\mathfrak{G}}_{\mathfrak{G}}+\mathcal{P}^{\mathfrak{G}}_{\mathfrak{G}}$ is invertible with
\begin{align}
&|(\mathrm{I}^{\mathfrak{G}}_{\mathfrak{G}}+\mathcal{P}^{\mathfrak{G}}_{\mathfrak{G}})^{-1}|_{s_0}\stackrel{\eqref{E1.33}}{\leq}2,\label{E2.12}\\
|(\mathrm{I}^{\mathfrak{G}}_{\mathfrak{G}}+\mathcal{P}^{\mathfrak{G}}_{\mathfrak{G}})^{-1}|_{s}\stackrel{\eqref{E1.92}}{\leq}&
C(s)(1+|\mathcal{P}^{\mathfrak{G}}_{\mathfrak{G}}|_s)\leq C'(s)N^{\mathfrak{e}}(N^{s-s_0}+N^{-(\nu+r)}|\mathcal{Q}|_{s+\nu+r})\label{E2.13}.
\end{align}
As a consequence equation \eqref{E2.17} is reduced to
\begin{equation*}
u_{\mathfrak{G}}=\tilde{\mathcal{P}}^{\mathfrak{B}}_{\mathfrak{G}} u_{\mathfrak{B}}+\tilde{\mathcal{S}}^{\mathfrak{A}}_{\mathfrak{G}} h_{\mathfrak{A}}
\end{equation*}
where
\begin{equation*}
\tilde{\mathcal{P}}^{\mathfrak{B}}_{\mathfrak{G}}=-(\mathrm{I}^{\mathfrak{G}}_{\mathfrak{G}}+
\mathcal{P}^{\mathfrak{G}}_{\mathfrak{G}})^{-1}\mathcal{P}^{\mathfrak{B}}_{\mathfrak{G}},
\quad\tilde{\mathcal{S}}^{\mathfrak{A}}_{\mathfrak{G}}=(\mathrm{I}^{\mathfrak{G}}_{\mathfrak{G}}
+\mathcal{P}^{\mathfrak{G}}_{\mathfrak{G}})^{-1}\mathcal{S}^{\mathfrak{A}}_{\mathfrak{G}}.
\end{equation*}
Hence, due to \eqref{E1.23}-\eqref{E1.24}, \eqref{E2.9}-\eqref{E2.10} and \eqref{E2.12}-\eqref{E2.13}, we get that \eqref{E2.14}-\eqref{E2.15} hold.
\\
\textbf{The~third~reduction:} There exist $\hat{\mathcal{P}}^{\mathfrak{B}}_{\mathfrak{A}}\in\mathcal{M}^{\mathfrak{B}}_{\mathfrak{A}}$, $\hat{\mathcal{S}}^{\mathfrak{A}}_{\mathfrak{A}}\in\mathcal{M}^{\mathfrak{A}}_{\mathfrak{A}}$ with
\begin{align}\label{E2.18}
|\hat{\mathcal{P}}^{\mathfrak{B}}_{\mathfrak{A}}|_{s_0}&\leq C(s_1,\tilde{\Theta}),\quad |\hat{\mathcal{S}}^{\mathfrak{A}}_{\mathfrak{A}}|_{s_0}\leq C(s_1)N^{\mathfrak{e}},
\end{align}
and, for all $s\geq s_0$,
\begin{equation}\label{E2.19}
\begin{aligned}
&|\hat{\mathcal{P}}^{\mathfrak{B}}_{\mathfrak{A}}|_{s}\leq C(s,\tilde{\Theta})N^{\mathfrak{e}}(N^{s-s_0}+N^{-(\nu+r)}|\mathcal{Q}|_{s+\nu+r}),\\ &|\hat{\mathcal{S}}^{\mathfrak{A}}_{\mathfrak{A}}|_{s}\leq C(s,\tilde{\Theta})N^{2\mathfrak{e}}(N^{s-s_0}+N^{-(\nu+r)}|\mathcal{Q}|_{s+\nu+r}),
\end{aligned}
\end{equation}
such that
\begin{align}\label{E2.20}
\mathcal{A}^{\mathfrak{A}}_{\mathfrak{A}}u_{\mathfrak{A}}=h_{\mathfrak{A}} \Rightarrow \hat{\mathcal{P}}^{\mathfrak{B}}_{\mathfrak{A}}u_{\mathfrak{B}}=\hat{\mathcal{S}}^{\mathfrak{A}}_{\mathfrak{A}}h_{\mathfrak{A}}.
\end{align}
Furthermore $((\mathcal{A}^{\mathfrak{A}}_{\mathfrak{A}})^{-1})^{\mathfrak{A}}_{\mathfrak{B}}$ is a left inverse of $\hat{\mathcal{P}}^{\mathfrak{B}}_{\mathfrak{A}}$.\\

In fact, with the help of the quality $\mathfrak{A}=\mathfrak{G}+\mathfrak{B}$, this holds:
\begin{align*}
\mathcal{A}^{\mathfrak{A}}_{\mathfrak{A}}u_{\mathfrak{A}}=h_{\mathfrak{A}}\Rightarrow \mathcal{A}^{\mathfrak{G}}_{\mathfrak{A}}u_{\mathfrak{G}}
+\mathcal{A}^{\mathfrak{B}}_{\mathfrak{A}}u_{\mathfrak{B}}=h_{\mathfrak{A}}.
\end{align*}
Combining this with \eqref{E2.16} leads to
\begin{align*}
\mathcal{A}^{\mathfrak{G}}_{\mathfrak{A}}(\tilde{\mathcal{P}}^{\mathfrak{B}}_{\mathfrak{G}} u_{\mathfrak{B}}+\tilde{S}^{\mathfrak{A}}_{\mathfrak{G}} h_{\mathfrak{A}})+\mathcal{A}^{\mathfrak{B}}_{\mathfrak{A}}u_{\mathfrak{B}}=h_{\mathfrak{A}}, \quad\text{namely}\quad
\hat{\mathcal{P}}^{\mathfrak{B}}_{\mathfrak{A}}u_{\mathfrak{B}}=\hat{\mathcal{S}}^{\mathfrak{A}}_{\mathfrak{A}}h_{\mathfrak{A}},
\end{align*}
where
\begin{equation*}
\hat{\mathcal{P}}^{\mathfrak{B}}_{\mathfrak{A}}=\mathcal{A}^{\mathfrak{G}}_{\mathfrak{A}}\tilde{\mathcal{P}}^{\mathfrak{B}}_{\mathfrak{G}}
+\mathcal{A}^{\mathfrak{B}}_{\mathfrak{A}},\quad
\hat{\mathcal{S}}^{\mathfrak{A}}_{\mathfrak{A}}=\mathrm{I}^{\mathfrak{A}}_{\mathfrak{A}}-
\mathcal{A}^{\mathfrak{G}}_{\mathfrak{A}}\tilde{\mathcal{S}}^{\mathfrak{A}}_{\mathfrak{G}}.
\end{equation*}
Since formulae \eqref{E2.18}-\eqref{E2.19} are proved in the similar way as shown in the proof of Lemma \ref{lemma3} (see the third reduction),  the detail is omitted.
Moreover if there exists some constant $C_{1}:=C_1(\nu,d,r)\geq2$ such that ${\mathfrak{B}}$ (the set of ($\mathcal{A},N$)-bad sites) admits a partition with
\begin{equation*}
\mathrm{diam}(\mathfrak{O}_\alpha)\leq N^{C_1},\quad\mathrm{d}(\mathfrak{O}_{\alpha},\mathfrak{O}_{\beta})>N^2, \quad\forall\alpha\neq\beta,
\end{equation*}
then we have the following result:
\\
\textbf{The~final~reduction:} Define $\mathcal{X}^{\mathfrak{B}}_{\mathfrak{A}}\in\mathcal{M}^{\mathfrak{B}}_{\mathfrak{A}}$ by
\begin{equation}\label{E2.31}
\mathcal{X}^{\mathfrak{n}}_{\mathfrak{n}'}:=
\begin{cases}
\hat{\mathcal{P}}^{\mathfrak{n}}_{\mathfrak{n}'}\quad&\text{if}\quad(\mathfrak{n},\mathfrak{n}')
\in\bigcup_{\alpha}(\mathfrak{O}_{\alpha}\times\hat{\mathfrak{O}}_{\alpha}),\\
0&\text{if}\quad(\mathfrak{n},\mathfrak{n}')\notin\bigcup_{\alpha}(\mathfrak{O}_{\alpha}\times\hat{\mathfrak{O}}_{\alpha}),
\end{cases}
\end{equation}
where $\hat{\mathfrak{O}}_{\alpha}:=\{\mathfrak{n}\in{\mathfrak{A}}: \mathrm{d}(\mathfrak{n},\mathfrak{O}_{\alpha})\leq\frac{N^2}{4}\}$. The definition of $\hat{\mathfrak{O}}_{\alpha}$ together with \eqref{E2.30} may indicate $\hat{\mathfrak{O}}_{\alpha}\cap\hat{\mathfrak{O}}_{\beta}=\emptyset,~\forall\alpha\neq\beta$.

\textbf{Step1}: Let us claim that $\mathcal{X}^{\mathfrak{B}}_{\mathfrak{A}}\in\mathcal{M}^{\mathfrak{B}}_{\mathfrak{A}}$ has a left inverse $\mathcal{Y}^{\mathfrak{A}}_{\mathfrak{B}}\in\mathcal{M}^{\mathfrak{A}}_{\mathfrak{B}}$ with
\begin{align}\label{E2.32}
\|\mathcal{Y}^{\mathfrak{A}}_{\mathfrak{B}}\|_{0}\leq2(N')^{\tau}.
\end{align}
Define $\mathcal{Z}^{\mathfrak{B}}_{\mathfrak{A}}:=\hat{\mathcal{P}}^{\mathfrak{B}}_{\mathfrak{A}}-\mathcal{X}^{\mathfrak{B}}_{\mathfrak{A}}$, which then gives that $\mathcal{X}^{\mathfrak{B}}_{\mathfrak{A}}=\hat{\mathcal{P}}^{\mathfrak{B}}_{\mathfrak{A}}-\mathcal{Z}^{\mathfrak{B}}_{\mathfrak{A}}$. Definition \eqref{E2.32} implies that
\begin{align*}
\mathcal{Z}^{\mathfrak{n}}_{\mathfrak{n}'}=0\quad\text{if}\quad \mathrm{d}(\mathfrak{n},\mathfrak{n}')\leq{{N}^2}/{4}.
\end{align*}
Combining this with \eqref{E1.61}, \eqref{E2.19} and $\mathrm{(A1)}$, for all $s_1\geq s_0+\nu+r+\varrho$, we obtain
\begin{align}
|\mathcal{Z}^{\mathfrak{B}}_{\mathfrak{A}}|_{s_0}&\leq({N}^2/4)^{-(s_1-\nu-r-\varrho-s_0)}|\mathcal{Z}^{\mathfrak{B}}_{\mathfrak{A}}|_{s_1-\nu-r-\varrho}\leq 4^{s_1}{N}^{-2(s_1-\nu-r-\varrho-s_0)}|\hat{\mathcal{P}}^{\mathfrak{B}}_{\mathfrak{A}}|_{s_1-\nu-r-\varrho}\nonumber\\
&\leq C(s_1,\tilde{\Theta})4^{s_1}{N}^{-2(s_1-\nu-r-\varrho-s_0)}N^{\mathfrak{e}}(N^{s_1-\nu-r-s_0}+N^{-(\nu+r)}|\mathcal{Q}|_{s_1-\rho})\nonumber\\
&\leq C'(s_1,\tilde{\Theta})N^{2\mathfrak{e}-(s_1-\rho)},\label{E2.33}\\
|\mathcal{Z}^{\mathfrak{B}}_{\mathfrak{A}}|_{s}&\leq|\hat{\mathcal{P}}^{\mathfrak{B}}_{\mathfrak{A}}|_{s}\leq C(s,\tilde{\Theta})N^{\mathfrak{e}}(N^{s-s_0}+N^{-(\nu+r)}|\mathcal{Q}|_{s+\nu+r}).\label{E2.34}
\end{align}
%
In addition, for $N\geq{N}(\tilde{\Theta},s_1)$ large enough and $s_1>2\mathfrak{e}+\chi\tau+\varrho$, formulae  \eqref{E1.66} and \eqref{E2.33} give
\begin{align*}
\|((\mathcal{A}^{\mathfrak{A}}_{\mathfrak{A}})^{-1})^{\mathfrak{A}}_{\mathfrak{B}}\|_{0}\|\mathcal{Z}^{\mathfrak{B}}_{\mathfrak{A}}\|_{0}
\leq&\|((\mathcal{A}^{\mathfrak{A}}_{\mathfrak{A}})^{-1})^{\mathfrak{A}}_{\mathfrak{B}}\|_{0}|\mathcal{Z}^{\mathfrak{B}}_{\mathfrak{A}}|_{s_0}
\leq C(s_1,\tilde{\Theta})N^{2\mathfrak{e}-(s_1-\rho)}(N')^{\tau}\\
\stackrel{\eqref{E2.1}}{=}&C'(s_1,\tilde{\Theta})N^{2\mathfrak{e}-(s_1-\rho)+\chi\tau}\leq1/2.
\end{align*}
%
Lemma \ref{lemma2} verifies that $\mathcal{X}^{\mathfrak{B}}_{\mathfrak{A}}$ has a left inverse $\mathcal{Y}^{\mathfrak{A}}_{\mathfrak{B}}\in\mathcal{M}^{\mathfrak{A}}_{\mathfrak{B}}$ with
\begin{align*}
\|\mathcal{Y}^{\mathfrak{A}}_{\mathfrak{B}}\|_{0}\stackrel{\eqref{E1.34}}{\leq} 2\|((\mathcal{A}^{\mathfrak{A}}_{\mathfrak{A}})^{-1})^{\mathfrak{A}}_{\mathfrak{B}}\|_{0}\leq 2(N')^{\tau}.
\end{align*}

\textbf{Step2}: Define $\tilde{\mathcal{Y}}^{\mathfrak{A}}_{\mathfrak{B}}$ by
\begin{equation}\label{E2.35}
\tilde{\mathcal{Y}}^{\mathfrak{n}'}_{\mathfrak{n}}:=
\begin{cases}
\mathcal{Y}^{\mathfrak{n}'}_{\mathfrak{n}}\quad&\text{if}\quad(\mathfrak{n},\mathfrak{n}')\in\bigcup_{\alpha}
(\mathfrak{O}_{\alpha}\times\hat{\mathfrak{O}}_{\alpha})\\
0&\text{if}\quad(\mathfrak{n},\mathfrak{n}')\notin\bigcup_{\alpha}(\mathfrak{O}_{\alpha}\times\hat{\mathfrak{O}}_{\alpha}).
\end{cases}
\end{equation}
If the fact
\begin{align}\label{E2.37}
({\mathcal{Y}}^{\mathfrak{A}}_{\mathfrak{B}}-\tilde{\mathcal{Y}}^{\mathfrak{A}}_{\mathfrak{B}})\mathcal{X}^{\mathfrak{B}}_{\mathfrak{A}}=0
\end{align}
holds, then $\tilde{\mathcal{Y}}^{\mathfrak{A}}_{\mathfrak{B}}$ is a left inverse of $\mathcal{X}^{\mathfrak{B}}_{\mathfrak{A}}$ with
\begin{align}\label{E2.38}
|\tilde{\mathcal{Y}}^{\mathfrak{A}}_{\mathfrak{B}}|_{s}\leq C(s)N^{C_1(s+\nu+r)+\chi\tau}.
\end{align}
Let us prove formula \eqref{E2.37}. For any $\mathfrak{n}\in {\mathfrak{B}}=\bigcup_{\alpha}\mathfrak{O}_{\alpha}$, there is $\alpha$ such that $\mathfrak{n}\in \mathfrak{O}_{\alpha}$ and
\begin{align*}
(({\mathcal{Y}}^{\mathfrak{A}}_{\mathfrak{B}}-\tilde{\mathcal{Y}}^{\mathfrak{A}}_{\mathfrak{B}})
\mathcal{X}^{\mathfrak{B}}_{\mathfrak{A}})^{\mathfrak{n}'}_{\mathfrak{n}}=
\sum\limits_{\mathfrak{n}''\notin\hat{\mathfrak{O}}_{\alpha}}({\mathcal{Y}}^{\mathfrak{A}}_{\mathfrak{B}}
-\tilde{\mathcal{Y}}^{\mathfrak{A}}_{\mathfrak{B}})^{\mathfrak{n}''}_{\mathfrak{n}}\mathcal{X}^{\mathfrak{n}'}_{\mathfrak{n}''}.
\end{align*}
Remark that $({\mathcal{Y}}^{\mathfrak{A}}_{\mathfrak{B}}-\tilde{\mathcal{Y}}^{\mathfrak{A}}_{\mathfrak{B}})^{\mathfrak{n}''}_{\mathfrak{n}}
=\mathcal{Y}^{\mathfrak{n}''}_{\mathfrak{n}}-\tilde{\mathcal{Y}}^{\mathfrak{n}''}_{\mathfrak{n}}=0$ if $\mathfrak{n}''\in\hat{\mathfrak{O}}_{\alpha}$.

If $\mathfrak{n}'\in \mathfrak{O}_{\alpha}$, then definition \eqref{E2.31} implies that $\mathcal{X}^{\mathfrak{n}'}_{\mathfrak{n}''}=0$, which shows  $(({\mathcal{Y}}^{\mathfrak{A}}_{\mathfrak{B}}-\tilde{\mathcal{Y}}^{\mathfrak{A}}_{\mathfrak{B}})
\mathcal{X}^{\mathfrak{B}}_{\mathfrak{A}})^{\mathfrak{n}'}_{\mathfrak{n}}=0$.

If $\mathfrak{n}'\in\mathfrak{O}_{\beta}$ with $\alpha\neq\beta$, for $\mathfrak{n}''\notin\hat{\mathfrak{O}}_{\beta}$, then $\mathcal{X}^{\mathfrak{n}'}_{\mathfrak{n}''}=0$ owing to \eqref{E2.31}. As a consequence
\begin{align*}
(({\mathcal{Y}}^{\mathfrak{A}}_{\mathfrak{B}}-\tilde{\mathcal{Y}}^{\mathfrak{A}}_{\mathfrak{B}})
\mathcal{X}^{\mathfrak{B}}_{\mathfrak{A}})^{\mathfrak{n}'}_{\mathfrak{n}}=&\sum\limits_{\mathfrak{n}''\in\hat{\mathfrak{O}}_{\beta}}
({\mathcal{Y}}^{\mathfrak{A}}_{\mathfrak{B}}-\tilde{\mathcal{Y}}^{\mathfrak{A}}_{\mathfrak{B}})^{\mathfrak{n}''}_{\mathfrak{n}}
\mathcal{X}^{\mathfrak{n}'}_{\mathfrak{n}''}{=}\sum\limits_{\mathfrak{n}''\in\hat{\mathfrak{O}}_{\beta}}
({\mathcal{Y}}^{\mathfrak{n}''}_{\mathfrak{n}}-{\tilde{\mathcal{Y}}}^{\mathfrak{n}''}_{\mathfrak{n}})\mathcal{X}^{\mathfrak{n}'}_{\mathfrak{n}''}\\
\stackrel{\eqref{E2.35}}{=}&\sum\limits_{\mathfrak{n}''\in\hat{\mathfrak{O}}_{\beta}}{\mathcal{Y}}^{\mathfrak{n}''}_{\mathfrak{n}}
\mathcal{X}^{\mathfrak{n}'}_{\mathfrak{n}''}\stackrel{\eqref{E2.31}}{=}\sum\limits_{\mathfrak{n}''\in \mathfrak{A}}{\mathcal{Y}}^{\mathfrak{n}''}_{\mathfrak{n}}\mathcal{X}^{\mathfrak{n}'}_{\mathfrak{n}''}\\
=&(\mathcal{Y}^{\mathfrak{A}}_{\mathfrak{B}}\mathcal{X}^{\mathfrak{B}}_{\mathfrak{A}})^{\mathfrak{n}'}_{\mathfrak{n}}
=(\mathrm{I}^{\mathfrak{B}}_{\mathfrak{B}})^{\mathfrak{n}'}_{\mathfrak{n}}=0.
\end{align*}
It is obvious that $\mathrm{diam}(\hat{\mathfrak{O}}_{\alpha})<2N^{C_1}$ due to \eqref{E2.30}. Based on this and definition \eqref{E2.35}, we obtain that $\tilde{\mathcal{Y}}^{\mathfrak{n}'}_{\mathfrak{n}}=0$ for all $|\mathfrak{n}-\mathfrak{n}'|\geq2N^{C_1}$. Thus it follows from \eqref{E1.62}, \eqref{E2.1}, {\eqref{E2.32} and \eqref{E2.35}} that
\begin{align*}
|\tilde{\mathcal{Y}}^{\mathfrak{A}}_{\mathfrak{B}}|_{s}\leq C(s)N^{C_1(s+\nu+r)}\|\tilde{\mathcal{Y}}^{\mathfrak{A}}_{\mathfrak{B}}\|_{0}{\leq}C(s)N^{C_1(s+\nu+r)+\chi\tau}.
\end{align*}

\textbf{Step3}: For $N\geq \bar{N}(\Upsilon,\tilde{\Theta},s_2)$ large enough and $s_1>C_1(s_0+\nu+r)+\chi\tau+2\mathfrak{e}+\varrho$, it follows from \eqref{E2.33} and \eqref{E2.38} that
\begin{align*}
|\tilde{\mathcal{Y}}^{\mathfrak{A}}_{\mathfrak{B}}|_{s_0}|\mathcal{Z}^{\mathfrak{B}}_{\mathfrak{A}}|_{s_0}\leq CN^{C_1(s+\nu+r)+\chi\tau}C(s_1,\tilde{\Theta})N^{2\mathfrak{e}-(s_1-\rho)}\leq1/2.
\end{align*}
Combining this with the equality $\hat{\mathcal{P}}^{\mathfrak{B}}_{\mathfrak{A}}=\mathcal{X}^{\mathfrak{B}}_{\mathfrak{A}}+\mathcal{Z}^{\mathfrak{B}}_{\mathfrak{A}}$, Lemma \ref{lemma2}, \eqref{E2.34} and \eqref{E2.38} establishes that $\hat{\mathcal{P}}^{\mathfrak{B}}_{\mathfrak{A}}$ has a left inverse ${^{[-1]}}(\hat{\mathcal{P}}^{\mathfrak{B}}_{\mathfrak{A}})$ with
\begin{align}
|^{[-1]}(\hat{\mathcal{P}}^{\mathfrak{B}}_{\mathfrak{A}})|_{s_0}\stackrel{\eqref{E1.33}}{\leq}&2
|\tilde{\mathcal{Y}}^{\mathfrak{A}}_{\mathfrak{B}}|_{s_0}
{\leq}CN^{C_1(s_{0}+\nu+r)+\chi\tau},\label{E2.39}\\
|^{[-1]}(\hat{\mathcal{P}}^{\mathfrak{B}}_{\mathfrak{A}})|_{s}\stackrel{\eqref{E1.92}}{\leq}&C(s)(|\tilde{\mathcal{Y}}^{\mathfrak{A}}_{\mathfrak{B}}|_{s}
+|\tilde{Y}^{\mathfrak{A}}_{\mathfrak{B}}|^2_{s_0}|Z^{\mathfrak{B}}_{\mathfrak{A}}|_{s})
{\leq} C'(s,\tilde{\Theta})N^{2\chi\tau+\mathfrak{e}+2C_1(s_0+\nu+r)}(N^{C_1s}+|\mathcal{Q}|_{s+\nu+r}).\label{E2.40}
\end{align}
Thus system \eqref{E2.5} is equivalent to
\begin{equation*}
\begin{cases}
u_{\mathcal{G}}=\tilde{\mathcal{P}}^{\mathfrak{B}}_{\mathfrak{G}}({^{[-1]}}
(\hat{\mathcal{P}}^{\mathfrak{B}}_{\mathfrak{A}}))\hat{\mathcal{S}}^{\mathfrak{A}}_{\mathfrak{A}}h_{\mathfrak{A}}+
\tilde{\mathcal{S}}^{\mathfrak{A}}_{\mathfrak{G}} h_{\mathfrak{A}},\\
u_{\mathfrak{B}}=({^{[-1]}}(\hat{\mathcal{P}}^{\mathfrak{B}}_{\mathfrak{A}}))
\hat{\mathcal{S}}^{\mathfrak{A}}_{\mathfrak{A}}h_{\mathfrak{A}}.
\end{cases}
\end{equation*}
This implies that
\begin{align*}
((\mathcal{A}^{\mathfrak{A}}_{\mathfrak{A}})^{-1})^{\mathfrak{A}}_{\mathfrak{G}}=
\tilde{\mathcal{P}}^{\mathfrak{B}}_{\mathfrak{G}}({^{[-1]}}\hat{\mathcal{P}}^{\mathfrak{B}}_{\mathfrak{A}})
\hat{\mathcal{S}}^{\mathfrak{A}}_{\mathfrak{A}}+\tilde{\mathcal{S}}^{\mathfrak{A}}_{\mathfrak{G}},\quad
((\mathcal{A}^{\mathfrak{A}}_{\mathfrak{A}})^{-1})^{\mathfrak{A}}_{\mathfrak{B}}=({^{[-1]}}
\hat{\mathcal{P}}^{\mathfrak{B}}_{\mathfrak{A}})\hat{\mathcal{S}}^{\mathfrak{A}}_{\mathfrak{A}}.
\end{align*}
The fact $\mathfrak{A}\in\mathfrak{N}$ with $\mathrm{diam}(\mathfrak{A})\leq4N'$ shows
\begin{align*}
\mathcal{Q}^{\mathfrak{n}}_{\mathfrak{n}'}=0\quad\text{if}\quad \mathrm{d}(\mathfrak{n},\mathfrak{n}')>{8{N'}},
\end{align*}
which leads to $|\mathcal{Q}|_{s+\nu+r}\leq C(N')^{\nu+r}|\mathcal{Q}|_{s}=CN^{\chi(\nu+r)}|\mathcal{Q}|_{s}$  due to \eqref{E1.62} and \eqref{E2.1}. Then it follows from \eqref{E1.23}, \eqref{E2.18}-\eqref{E2.19} and \eqref{E2.39}-\eqref{E2.40}  that
\begin{align}
|((\mathcal{A}^{\mathfrak{A}}_{\mathfrak{A}})^{-1})^\mathfrak{A}_{\mathfrak{B}}|_{s}\leq& C''(s,\tilde{\Theta})N^{2\mathfrak{e}+2\chi\tau+2C_1(s_0+\nu+r)}(N^{C_1s}+|\mathcal{Q}|_{s+\nu+r})\nonumber\\
\leq&C'''(s,\tilde{\Theta})N^{2\mathfrak{e}+2\chi\tau+2C_1(s_0+\nu+r)+\chi(\nu+r)}(N^{C_1s}+|\mathcal{Q}|_{s}).\label{E2.41}
\end{align}
In addition
\begin{align*}
|((\mathcal{A}^{\mathfrak{A}}_{\mathfrak{A}})^{-1})^\mathfrak{A}_{\mathfrak{B}}|_{s_0}
\stackrel{\eqref{E1.24}}{\leq}C|{^{[-1]}}\hat{\mathcal{P}}^{\mathfrak{B}}_{\mathfrak{A}}|_{s_0}
|\hat{\mathcal{S}}^{\mathfrak{A}}_{\mathfrak{A}}|_{s_0}
\stackrel{\eqref{E2.18},\eqref{E2.39}}{\leq}C'(s_1)N^{\mathfrak{e}+\chi\tau+C_1(s_0+\nu+r)}.
\end{align*}
Combining this with \eqref{E1.23}, \eqref{E2.14}-\eqref{E2.15}, \eqref{E2.41} gives that
\begin{align*}
|((\mathcal{A}^{\mathfrak{A}}_{\mathfrak{A}})^{-1})^\mathfrak{A}_{\mathfrak{G}}|_{s}&\leq C_{4}(s,\tilde{\Theta})N^{2\mathfrak{e}+2\chi\tau+2C_1(s_0+\nu+r)+\chi(\nu+r)}(N^{C_1 s}+|\mathcal{Q}|_{s}).
\end{align*}
Consequently,  for $N\geq \bar{{N}}(\Upsilon,\tilde{\Theta},s_1)$ large enough, we have
\begin{align*}
|(\mathcal{A}^\mathfrak{A}_{\mathfrak{A}})^{-1}|_{s}\leq&|((\mathcal{A}^{\mathfrak{A}}_{\mathfrak{A}})^{-1})^\mathfrak{A}_{\mathfrak{G}}|_{s}
+|((\mathcal{A}^{\mathfrak{A}}_{\mathfrak{A}})^{-1})^\mathfrak{A}_{\mathfrak{B}}|_{s}\\
\leq&C_{5}(s,\tilde{\Theta})N^{2\mathfrak{e}+2\chi\tau+2C_1(s_0+\nu+r)+\chi(\nu+r)}(N^{C_1s}+|\mathcal{Q}|_{s})\\
\leq&\frac{1}{4}(N')^{\tau_2}((N')^{\delta s}+|\mathcal{Q}|_{s}),
\end{align*}
if $\chi^{-1}C_1<\delta$, $\chi^{-1}(2\mathfrak{e}+2\chi\tau+2C_1(s_0+\nu+r)+\chi(\nu+r))<\tau_2$.
%
\end{proof}




\end{document}